\documentclass[12pt]{amsart}
\usepackage{amstext}
\usepackage{amsthm}
\usepackage{amsmath}	
\usepackage{amssymb}
\usepackage{latexsym}
\usepackage{amsfonts}
\usepackage{graphicx}
\usepackage{texdraw}
\usepackage{graphpap}

\usepackage{mathtools}
\mathtoolsset{showonlyrefs,showmanualtags}

\usepackage{hyperref} 
\hypersetup{
    colorlinks=true,       
    linkcolor=blue,          
    citecolor=magenta,        
    filecolor=magenta,      
    urlcolor=cyan,            
    hypertexnames=false, 
}  

\usepackage{amsrefs} 
\usepackage{mathrsfs} 
\usepackage{tikz}

\input txdtools

\begin{document}

\newcommand{\ci}[1]{_{ {}_{\scriptstyle #1}}}

\newcommand{\norm}[1]{\ensuremath{\left\|#1\right\|}}
\newcommand{\abs}[1]{\ensuremath{\left\vert#1\right\vert}}
\newcommand{\p}{\ensuremath{\partial}}
\newcommand{\pr}{\mathcal{P}}

\newcommand{\pbar}{\ensuremath{\bar{\partial}}}
\newcommand{\db}{\overline\partial}
\newcommand{\D}{\mathbb{D}}
\newcommand{\B}{\mathbb{B}}
\newcommand{\Sp}{\mathbb{S}}
\newcommand{\T}{\mathbb{T}}
\newcommand{\R}{\mathbb{R}}
\newcommand{\Z}{\mathbb{Z}}
\newcommand{\C}{\mathbb{C}}
\newcommand{\N}{\mathbb{N}}
\newcommand{\scrH}{\mathcal{H}}
\newcommand{\scrL}{\mathcal{L}}
\newcommand{\td}{\widetilde\Delta}

\newcommand{\La}{\langle }
\newcommand{\Ra}{\rangle }
\newcommand{\rk}{\operatorname{rk}}
\newcommand{\card}{\operatorname{card}}
\newcommand{\ran}{\operatorname{Ran}}
\newcommand{\osc}{\operatorname{OSC}}
\newcommand{\im}{\operatorname{Im}}
\newcommand{\re}{\operatorname{Re}}
\newcommand{\tr}{\operatorname{tr}}
\newcommand{\vf}{\varphi}
\newcommand{\f}[2]{\ensuremath{\frac{#1}{#2}}}


\newcommand{\entrylabel}[1]{\mbox{#1}\hfill}

\newenvironment{entry}
{\begin{list}{X}%
  {\renewcommand{\makelabel}{\entrylabel}%
      \setlength{\labelwidth}{55pt}%
      \setlength{\leftmargin}{\labelwidth}
      \addtolength{\leftmargin}{\labelsep}%
   }%
}%
{\end{list}}


\numberwithin{equation}{section}

\newtheorem{thm}{Theorem}[section]
\newtheorem{lm}[equation]{Lemma}
\newtheorem{cor}[equation]{Corollary}
\newtheorem{conj}[equation]{Conjecture}
\newtheorem{prob}[equation]{Problem}
\newtheorem{prop}[equation]{Proposition}
\newtheorem*{prop*}{Proposition}

\newtheorem{lemma}[equation]{Lemma} 
\newtheorem{proposition}[equation]{Proposition} 
\newtheorem{theorem}[equation]{Theorem} 
\newtheorem{corollary}[equation]{Corollary} 
\newtheorem{conjecture}[equation]{Conjecture}
\newtheorem{priorResults}{Theorem} 
\renewcommand*{\thepriorResults}{\Alph{priorResults}} 

\theoremstyle{definition}
\newtheorem{definition}[equation]{Definition} 

\theoremstyle{remark}
\newtheorem{rem}[thm]{Remark}
\newtheorem*{rem*}{Remark}

\theoremstyle{remark}
\newtheorem{remark}[equation]{Remark}

\newtheorem*{ack}{Acknowledgment}

\title[Two Weight Cauchy Transform Inequalities]{Two Weight Inequalities for the \\ Cauchy Transform from $ \mathbb{R}$ to $ \mathbb{C} _+$} 

\author[M.~T.~Lacey]{Michael T. Lacey$^{1}$}   

\address{Michael T. Lacey, School of Mathematics\\ Georgia Institute of Technology\\ 686 Cherry Street\\ Atlanta GA  USA 30332-0160}
\email {lacey@math.gatech.edu}
\thanks{1.  Research supported in part by National Science Foundation DMS grants  \# 0968499 and \# 1265570, a grant from the Simons Foundation (\#229596 to Michael Lacey), and the Australian Research Council through grant ARC-DP120100399.}



\author[E.~T.~Sawyer]{Eric T. Sawyer$^{2}$}
\address{Eric T. Sawyer, Department of Mathematics \& Statistics\\ McMaster University\\ 1280 Main Street West\\ Hamilton, Ontario, Canada L8S 4K1 }
\email{sawyer@mcmaster.ca}
\thanks{2. Research supported in part by NSERC}

\author[C.-Y.~Shen]{Chun-Yen Shen$^{3}$}
\address{Chun-Yen Shen, Department of Mathematics\\ National Central University \\ Chungli, Taiwan 32054}
\email{chunyshen@gmail.com}
\thanks{3.  Research supported in part by the NSC, through grant NSC102-2115-M-008-015-MY2.}

\author[I.~Uriarte-Tuero]{Ignacio Uriarte-Tuero$^{4}$}
\address{Ignacio Uriarte-Tuero, Department of Mathematics \\ Michigan State University \\ East Lansing MI USA 48824}
\email{ignacio@math.msu.edu}
\thanks{4.  Research supported in part by a National Science Foundation DMS grants \# 0901524 and \# 1056965, MTM2010-16232, MTM2015-65792-P (MINECO, Spain), and a Sloan Foundation Fellowship.}

\author[B.~D.~Wick]{Brett D. Wick$^{5}$}
\address{Brett D. Wick, Department of Mathematics\\ Washington University -- St. Louis\\ One Brookings Drive\\ St. Louis, MO USA 63130-4899}
\email{wick@math.wustl.edu}
\thanks{5.  Research supported in part by a National Science Foundation DMS grants \# 0955432 and \#1500509.}

\begin{abstract} 
We characterize those pairs of weights $ \sigma $ on $ \mathbb{R}$ and $ \tau $ on $ \mathbb{C}_+$ for which the Cauchy transform 
$\mathsf{C}_{\sigma } f (z) \equiv \int _{\mathbb{R}} \frac {f(t)} {t-z} \; \sigma (dt)$, $ z\in \mathbb{C} _+$,  is bounded from $L ^2(\mathbb{R};\sigma)$ to $L ^{2}(\mathbb{C}_+; \tau)$.  The characterization is in terms of an $A_2$ condition on the pair of weights and testing conditions for the transform, 
extending  the recent solution of the two weight inequality for the Hilbert transform.  
As corollaries of this result we derive (1)  a characterization of embedding measures for the model space $K_\vartheta$, for arbitrary  inner function $ \vartheta $, 
and (2) a characterization of the (essential) norm of composition operators mapping $K_\vartheta$ into  a general class of Hardy and Bergman spaces.  
\end{abstract}

\maketitle
\setcounter{tocdepth}{1} 
\tableofcontents

\section{Introduction} 

In this paper we characterize the boundedness for the Cauchy transform:
$$
\mathsf C _{\sigma} f(z)\equiv\int_{\mathbb{R}} \frac{f(t)}{t-z}\sigma(dt)
$$
as a map between $L^2(\mathbb{R};\sigma)$ and $L^2(\mathbb{R}^2_+;\tau)$, where $ \sigma $ and $ \tau $ are two 
arbitrary weights, i.e.\thinspace locally finite positive Borel measures.  The characterization is in terms of 
a joint Poisson $ A_2$  condition and a set of testing conditions.  

We are motivated by the study of the model space $ K _{\vartheta } = H ^2 \ominus \vartheta H ^2 $, where $ \vartheta $ is an inner function. 
These spaces are essential to the Nagy-Foias model for contractions on a Hilbert space.  
Function theoretic properties of the space $K_\vartheta$  are therefore of significant interest, and for the basics we point the reader to the text \cite{MR827223} and survey \cite{MR2198367} and the references therein for a guide to this intricate literature.  

In Theorem~\ref{K_varthetaCM}, we characterize the Carleson measures for $ K _{\vartheta }$ spaces, a question posed by 
Cohn \cite{MR705235}.  Previously, Aleksandrov \cite{MR1464420} characterized \emph{isometric} Carleson measures, and otherwise 
definitive results have only been proved in the so-called  `one component' case \cites{MR705235,MR849293}.  

We also characterize the norm of a composition operator from a $ K _{\vartheta } $ space to 
any one of a general class of analytic function spaces, which include Hardy and the entire scale of Bergman spaces.  
See Theorem~\ref{t:composition}. 
The operator-theoretic properties of composition operators have been of intense interest for 60 years, see for instance the 
text \cite{MR1397026}, but the only result concerning composition operators on an arbitrary model space is the 
elegant characterization of  compactness  from $ K _{\vartheta }$ to $ H ^2 $   obtained by Lyubarskii-Malinnikova \cite{12055172}.  

Our methods to attack the characterization question use real variable techniques, and so it is more convenient to describe the main results in terms of the Riesz transform.  
Let $\mathsf{R}$ denote the one-dimensional Riesz  transform acting on the plane. This is an operator defined on functions in the following manner.  For $x\in\mathbb{R}_+ ^2 $  and a signed measure $\nu$ on $\R$ we are interested in the family of   $ \mathbb R ^2 $-valued operators 
given by 
\begin{equation}
\label{Cauchy_Def}
\mathsf{R} \nu(x)\equiv\int \frac {x -t} { \lvert  x-t\rvert ^2   }{\nu(dt)} , \qquad x\in \mathbb R ^2 _+.  
\end{equation}
We write the coordinates of this operator as $(\mathsf{R}  ^{1}, \mathsf{R}  ^{2})$.  
The second coordinate $ \mathsf R ^{2}$ is the Poisson transform $ \mathsf P $, and the Cauchy transform is 
\begin{equation*}
\mathsf C \nu  \equiv \mathsf R ^{1}  \nu + i \mathsf R ^{2}   \nu . 
\end{equation*}

Let $\sigma$ denote a weight on $\R$ and $\tau$ denote a measure on the upper half plane $\mathbb R ^2 _+$. 
Finding necessary and sufficient conditions on the pair of measures $\sigma$ and $\tau$ so that
the estimate below holds is the \emph{two weight problem} for the Riesz transform
\begin{equation}
\label{Cauchy_Est1}
 \left\Vert \mathsf{R} \left(\sigma f\right)\right\Vert_{L^2(\mathbb R_+ ^2;\tau)}\le \mathscr N \left\Vert f\right\Vert_{L^2(\mathbb{R};\sigma)}. 
\end{equation}
This Theorem characterizes the two weight inequality, under restrictions on the supports of the two weights.  

\begin{theorem}\label{t:main}  Let $ \sigma $ be a weight on  $ \mathbb R $ and $ \tau $ a weight on the closed upper half-plane  $ \mathbb R ^2  _+$.  
The two weight inequality \eqref{Cauchy_Est1} holds if and only if  these conditions hold uniformly over all intervals $ I \subset \mathbb R $,  
and Carleson cubes $ Q_I=I \times [0, \lvert  I\rvert) $.
An  $ A_2 $ condition holds: For a finite positive constant $ \mathscr A_2$, 
\begin{equation}  \label{e:A2}
\begin{split}	
\frac{ \tau (Q_I) } {\lvert  I\rvert } 
\times 
\int  _{\mathbb R \setminus I} \frac{\lvert  I\rvert } { ( \lvert  I\rvert  + \textup{dist} (t,I))^2} \; \sigma (dt) &\leq \mathscr A_2  ,
\\ \frac {\sigma (I)} {\lvert  I\rvert } \times 
 \int _{\mathbb R^2_+ \setminus Q_I} \frac { \lvert  I\rvert  }  { (\lvert  I\rvert + \textup{dist} ( x, Q_I )) ^2 } \; \tau (dx) 
 & \leq \mathscr A_2,
\end{split}
\end{equation}
and,  these  testing inequalities hold: For a finite positive constant $ \mathscr T$, 
\begin{align} \label{e:T1}
 \int _{Q_I} \lvert  \mathsf R _{\sigma } \mathbf 1_{I}(x)\rvert ^2 \; \tau(dx)  &\le \mathscr T ^2  \sigma (I) , 
\\ \label{e:T*1} 
\int _{I} \lvert \mathsf R  _{\tau  } ^{\ast}  \mathbf 1_{Q_I} (t)\rvert ^2 \; \sigma(dt)	& \le \mathscr T ^2   \tau (Q_I) .
\end{align}
Moreover, if $ \mathscr T$ and $ \mathscr A_2$ are the best constants in these inequalities, then $\mathscr N\simeq \mathscr R 
\equiv  \mathscr A_2 ^{1/2} + \mathscr T$. 
\end{theorem}

Restricting $ \tau $ to be supported on $ \mathbb R \times \{0\}$ would reduce to the Hilbert transform case. Then the Theorem above 
is the foundational result of  
 Lacey-Sawyer-Shen-Uriarte-Tuero \cite{MR3285857} and Lacey \cite{MR3285858}, with the further refinements of Hyt\"onen \cite{13120843},  answering a conjecture of Nazarov-Treil-Volberg \cites{10031596,V}. 
The necessity of the conditions above, given the norm condition are obvious for the testing inequalities \eqref{e:T1}---\eqref{e:T*1}, 
and the necessity of the $ A_2$ condition is known.  Thus the main content is the sufficiency of the conditions above 
for the norm inequality.   

A version of the above theorem holds in a more general context. In any dimension, as long as one weight is supported on the real line, the Cauchy transform 
can be replaced by an arbitrary fractional Riesz transform. See \cite{SSUT}.

\subsection{Other Results}

We now present   applications of the main result.  Because of the close connection with analytic functions and the Cauchy transform our applications are drawn from this area.

\subsubsection{Setting on the Disk, Compactness}  
There are two forms of a Cauchy transform on the disk. For $ \sigma $ a weight on $ \mathbb T $, and $ f\in L ^2 (\mathbb T , \sigma )$, the transform could be either 
\begin{equation*}
\int _{\mathbb T } \frac{ f (w)} {w-z} \sigma (dw), \quad \textup{or} \quad 
\int _{\mathbb T }\frac{ f (w)} {1-\overline w z} \sigma (dw). 
\end{equation*}
The two are unitarily equivalent, via division by $\overline{w}$ and conjugation. We prefer the second formulation, 
and denote it by $ \mathsf C _{\sigma } f (z)$.  
The main theorem, formulated on the disk, is below. 
 
\begin{theorem}\label{t:disk}  Let $ \sigma $ be a weight on $ \mathbb T = \partial \mathbb D $, and $ \tau $ a weight on $ \overline {\mathbb D }$.  The inequality below holds, for some finite positive $ \mathscr C$,  
\begin{equation}\label{e:cauchyDisk}
\begin{split}
 \lVert  \mathsf C_ \sigma  f \rVert_ {L ^2 (\overline {\mathbb D }; \tau )}\le \mathscr C \lVert f\rVert_{L ^2 (\mathbb T ; \sigma )} ,
\end{split}
\end{equation}
if and only if these constants are finite: For the Poisson extension operator $ \mathsf P$ on the disk, 
\begin{gather}\label{e:Disk}
 \sigma (\mathbb T) \cdot \tau (\overline {\mathbb D }) + \sup _{z\in  {\mathbb D }} \bigl\{\mathsf P (\sigma \mathbf 1_{\mathbb T \setminus I_z}) (z) \mathsf P \tau (z) 
+ \mathsf P \sigma (z) \mathsf P (\tau \mathbf 1_{\overline {\mathbb D} \setminus B_{I_z}}) (z) \bigr\}   
\equiv \mathscr A_2 , 
\\
\sup _{I} 
\sigma (I) ^{-1} \int _{B_I}  \lvert  \mathsf C _{\sigma } \mathbf 1_{I}(z)\rvert ^2 \tau (dA(z))  \equiv \mathscr T ^2 , 
\\\sup _{I} \tau (B_I) ^{-1} 
\int _{I}  \lvert   \mathsf C _{\tau} ^{\ast}  \mathbf 1_{B_I}(w)\rvert ^2 \sigma (dw)  \equiv
\mathscr T ^2, 
\end{gather}
where these conventions hold. In the first inequality, to each $z\in \mathbb{D}$ we associate the interval $I_z\subset \mathbb{T}$ with center $\frac{z}{\left\vert z\right\vert}$ and length $1-\left\vert z\right\vert$.  The last two inequalities are uniform over all intervals $ I\subset \mathbb T $, with $ \lvert  I\rvert\leq \tfrac 12  $, 
 $ B_I \equiv \{ z\in \overline  {\mathbb D } \::\:   z= r \operatorname e ^{i \theta },  \lvert  1-r \rvert \leq \lvert  I\rvert,  \operatorname e ^{i \theta } \in I \}$ is the Carleson box over $ I$. 
Finally, for the best constants $ \mathscr C, \mathscr A_2, \mathscr T$ in these inequalities, we have $ \mathscr C \simeq \mathscr A_2 ^{1/2} + \mathscr T$.  
\end{theorem}

We are not aware of how to derive this theorem from the results on the upper half-plane, rather the proof must be repeated, taking 
into account  a few minor complications. First, it is elementary to see that the kernel of the Cauchy transform 
satisfies 
\begin{equation} \label{e:CPQ}
\frac {2} {1- z \overline w} = 1+ P _{z} (w) + iQ _{z} (w) 
\end{equation}
where the right side consists of the Poisson and conjugate Poisson kernels, namely for $ z\in \mathbb D $ and 
$ w\in \mathbb T $, 
\begin{equation*}
P_{z} (w) = \frac {1- \lvert  z\rvert ^2  } {\lvert  w-z\rvert ^2  }, \qquad 
Q_{z} (w) = \frac {2 \textup{Im} (z \overline w)  } {\lvert  w-z\rvert ^2  }. 
\end{equation*}
The leading constant term on the right in \eqref{e:CPQ}  
leads to the global term  $ \sigma (\mathbb T ) \tau (\overline {\mathbb D })$ 
in the $ A_2 $ condition \eqref{e:Disk}.  The necessity of the other conditions is then seen 
in a manner similar to the setting of the upper half-plane.  In the sufficient condition, the main point is to 
prove boundedness of the transform 
\begin{equation*}
\int _{\mathbb T } \{P _{z} (w) + iQ _{z} (w)\} f (w) \; \sigma (dw).  
\end{equation*}
This is then in a form closely matching the vector Riesz transform.  

One reduces to the case when the weight $\sigma$ is supported on a fixed small arc of $\mathbb{T}$. This is accomplished by showing that a version of the weak-boundedness principle holds so that one can see that the testing inequalities still hold for the restricted weight $ \sigma $.  The weak boundedness statement of interest is: $\left\vert \left\langle \mathsf C_{\sigma} \mathbf 1_I, \mathbf 1_{B_{I'}}\right\rangle_{\tau}\right\vert\lesssim \mathscr A_2^{\frac{1}{2}}\sigma(I)^{\frac{1}{2}}\tau(B_{I'})^{\frac{1}{2}}$, when $I$ and $B_{I'}$ have comparable side lengths and their distance is close.   (Compare to Proposition \ref{p:weak} for the statement of on $\mathbb{R}$ and $\mathbb{R}^2_+$). 
One should also reduce to the case where the weight $\tau$ is supported in a narrow annulus close to the boundary of the arc that supports $ \sigma $.   We reduce to an annulus close to the boundary of the arc so as to avoid the origin (which would create complications in how to extend the grid).  The random grids on the circle $\mathbb{T}$ are constructed by a rotation of the standard lattice on $\mathbb{T}$. The corresponding grid in $\mathbb{D}$ is then constructed in the analogous manner, resulting in the standard Bergman tree. Modifications of energy and monotonicity are handled in an analogous fashion, with derivative calculations  being done in radial and angular coordinates relative to the arc.   The remainder of the proof then follows similarly to what is done in the case discussed in detail in the rest of the paper. In the interest of brevity, we leave the modifications to the interested reader.

We now turn to compactness of the Cauchy transform.  

\begin{theorem}\label{t:compact} Under the assumptions of Theorem~\ref{t:disk}, the operator $ \mathsf C _{\sigma }$ is compact if and only if 
$ \mathsf C _{\sigma } \::\: L ^2 (\mathbb T ; \sigma ) \to L ^2 (\overline  {\mathbb D } ; \tau )$ is bounded and these conditions hold:
\begin{gather*}
\lim _{r \uparrow 1}  \sup _{\lvert  z\rvert=r } \big\{ \mathsf P (\sigma \mathbf 1_{\mathbb T \setminus I_{z}}) (z) \mathsf P \tau (z) 
+ \mathsf P \sigma (z) \mathsf P (\tau \mathbf 1_{\overline  {\mathbb D }  \setminus B_{I_z}} ) (z)\big\}
=0 ,
\\
\lim _{ \epsilon \downarrow 0} 
\sup _{ \lvert  I\rvert < \epsilon  }  \sigma (I) ^{-1} 
\int _{B_I}  \lvert  \mathsf C _{\sigma } \mathbf 1_{I}(z)\rvert ^2 \tau (dA(z))  =0, 
\\
\lim _{ \epsilon \downarrow 0} 
\sup _{ \lvert  I\rvert < \epsilon  }    \tau (B_I) ^{-1} 
\int _{I}  \lvert  \mathsf C _{\tau} ^{\ast}  ( \mathbf 1_{B_I})(t)\rvert ^2 \sigma (dt)  
=0.
\end{gather*}
\end{theorem}

The details of the proof will not be given; they are a bit easier than those of \cite{primer}*{\S9}.  Indeed, in \cite{primer}*{\S9} to test compactness one must check intervals that are both small and those that are moving toward infinity.  Whereas, in our case, since we are working in the disc, only intervals of small size play a role in determining the compactness.  With this result in hand, one can state characterizations of compact analogs of the two theorems that follow.

\subsubsection{Carleson Measures for the Space $K_\vartheta$}
Let $H^2=H^2(\mathbb{D})$ denote the Hardy space of analytic functions on the unit disk $\D$.  
 Let $\vartheta$ be an inner function on $\mathbb{D}$, namely an analytic function such that $\left\vert\vartheta (\xi)\right\vert =1$ for almost every $\xi\in\mathbb{T}$.  Such functions have a canonical factorization given by 
\begin{equation}\label{e:factorization}
\vartheta (z) = B _{\Lambda } (z) \operatorname {exp}\left(-\int _{\mathbb T } \frac { \xi + z} {\xi -z} \; \nu (d \xi ) \right), 
\end{equation}
where $ B _{\Lambda } $ is a Blaschke product with zero set given by $ \Lambda \subset \mathbb D $ and $ \nu $ is a measure on $ \mathbb T $ singular with respect to Lebesgue measure.  
The space $K_\vartheta\equiv H^2\ominus\vartheta H^2$ is called the \emph{model space associated to $\vartheta$.}
Functions in this space admit an analytic continuation through the set $ \mathbb T \setminus \Sigma (\vartheta )$, where the latter set is the \emph{spectrum of $ \vartheta $,} 
defined to be the closed set 
\begin{equation*}
\Sigma (\vartheta ) \equiv  \textup{clos} (\Lambda  \cup \textup{supp} (\nu )) 
=  \left\{ \zeta \in \overline  {\mathbb D } \::\:  \liminf _{\substack{z\to \zeta \\ z\in \mathbb D  }}  \vartheta (z) = 0\right\}. 
\end{equation*}

Function theoretic properties of the space $K_\vartheta$  are   of significant interest, and 
we concentrate here on Carleson measures for the space.  
 Recall that a measure $\mu$ is a $K_\vartheta$-Carleson measure if we have the following estimate holding:
$$
\int_{\overline{\D}}\abs{f(z)}^2 d\mu(z)\leq C(\mu)^2\norm{f}_{K_\vartheta}^2\quad\forall f\in K_\vartheta.
$$
Since $ K _{\vartheta }$ is a subspace of $ H ^2 $, every Carleson measure for $ H ^2 $ is also one for $ K _{\vartheta }$, 
but its norm may be significantly smaller. And, a Carleson measure for $ K _{\vartheta }$ need not be one for $ H ^2 $. 

This problem has been intensely studied by numerous authors, with the question of characterization posed by Cohn \cite{MR705235} in 1982.  
An attractive special case when $ \vartheta $ satisfies the `one-component', or `connected level set' condition, namely that the enlargement of the spectrum, given by 
\begin{equation*}
\Omega (\epsilon ) \equiv \{z\in \mathbb D \::\: \lvert  \vartheta (z)\rvert < \epsilon   \}, \qquad \epsilon >0
\end{equation*}
is connected for some $ \epsilon >0$.  In this case, Cohn \emph{op.\thinspace cite} and Treil and Volberg \cite{MR849293}, showed that $ \mu $ is $ K _{\vartheta }$-Carleson if and only if 
the Carleson condition $ \mu (B_I) \lesssim \lvert  I\rvert $ holds for all intervals $ I$ such that the Carleson box $ B_I$ intersects $ \Omega (\epsilon )$.   See also the alternate proof obtained by Aleksandrov in \cite{MR1327512}.   For more general $ \vartheta $, see however the counterexample of Nazarov-Volberg \cite{NV}, based on the famous counterexample of Nazarov  \cite{N1} to the Sarason conjecture.   
Apparently, there are very few results known for general $ \vartheta $, with one 
of these being the remarkable results of Aleksandrov \cite{MR1464420} characterizing those $ \mu $ for which $ K _{\vartheta }$ isometrically embeds into $ L ^2 (\overline{\mathbb{D}};\mu )$, under the 
natural embedding map. 

In the case when the measure $\mu$ is supported on $\mathbb{T}$ a characterization of the $ K _{\vartheta }$-Carleson measures is a 
corollary of the two weight Hilbert inequality obtained in \cites{MR3285858,primer}.  However, for measures with more general supports, we need the full characterization obtained in this paper.
We now translate this problem to one about weighted estimates for the Cauchy transform by following the exposition of Nazarov and Volberg \cite{NV}.  Let $\sigma$ denote the Clark measure on $\mathbb{T}$ associated to $\vartheta$ (recall that this is the measure defined by $\frac{1+\vartheta(z)}{1-\vartheta(z)}=\int_{\mathbb{T}}\frac{1+z\overline{w}}{1-z\overline{w}}d\sigma(w)$).  Then we have that $L^2(\mathbb{T};\sigma)$ is unitarily equivalent to $K_\vartheta$ via a unitary $U$.  Moreover, we have that $U^*:L^2(\mathbb{T};\sigma)\to K_\vartheta$ has the integral representation given by
$$
U^*f(z)\equiv(1-\vartheta(z)) \int_{\mathbb{T}} \frac{f(\xi)}{1-\overline{\xi} z}\sigma(d\xi). 
$$
Note that $U^*$ is, up to a multiplication operator, the Cauchy transform $\mathsf{C}$ of  $f$ with respect to the measure $\sigma$.  For the inner function $\vartheta$ and measure $\mu$, define a new measure $\nu_{\vartheta,\mu}\equiv \lvert  1-\vartheta\rvert ^2  \mu$.  Then we have that $\mu$ is a Carleson measure for $K_\vartheta$ if and only if $\mathsf{C}: L^2(\mathbb{T};\sigma)\to L^2(\overline{\mathbb{D}};\nu_{\vartheta,\mu})$ is bounded.  Indeed, suppose that $\mathsf{C}:L^2(\mathbb{T};\sigma)\to L^2(\overline{\mathbb{D}};\nu_{\vartheta,\mu})$ is bounded.  Then we have that
$$
\int_{\overline{\mathbb{D}}} \left\vert U^* f(z)  \right\vert^2 d\mu(z)=\int_{\overline{\mathbb{D}}} \left\vert (1-\vartheta(z)) \mathsf{C}(f\sigma)(z) \right\vert^2 d\mu(z)=\int_{\overline{\mathbb{D}}} \left\vert \mathsf{C}(f\sigma)(z) \right\vert^2 d\nu_{\vartheta,\mu}(z)\leq\mathcal{C}^2\left\| f\right\|_{L^2(\mathbb{T};\sigma)}^2.
$$
For $g\in K_\vartheta$ let $f=Ug$, then the above inequality gives
$$
\int_{\overline{\mathbb{D}}} \left\vert g(z)  \right\vert^2 d\mu(z)\leq \mathcal{C}^2 \left\| U f\right\|_{L^2(\mathbb{T};\sigma)}^2=\left\| g\right\|_{K_\vartheta}^2.
$$
Thus, we have that $C(\mu)^2\leq\mathcal{C}^2$.  However, this argument is completely reversible, and we in fact arrive at $C(\mu)=\mathcal{C}$.  So understanding the Carleson measures for $K_\vartheta$ is equivalent to deducing the boundedness of $\mathsf{C}:L^2(\mathbb{T};\sigma)\to L^2(\overline{\mathbb{D}};\nu_{\vartheta,\mu})$  (a similar argument applies to the Hardy space of the upper half plane $\mathbb{C}_+$).  Our characterization of the Carleson measures for $K_\vartheta$ is given by the following theorem.

\begin{theorem}
\label{K_varthetaCM}
Let $\mu$ be a non-negative Borel measure supported on $\overline{\mathbb D }$ and let $\vartheta$ be an inner function on $\mathbb D $ with Clark measure $ \sigma $. 
Set $\nu_{\mu,\vartheta}= \lvert 1-\vartheta \rvert ^2 \mu$.  The following are equivalent:
\begin{itemize}
\item[(i)] $\mu$ is a Carleson measure for $K_\vartheta$, namely,
$$
\int_{\overline{\mathbb{D}}}\abs{f(z)}^2 d\mu(z)\leq C(\mu)^2\norm{f}_{K_\vartheta}^2\quad\forall f\in K_\vartheta;
$$
\item[(ii)] The Cauchy transform $\mathsf{C}$ is a bounded map between $L^2(\mathbb{T};\sigma)$ and $L^2(\overline{\mathbb{D}};\nu_{\mu,\vartheta})$, i.e.,  $\mathsf{C}:L^2(\mathbb{T};\sigma)\to L^2(\overline{\mathbb{D}};\nu_{\vartheta,\mu})$ is bounded;
\item[(iii)]  The three conditions in \eqref{e:Disk} hold for the pair of measures $ \sigma $ and $ \nu_{\mu,\vartheta}$. 
\end{itemize}
Moreover, 
$$ 
C(\mu)\simeq \left\Vert \mathsf{C}\right\Vert_{L^2(\mathbb{T};\sigma)\to L^2(\overline{\mathbb{D}};\nu_{\vartheta,\mu})}\simeq \mathscr A_2 ^{1/2} + \mathscr T.
$$
\end{theorem}
The equivalence between (i) and (ii) is sketched before the theorem, while the equivalence between (ii) and (iii) follows from interpreting Theorem \ref{t:disk} in the context at hand.   
The Clark measure $ \sigma $ is a singular measure supported on the set $ \{z : \vartheta (z) =1\}$, hence it and 
$ \nu_{\mu,\vartheta}= \lvert 1-\vartheta \rvert ^2 \mu$ do not have common point masses, so that the $ A_2 $ conditions 
could be phrased more simply. 

There is a variant of Theorem~\ref{t:compact} that holds, characterizing those measures $ \mu $ such that $ K _{\vartheta } $ embeds compactly into $ L ^2 (\overline {\mathbb D } ; \mu )$.   
The interested reader can combine Theorems \ref{K_varthetaCM} and \ref{t:compact} to formulate it.  
Variants of these results hold with the disk replaced by $ \mathbb C _+$, and we again leave the details to the reader. 

\subsubsection{Composition Operators on $K_\vartheta$}

Let $\varphi:\D\to\D$ be holomorphic.  The composition operator with symbol $\varphi$ is
$
C_{\varphi}f=f\circ \varphi 
$.
The Littlewood subordination principle implies that these operators are bounded from $ H ^2$ to $ H ^2$, and to the Bergman spaces $ A _{\alpha }$ given by  the norm 
\begin{equation*}
\lVert f\rVert_{ A _{\alpha}} ^2  \equiv \int _{\mathbb D } \lvert  f (z)\rvert ^2  (1 - \lvert  z\rvert ^2  )  ^{\alpha } dA(z), \qquad -1 < \alpha < \infty .  
\end{equation*}
However, the upper bound supplied by the subordination argument will not be sharp in general. 
The subject at hand is to describe finer operator theoretic properties in terms of  the properties of $ \varphi $. 
Since $ K _{\vartheta }$ is a subspace of $ H ^2 $, it follows that $ C _{\varphi }$ is bounded as a map from $ K _{\vartheta }$ to $ H ^2 $, or any of the Bergman spaces above.  
In this setting, we can characterize three properties of $ C _{\varphi }$:   the norm of $ C _{\varphi }$, its essential norm, and compactness. 
There is an extensive literature on the properties of composition operators, see the text \cite{MR1397026} for a relatively recent guide to it.

Let $ \tau $ be a weight on $ \overline  {\mathbb D }$, and define a Hilbert space of analytic functions by taking the closure of $ H ^{\infty } (\mathbb D )$ with respect to  the norm for $ L ^2 (\overline  {\mathbb D } ; \tau ) $. Call the resulting space $ H ^2 _{\tau }$.  Thus, if $ \tau $ is Lebesgue measure on $ \mathbb T $, the space $ H ^2 _{\tau }$ is the Hardy space, and if $ \tau (dA(z)) = (1 - \lvert  z\rvert ^2  )  ^{\alpha } dA(z)$, it is the Bergman space.  

To the function $\varphi$ and weight $ \tau $ we associate the pullback measure $\tau_\varphi$ defined as a measure on $\overline{\D}$, as 
$
\tau_{\varphi}(E)\equiv \tau (\varphi^{-1}(E))
$.  
Then
\begin{equation*}
\lVert C _{\varphi } f \rVert_{H ^2 _ \tau } ^2 
= \int _{\overline  {\mathbb D  } } \lvert  f \circ \varphi (z)\rvert ^2  \; \tau (d A(z)) 
= \int _{\overline  {\mathbb D } }  \lvert  f  (z)\rvert  ^2 \; \tau_{\varphi}(d A(z)). 
\end{equation*}
That is, $C_{\varphi}:K_\vartheta\to H^2 _{\tau }$ is unitarily equivalent to the embedding operator $I_{\tau_{\varphi}}: K_\vartheta\to L^2(\overline{\mathbb{D}};\tau_{\varphi})$.  Thus, we have that the boundedness of the composition operator $C_{\varphi}:K_{\vartheta}\to H^2 _{\tau }$ is equivalent to determining when $\tau_{\varphi}$ is a Carleson measure for $K_\vartheta$.  By Theorem~\ref{t:disk} we have the following answer. 

\begin{theorem} \label{t:composition}
Let $\vartheta$ be an inner function.  Let $\varphi:\D\to\D$ be analytic and let $\tau_{\varphi}$ denote the pullback measure associated to $\varphi$.  The following are equivalent:
\begin{itemize}
\item[(i)] $C_{\varphi}:K_\vartheta\to H^2 _{\tau}$ is bounded;
\item[(ii)] $\tau_{\varphi}$ is a Carleson measure for $K_\vartheta$, namely,
$$
\int_{\overline{\mathbb{D}}}\abs{f(z)}^2 \tau_{\varphi}(dA(z))\leq C(\tau_{\varphi})^2\norm{f}_{K_\vartheta}^2\quad\forall f\in K_\vartheta;
$$
\item[(iii)] The conditions of Theorem~\ref{t:disk} \textnormal{(iii)} hold for the pair of weights $ \sigma $ on $ \mathbb T $ and 
 $\nu_{\tau_{\varphi},\vartheta}=  \lvert 1-\vartheta\rvert ^2 \tau_{\varphi}$. 
\end{itemize}
\end{theorem}

A corresponding characterization of compactness can be obtained, in terms of the limits in Theorem~\ref{t:compact} being zero. 
If one is interested in the essential norm, the limits in Theorem~\ref{t:compact} should be taken to be limit superiors.  
A theorem of this level of generality is new even when $ K _{\vartheta }$ is replaced by the Hardy space $ H ^2 $. 

In the setting of composition operators from $ H ^2 $ to $ H ^2 $, MacCluer \cite{MR783578} characterized the compact operators in terms of the 
the measure $ \lvert  \mathbb T \cap \varphi ^{-1} (E)\rvert $ being a \emph{vanishing} Carleson measure.  Shapiro \cite{MR881273} calculates the 
essential norm of the same operators in terms of the Nevanlinna counting function $ N _{\varphi }$. Specializing his result to compactness,  the characterization states that 
$ N _{\varphi } (w) = o ( 1- \lvert  w\rvert ) $ as $ \lvert  w\rvert \to 1 $. 
The connection between the two approaches 
is analyzed in  Lef{\`e}vre-Li-Queff{\'e}lec-Rodr{\'{\i}}guez-Piazza \cite{MR2836660}.  In the setting of the Theorem above, one weight is the pullback measure, and the 
other is Lebesgue measure on the unit circle. If the pullback measure is a (vanishing) Carleson measure the conditions above can be verified 
by ad hoc means.  The survey \cite{MR2198367} includes additional points of view and references related to Shapiro's results.  

In the setting of composition operators from $ K _{\vartheta }$ to $ H ^2 $,  there is an elegant characterization of 
compactness due to  Lyubarskii-Malinnikova \cite{12055172}*{Theorem 1} expressed in terms of the Nevanlinna counting function of $ \vartheta $, namely that 
\begin{equation*}
 N _{\varphi } (w) \frac {1 - \lvert   \vartheta (w)\rvert ^{2}  } {1- \lvert  w\rvert ^2  } = o ( 1- \lvert  w\rvert ), \qquad  \lvert  w\rvert \to 1 . 
\end{equation*}
Roughly speaking, this condition only imposes the Shapiro condition as $ w$ approaches the spectrum $ \Sigma (\vartheta ) \subset \mathbb T $.    
We are not aware any other results at this level of generality for composition operators on $ K _{\vartheta } $ spaces.  

\subsection{Proof and Organization}

We concern ourselves with the proof of Theorem~\ref{t:main}.  The testing inequalities \eqref{e:T1} and \eqref{e:T*1} 
are obviously necessary, and we provide the known argument for necessity of the $ A_2$ condition \eqref{e:A2} below. 
The bulk of the argument concerns the sufficiency of the $ A_2$ condition and testing conditions for the norm estimate. 
For this we follow the model of the argument derived from the beautiful strategy of Nazarov-Treil-Volberg \cite{10031596}. 
This strategy works however for any Calder\'on-Zygmund operator, whereas delicate properties of the operator at hand 
must inform the proof. These additional elaborations were provided for the 
Hilbert transform in   \cites{10014043,MR3285858,MR3285857,13120843}.

Central here is the notion of \emph{monotonicity} and \emph{energy} inequalities, which control subtle off-diagonal terms in the proof.
These conditions are asymmetric with respect to the role of the  weights, and so these two conditions are different in the current setting.  
One of them involves geometric arguments that are not present in the setting of the Hilbert transform, and we present that case first below. 
The other uses \emph{both components} of the Riesz transform in order to control only \emph{part of} the energy.  
Both versions of the energy inequality require more sophisticated formulations than those for the Hilbert transform.  

Following the development of the energy inequalities, they must be bootstrapped to more complicated inequalities, in two different (highly non-obvious) ways. 
In addition, one must incorporate stopping data from the functions on which one is testing the norm of the Cauchy transform.  
Again the arguments are asymmetric with respect to the weights, but share many commonalities with the case of the Hilbert transform. 
The most delicate part of the argument, the \emph{Local Estimate}, 
 requires a more careful analysis, due to the more sophisticated formulation of the energy inequality. 
Accordingly, we present only the more novel of the two cases in full.

The next section has the standard random dyadic grid construction of \cite{10031596}; section \S\ref{s:necessary} is the essence of the matter, deriving the energy inequalities. Following that, the more robust parts of the argument are presented, with \S\ref{s:global} focusing on the global to local reductions, \S\ref{s:local} studying the local estimates, \S\ref{s:Bbelow} studying one of the bilinear forms that arise in \S\ref{s:global}. Finally, in \S\ref{s:elementary} we focus on the elementary estimates for bilinear forms arising from the analysis in the proof.

\section{Dyadic Grids, Good and Bad Decomposition} \label{s:dyadic}

Let $\hat{\mathcal{D}}$ denote the standard dyadic grid in $\R$.  A \textit{random} dyadic grid $\mathcal{D}$ is specified by $\xi\in\{0,1\}^{\Z}$ 
and choice of $ 1\le \lambda \le 2$.  The elements of $ \mathcal D$ are  given by
$$
I\equiv \hat{I}\dot+\xi = \lambda \biggl\{ \hat{I}+\sum_{n: 2^{-n}<|\hat{I}|} 2^{-n}\xi_n \biggr\}. 
$$
Place the uniform probability measure $\mathbb{P}$ on $ \xi \in \{0,1\}^{\Z}$, and choose $ \lambda $ with respect to normalized measure 
on $ [1,2]$ with measure $ \frac {d \lambda } \lambda $.

Fix $0<\epsilon<1$ and $r\in\N$.  An interval $I\in\mathcal{D}$ is said to be $(\epsilon,r)$-\textit{bad} if there is an interval $J\in\mathcal{D}$ such that $\vert J\vert>2^{r}\vert I\vert$ and $\textnormal{dist}(I,\partial J)<\vert I\vert^{\epsilon}\vert J\vert^{1-\epsilon}$.  Otherwise, an interval $I$ will be called $(\epsilon,r)$-\textit{good}.  We have the following well-known properties associated to the random dyadic grid $\mathcal{D}$.

\begin{prop}
\label{RandomLattice_1d}
The following properties hold:
\begin{enumerate}
\item The property of $I=\hat{I}\dot +\xi$ with $\hat{I}\in\hat{\mathcal{D}}$ being $(\epsilon,r)$-good depends only on $\xi$ and $\vert I\vert$;
\item $\mathbf{p}_{\textnormal{good}}\equiv\mathbb{P}\left(I \textnormal{ is } (\epsilon,r)-\textnormal{good}\right)$ is independent of $I$;
\item $\mathbf{p}_{\textnormal{bad}}\equiv 1-\mathbf{p}_{\textnormal{good}}\lesssim \epsilon^{-1} 2^{-\epsilon r}$.
\end{enumerate}
\end{prop}

We now indicate how, associated to a dyadic lattice $\mathcal{D}$ on $\R$, we create a ``dyadic lattice'' on $\mathbb R^2_+$.  For any interval $I\in\mathcal{D}$, we define $Q_I$, the Carleson square over the interval $I$, as the set 
$$
Q_I\equiv I\times[0,\vert I\vert).
$$
We then set $\mathcal D_+\equiv\left\{ Q_I: I\in\mathcal{D}\right\}$.  $\hat{\mathcal{D}}_+$ will denote the Carleson cubes associated to the standard dyadic lattice $\hat{\mathcal{D}}$ on $\R$.  Note that if we write $Q\in\mathcal D_+$, then there is a corresponding $I\in\mathcal{D}$ such that $Q=Q_I$.  For a cube $Q$ we let $\ell (Q)$ denote the side length of the cube, i.e for $Q=Q_I$, we have that $\ell (Q)=\ell (Q_I)=\vert I\vert=\vert Q\vert^{\frac{1}{2}}$.

The collection $ \mathcal D_+$ is extended to a dyadic grid $ \mathcal D ^2 $ on $ \mathbb R ^2_+$, defined to be all cubes of the form 
$ I \times \lvert  I\rvert ([0, 1)+ n) $, for $ n\in \mathbb N $.  
Similar to above, we have the following notion of good and bad cubes.  Fix $0<\epsilon<1$ and $r\in\N$.  A cube $Q\in\mathcal{D} ^2 $ is said to be $(\epsilon,r)$-\textit{bad} if there is a cube $Q'\in\mathcal{D} ^2 $ such that $\ell (Q')>2^{r}\ell (Q)$ and $\textnormal{dist}(Q,\partial Q')<\ell (Q)^{\epsilon}\ell (Q')^{1-\epsilon}$.  Otherwise, a cube $Q$ will be called $(\epsilon,r)$-\textit{good}.  Similar to Proposition \ref{RandomLattice_1d} we have the following properties holding.

\begin{prop}
\label{RandomLattice_UpperHalfPlane}
The following properties hold:
\begin{enumerate}
\item The property of $Q=\hat{Q}+(\xi,0)$ with $\hat{Q}\in\hat{\mathcal{D}}_+$ being $(\epsilon,r)$-good depends only on $\xi$ and $\ell (Q)$;
\item $\mathbf{p}_{\textnormal{good}}\equiv\mathbb{P}\left(Q \textnormal{ is } (\epsilon,r)-\textnormal{good}\right)$ is independent of $Q$;
\item $\mathbf{p}_{\textnormal{bad}}\equiv 1-\mathbf{p}_{\textnormal{good}}\lesssim \epsilon^{-1} 2^{-\epsilon r}$.
\end{enumerate}
\end{prop}


We introduce the Haar basis adapted to the weights $\sigma$ on $\R$ and $\tau$ on $\mathbb R^2_+$.  
To avoid cumbersome notation later on, for a set $E$ we will identify the set with its indicator function, i.e. $E\equiv \mathsf{1}_E$.  For $I\in\mathcal{D}$,  
if $ \sigma $ assigns positive mass to both children of $ I$,  define
\begin{equation}
\label{Haar-1d}
h_I^{\sigma}\equiv \sqrt{\frac{\sigma(I_+)\sigma(I_-)}{\sigma(I)}}\left(\frac{I_+}{\sigma(I_+)}-\frac{I_-}{\sigma(I_-)}\right).
\end{equation}
Otherwise, set $ h ^{\sigma }_I= 0$. 
Note that this is a $L^2(\R;\sigma)$ normalized function and has integral $0$ with respect to $\sigma$.  Also $\{h_I^{\sigma} \::\: I\in\mathcal{D}, h ^\sigma _I \neq 0\}$ is an orthonormal basis of $L^2(\R;\sigma)$.  We will also let $\widehat{f}_{\sigma}(I)\equiv\left\langle f,h_I^{\sigma}\right\rangle_\sigma$ and will let
$$
\Delta_I^{\sigma} f\equiv \left\langle f,h_I^{\sigma}\right\rangle_\sigma h_I^{\sigma}=I_+\mathbb{E}_{I_+}^{\sigma} f+I_-\mathbb{E}_{I_-}^{\sigma}f-I\mathbb{E}_{I}^{\sigma} f.
$$ 
Otherwise, we set $ \Delta_I^{\sigma} f\equiv 0$. 
We have the following identify from this formula
$$
\mathbb{E}_J^{\sigma} f=\sum_{I\supsetneq J} \mathbb{E}_J^{\sigma}\Delta_I^{\sigma} f.
$$

We next discuss the Haar basis on $\mathbb{R}^2_+$.  Given a cube $ Q \in\mathcal D^2$, with $ \tau (Q') >0$, for at least two children $ Q'$ of $ Q$,  set
\begin{align*}
\Delta ^{\tau} _{Q}  g \equiv \sum_{\substack{\textup{$ Q'$ a child of $ Q$}\\ \tau (Q') >0}}  \mathbb E ^{\tau } _{Q'} g \cdot Q'- \mathbb E ^{\tau } _{Q} g  \cdot Q.  
\end{align*}
This is the standard martingale difference. 

Say that $ \mathcal D$ is \emph{admissible} for $ \sigma $ and $ \tau $ if and only if $ \sigma $ does not have a point mass at the endpoint 
of any interval $ I\in \mathcal D$ and $ \tau $ does not  assign positive mass  to the boundary of any $ Q\in  \mathcal D ^2  $.  
With probability one the  random grid $ \mathcal D$ is admissible, and we always assume this below.
This is so, due to the incorporation of the dilation factor into the definition of the random grid, though throughout, we will assume that the dilation factor is $1$.  
Thus, with probability one, we can define the Haar bases $\{\Delta _Q^{\tau}\}_{Q\in\mathcal D^2}$ and $\{h_I^{\sigma}\}_{I\in\mathcal{D}}$ as above.  Now, write the identity operator in $L^2(\R;\sigma)$ as
$$
f=P^\sigma_{\textnormal{good}} f+P^\sigma_{\textnormal{bad}}f \quad\textnormal{ where } P^\sigma_{\textnormal{good}} f\equiv \sum_{I\in\mathcal{D}: I \textnormal{ is } (\epsilon, r)-\textnormal{good}} \Delta_I^{\sigma} f.
$$
Similar notation applies for the identity operator on $L^2(\mathbb R^2_+;\tau)$. For the remainder of the paper we let $\Vert\cdot\Vert_{\sigma}$ and $\left\langle\cdot,\cdot\right\rangle_{\sigma}$ denote the norm and inner product in $L^2(\mathbb{R};\sigma)$.  Identical notation applies for $\Vert\cdot\Vert_{\tau}$ and $\left\langle\cdot,\cdot\right\rangle_{\tau}$ in $L^2(\mathbb{R}^2_+;\tau)$.

We have the following proposition.
\begin{prop}  
\label{BadExpectation}
The following estimate holds:
$$
\mathbb{E}\left\Vert P_{\textnormal{bad}}^\sigma f\right\Vert_{\sigma}^2\lesssim \epsilon^{-1} 2^{-\epsilon r}\Vert f\Vert_{\sigma}^2.
$$
An identical estimate is true for the function $g$ and the weight $\tau$.
\end{prop}

Using the ideas of good and bad dyadic cubes in non-homogeneous harmonic analysis, one can deduce that it suffices to study only good functions.  
This reduction is very familiar, and its requirements well-known.  
To make this reduction we will make these standing assumptions:  First, that there is an \emph{a priori} inequality 
\begin{equation*}
 \left\Vert \mathsf{R}_\sigma f\right\Vert_{\tau}\le \mathscr N \left\Vert f\right\Vert_{\sigma}. 
\end{equation*}
And, second, for any pair of functions $ f \in L ^2 (\mathbb R , \sigma )$ and $ g \in L ^2 (\mathbb R ^2 _+, \tau  )$ 
with disjoint compact supports, there holds 
\begin{equation}\label{e:canonical*}
\langle \mathsf R _{\sigma } f, g \rangle _{\tau } = \int _{\mathbb R } \int _{\mathbb R ^2 _+} 
f (y) g (x) \frac { x-y} { \lvert  x-y\rvert ^2  } \; \tau (dx) \, \sigma (dy) . 
\end{equation}
We will refer  to this as the \emph{canonical value} of $ \langle {\mathsf R} _{ \sigma }  f   , g \rangle _{\tau } $.  

This lemma lets one reduce to only considering good functions in the bilinear form associated to $\mathsf{R}_\sigma$. 

\begin{lm}
\label{Reduction_1}
Let  $\sigma$ and $\tau$ be a pair of weights for which the \emph{a priori} inequality \eqref{Cauchy_Est1} holds, and  $\mathscr R \equiv \mathscr A_2 ^{1/2} + \mathscr T$ is finite.  Suppose that for a choice of $ 0< \epsilon <1$,  all $ r\in \mathbb N $,  and all  admissible dyadic grids $\mathcal{D}$ and $\mathcal D^2$ 
\begin{equation}
\label{ReductiontoGood}
\bigl\lvert \langle\mathsf{R}_{\sigma} ( P_{\textnormal{good}}^\sigma f) , P_{\textnormal{good}}^\tau g\rangle_{\tau}\bigr\rvert
\le C _{\epsilon ,r}\mathscr R\left\Vert f\right\Vert_{\sigma}\left\Vert g\right\Vert_{\tau}.
\end{equation}
Then  \eqref{Cauchy_Est1} holds with $\mathscr{N}\lesssim\mathscr R$.
\end{lm}

\begin{proof}[Proof of Lemma~\ref{Reduction_1}]  
Let $f\in L^2(\R;\sigma)$ and $g\in L^2(\mathbb R^2_+;\tau)$ be arbitrary functions.  Then we have 
\begin{equation*}
\vert \langle \mathsf{R}_\sigma f,g\rangle_{\tau}\vert \leq\vert \langle \mathsf{R}_\sigma P_{\textnormal{good}}^\sigma f,P_{\textnormal{good}}^\tau g\rangle_{\tau}\vert + \vert \langle \mathsf{R}_\sigma P_{\textnormal{good}}^\sigma f,P_{\textnormal{bad}}^\tau g\rangle_{\tau}\vert +  \vert \langle \mathsf{R}_\sigma P_{\textnormal{bad}}^\sigma f, g\rangle_{\tau}\vert.
\end{equation*}
Now using \eqref{ReductiontoGood} and the \emph{a priori} inequality \eqref{Cauchy_Est1}, 
$$
\left\vert \left\langle \mathsf{R}_\sigma f,g\right\rangle_{\tau}\right\vert\le C _{r, \epsilon }\mathscr R\Vert f\Vert_\sigma\Vert g\Vert_{\tau}+\mathscr N(\Vert f\Vert_\sigma\Vert \mathsf{P}_{\textnormal{bad}}^{\tau}g\Vert_\tau+\Vert g\Vert_\tau\Vert \mathsf{P}_{\textnormal{bad}}^{\sigma}f\Vert_\tau).
$$
Taking expectation over the choice of random grid $\mathcal{D}$ and $\mathcal D^2$ we have by Proposition \ref{BadExpectation} that
$$
\left\vert \left\langle \mathsf{R}_\sigma f,g\right\rangle_{\tau}\right\vert\lesssim C _{\epsilon ,r}\mathscr R\Vert f\Vert_\sigma\Vert g\Vert_{\tau}+\epsilon^{-1}2^{-r\epsilon}\mathscr N \Vert f\Vert_\sigma\Vert g\Vert_\tau.
$$
But, for appropriate selection of $f\in L^2(\R;\sigma)$ and $g\in L^2(\mathbb R^2_+;\tau)$ we have
$$
\Vert f\Vert_\sigma\Vert g\Vert_\tau \mathscr N\lesssim \left\vert \left\langle \mathsf{R}_\sigma f,g\right\rangle_{\tau}\right\vert\lesssim  C _{\epsilon ,r}\mathscr R\Vert f\Vert_\sigma\Vert g\Vert_{\tau}+\epsilon^{-1}2^{-r\epsilon} \mathscr N \Vert f\Vert_\sigma\Vert g\Vert_\tau.
$$
Choosing $r$ sufficiently large for a given $\epsilon$ then lets one conclude that
$
\mathscr N\lesssim\mathscr R
$.
\end{proof}

One way to obtain the \emph{a priori} inequality is to  impose a standard truncation on the singular integral.  
They are defined as follows, for all $ 0 < \alpha < \beta $
\begin{gather}\label{e:standard}
\begin{split}
{\mathsf R} _{\alpha , \beta } (\sigma f) (x) \equiv 
\int _{ \mathbb R }  K _{\alpha , \beta } (x, t) f (t) \; \sigma (dt) ,
\\ \textup{where} \quad 
 \lvert x- t\rvert  \cdot \lvert    K _{\alpha , \beta } (x,t)\rvert + 
 \lvert x- t\rvert ^2   \cdot \lvert  \nabla  K _{\alpha , \beta } (x,t)\rvert 
 \le C \mathbf 1_{  [\alpha /2, 2 \beta]} ( \lvert  x-t\rvert  )\,, 
\\ \textup{and} \quad 
   K _{\alpha , \beta } (x, t) = \frac {x-t} {\lvert x-  t\rvert ^2  }, \qquad  \textup{if $ $} \alpha <  \lvert x- t\rvert  < \beta,  \ t \in \mathbb R ^2 _+.   
\end{split}
\end{gather}
Thus, the choice of kernel $  K _{\alpha , \beta }$ is compactly supported, satisfies a size  and gradient condition. 
Finally, it  agrees with the Riesz  transform kernel for most values where it is not zero. 
With a choice of standard truncations one can then define a uniform norm over certain truncations to be 
the best constant $  {\mathscr N _{\alpha _0, \beta _0}}$ 
in the inequality below, in which $ 0 < \alpha _0 < \beta _0$
\begin{equation}\label{e:dotN}
 \sup _{\alpha _0< \alpha < \beta < \beta _0  } 
 \lVert {\mathsf R} _{\alpha , \beta } (\sigma f) \rVert_{\tau} \le  {\mathscr N _{\alpha _0, \beta _0}} \lVert f\rVert_{\sigma} . 
\end{equation}
It is elementary to see that $ \mathscr N _{\alpha _0, \beta _0} \le C _{\alpha _0 , \beta _0} \mathscr A_2$.  
That is, assuming the $A_2$ condition, there is always an \emph{a priori} inequality for standard truncations.  
Moreover, under the assumptions of the main theorem, one can show that the limit below will be finite:   
\begin{equation*}
\lim _{\alpha \downarrow 0} \lim _{\beta \uparrow \infty } 
\langle {\mathsf R} _{\alpha , \beta } (\sigma f) , g \rangle _{\tau}, 
\qquad f \in L ^2 (\mathbb R , \sigma ),\  g \in L ^2 (\mathbb R ^2 _+, \tau  ).  
\end{equation*}
This is left to the reader, as the additional complications needed to prove this do not require any 
ideas that go beyond the scope of this paper.

\section{Necessary Conditions} \label{s:necessary}

We begin with some conventions.  
\begin{itemize}
\item For two dyadic intervals $ I,J$ and integer $ s$ we write $ J \Subset_ s I$, and say `$ J$ is $ s$-strongly contained in $ I$' if $ J\subset I$ and $ 2 ^{s} \lvert  J\rvert\le \lvert  I\rvert  $.    
We are interested in the cases of $ s=r$, the integer associated with goodness, and $ s=4r$. 
The value of $ 4r$ is used in \S\ref{s:global}, and the value of $ r$ is used in this section and \S\ref{s:local}.  

\item  The center of the interval $ J$ is denoted  $ t_J$ and $ x_Q$ is the center of cube $ Q$. 
The Poisson average at interval $ I$ is $ \mathsf P ( f, I) \equiv \mathsf P f (x_I, \lvert  I\rvert )$, and we frequently appeal to the approximation 
\begin{equation*}
\mathsf P ( f, I)  = \int  _{\mathbb R }\frac {\lvert  I\rvert } { \lvert  I\rvert ^2 + \textup{dist} (t,I) ^2  } f (t)  \; dt 
\simeq 
 \int  _{\mathbb R }\frac {\lvert  I\rvert } { [\lvert  I\rvert  + \textup{dist} (t,I)] ^2  } f (t)  \; dt. 
 \end{equation*}
 The same approximation holds for the Poisson average on $ \mathbb R ^2 _+$. 

\item We are working with operators that carry $ L ^2 (\mathbb R ; \sigma )$ into $ L ^2 (\mathbb R ^2 _+; \tau )$, as well as their duals.
We will use $ f, \phi \in L ^2 (\mathbb R ; \sigma )$ and $ g, \varphi \in L ^2 (\mathbb R ^2 _+; \tau )$ to denote the functions being acted on by the operators in question.

\item There is opportunity for confusion about ranges of integration.  Generically, we will have the integration variable $ t \in \mathbb R $, and $ x\in \mathbb R ^2 _+$. 
But, if we write $ x - t$, we are viewing $ t \in \mathbb R \subset \mathbb R ^2 $, with the natural inclusion.   If the role of the dimensions are important for $ x$, we will write $ x=(x_1, x_2)$  or $ z= x_1 + i x_2$ as integrating variables.

\item Letters in sans-serif denote operators, for example $ \mathsf R$ denotes the Riesz transforms, and we will have other related variants, such as the Poisson and a Bergman-like operator.    

\item The `hard' case is the analysis of the operators in the  direction from  $ L ^2 (\mathbb R ^2 _+; \tau )$ into  $ L ^2 (\mathbb R ; \sigma )$.  This case is treated first,  followed by the reverse direction. The two cases are not dual.  Herein, we develop the \emph{monotonicity principle}, 
and the \emph{energy principle,} both in parts I and II.  

\item For a dyadic interval $ I\in \mathcal D$, we set $ \pi I $ to be its parent in $ \mathcal D$, the minimal interval in $ \mathcal D$ 
which strictly contains $ I$.  For a dyadic subtree $ \mathcal F\subset \mathcal D$, we set $ \pi _{\mathcal F} I$ to be the 
minimal element $ F\in \mathcal F$ with $ I\subset F$ (so $ \pi _{\mathcal F} F= F$).  
At a key point, we will use the notation $ \pi ^{2} _{\mathcal F} I$ to be the minimal element of $ \mathcal F$ which strictly contains $ \pi _{\mathcal F} I$  (it is the $ \mathcal F$-grandparent of $ I$ for $I\notin\mathcal{F}$ and it the $\mathcal{F}$-parent of $I$ for $I\in\mathcal{F}$).

\end{itemize}

\subsection{List of Common Notation}
\begin{itemize}
\item[$\mathcal{D}$:]  A random dyadic lattice.
\item[Good/Bad:] Fix $0<\epsilon<1$ and $r\in\N$.  An interval $I\in\mathcal{D}$ is said to be $(\epsilon,r)$-\textit{bad} if there is an interval $J\in\mathcal{D}$ such that $\vert J\vert>2^{r}\vert I\vert$ and $\textnormal{dist}(I,\partial J)<\vert I\vert^{\epsilon}\vert J\vert^{1-\epsilon}$.  Otherwise, an interval $I$ will be called $(\epsilon,r)$-\textit{good}.
\item[$Q_I$:] For any interval $I\in\mathcal{D}$, we define $Q_I$, the Carleson square over the interval $I$, as the set  
$$
Q_I\equiv I\times[0,\vert I\vert).
$$
\item[$ \ell(Q)$:] The side length of the square $Q$;
\item[$t_J$:]  The center of the interval $ J$;
\item[$x_Q$:] The center of the cube $Q$.
\item[$ h_I^{\sigma}$:] The Haar function adapted to $\sigma$:
\begin{equation*}
h_I^{\sigma}\equiv \sqrt{\frac{\sigma(I_+)\sigma(I_-)}{\sigma(I)}}\left(\frac{I_+}{\sigma(I_+)}-\frac{I_-}{\sigma(I_-)}\right).
\end{equation*}
\item[$ \mathsf P ( f, I)$:] The Poisson average at interval $ I$ of $f$, i.e., $ \mathsf P ( f, I) \equiv \mathsf P f (x_I, \lvert  I\rvert )$.  The standard approximation will be frequently used:
\begin{equation*}
\mathsf P ( f, I)  = \int  _{\mathbb R }\frac {\lvert  I\rvert } { \lvert  I\rvert ^2 + \textup{dist} (t,I) ^2  } f (t)  \; dt 
\simeq 
 \int  _{\mathbb R }\frac {\lvert  I\rvert } { [\lvert  I\rvert  + \textup{dist} (t,I)] ^2  } f (t)  \; dt. 
 \end{equation*}
 The same approximation holds for the Poisson average on $ \mathbb R ^2 _+$. 
\item[$ E (\sigma , I)$:] Energy adapted to $\sigma$ and the dyadic interval $I$:
 \begin{equation*}
E (\sigma , I) ^2  \equiv   \sigma (I) ^{-1} \sum_{\substack{J \::\: J\subset I\\ \textup{$ J$ is good} }} \left\langle\frac t {\lvert  I\rvert }, h ^{\sigma } _{J} \right\rangle_{\sigma } ^2 .
\end{equation*}
\item[$ J \Subset_ s I$:] For two dyadic intervals $ I,J$ and integer $ s$ we write $ J \Subset_ s I$, and say `$ J$ is $ s$-strongly contained in $ I$ if $ J\subset I$ and $ 2 ^{s} \lvert  J\rvert\le \lvert  I\rvert  $.    
We are interested in the cases of $ s=r$, the integer associated with goodness, and $ s=4r$. 
The value of $ 4r$ is used in \S\ref{s:global}, and the value of $ r$ is used in this section and \S\ref{s:local}.  
\item[$\pi I$:] For a dyadic interval $ I\in \mathcal D$, we set $ \pi I $ to be its parent in $ \mathcal D$, the minimal interval in $ \mathcal D$ which strictly contains $ I$.  
\item[$\pi_{\mathcal{F}} I$:] For a dyadic subtree $ \mathcal F\subset \mathcal D$, we set $ \pi _{\mathcal F} I$ to be the  minimal element $ F\in \mathcal F$ with $ I\subset F$ (so $ \pi _{\mathcal F} F= F$).  
\item[$\pi^{2}_{\mathcal{F}} I$:]  The minimal element of $ \mathcal F\subset\mathcal D$ which strictly contains $ \pi _{\mathcal F} I$  (it is the $ \mathcal F$-grandparent of $ I$). 
\item[$\mathcal D _{f}^r $:] $ \mathcal D _{f}^r$ is a grid containing all the children of intervals in the Haar support of $ f$ defined by: $ \mathcal D _{f}^r \equiv \{ I\in \mathcal D \::\: I\subset I^0\,\  \log_2 \lvert  I\rvert \in r \mathbb Z +s_f \}$.  
\item[$ \tilde \pi _{\mathcal F} J$:]  The smallest member of $ \mathcal F$ so that $ J\Subset _{4r}  F$.
\item[$ \mathsf  T _{\tau } g$:]   The operator  which maps functions on $ \mathbb R ^2 _+$ to $ \mathbb R ^2 _+$ used to establish the energy estimate, where $ \tau $ is a measure on $ \mathbb R ^2 _+ $.  
\begin{gather}
\label{e:Tdef}
\mathsf T _{\tau } g (x)  \equiv \int _{\mathbb R ^2 _+}  \frac { g (y) \cdot x_2   } { y^2_2 + (y_1-x_1)^2 + x_2^2  } \; \tau (dy)
\end{gather}
\item[$ V _I$:]  A set relevant to the energy inequality, see Figure~\ref{f:vee}.  
\begin{gather}
 V _I \equiv  \bigcup _{t\in I}\{ x = (x_1, x_2) \::\:   2\lvert  x_1-t\rvert <  x_2   \} .  
\end{gather}
\item[$\widehat Q _{F',F} $] The set $ \widehat Q _{F',F} = Q _{F'} \setminus Q _{F''}$, where 
$ F''$ is the $  \mathcal F$-child of $ F'$ that contains $ F$. 
\item[$ \mathcal W\!I$:]  The partition of the interval $ I$ consisting of those maximal intervals 
$ K\Subset_r I$ with $ \textup{dist} (K, \partial I) \ge \lvert  K\rvert ^{\epsilon } \lvert  I\rvert ^{1- \epsilon }  $.
\item[$H _F^{\tau } g $, $
\tilde H _F^{\sigma } f$:] Haar projection onto a certain collection of intervals associated to $F\in\mathcal{F}$:
\begin{align*}
H _F^{\tau } g &\equiv   \sum _{I \::\: \pi _{\mathcal F} I=F} \Delta ^{\tau } _{Q_I} g  , 
\\
\tilde H _F^{\sigma } f &\equiv \sum _{\substack{J \::\:  \tilde \pi _{\mathcal F} J=F}} \Delta ^{\sigma  } _{J} f.   
\end{align*}
\item[$ P _{\textup{Car}} ^{\tau } g$:] Haar projection associated to Carleson cubes: $$ P _{\textup{Car}} ^{\tau } g \equiv \sum_{I \in \mathcal D} \Delta ^{\tau } _{Q_I} g $$
\item[$  \tilde {\mathsf T }_{\tau } g (K)$:]  Approximation to the Bergman-type operator $T_\tau g$:
\begin{align}
 \tilde {\mathsf T }_{\tau } g (K)  &\equiv \int _{\mathbb R ^2 _+}  \frac { g (y)  } { y^2_2 + \lvert  K\rvert ^2  + \textup{dist}(y_1, K)^2 } \; \tau (dy)   , 
\end{align}
\item[$\mu_K$:] Measure indexed by the dyadics:
\begin{align}
\mu_K &\equiv 
\sum_{\substack{J \::\: J\subset K \\ \tilde \pi _{\mathcal F}   J = F}}
\langle  t, h ^{\sigma } _{J}\rangle _{\sigma } ^2 , \qquad F\in \mathcal F,\ K\in \mathcal WF. 
\end{align}
\item[$ \ddot\pi _{\mathcal G} P_2$:] The minimal element $ G\in \mathcal G$ with  $ P_2 \subset G$ and  $ P_2 \Subset_r  \pi G$.
\item[$ \Pi ^{\sigma } _{G }$:]  
 The projection onto the Haar coefficients $ P_2\in \mathcal Q_2$ with $ \ddot\pi _{\mathcal G}P_2=G$.   
\item[$ \ddot\pi _{\mathcal L}P_2$:] The minimal element of $ L\in \mathcal L$ with $ P_2 \Subset_r L$.
\item[$ \ddot\pi ^{t+1} _{\mathcal L} P_2$:] The minimal member of $ \mathcal L$ that \emph{strictly} contains $ \ddot\pi ^{t} _{\mathcal L} P_2$, where $\ddot\pi^1_{\mathcal{L}}=\ddot\pi_{\mathcal{L}}$.
\end{itemize}

\subsection{$ A_2$ Condition}

We show that the assumed norm inequality implies the $ A_2$ condition. 
\begin{proposition}\label{p:A2}  There holds  $ \mathscr A_2 \lesssim{\mathscr N} ^2 $.  
\end{proposition}

\begin{proof}
Fix an interval $ I$, and  note that 
\begin{equation*}
\lvert  \mathsf R ^{\ast} _{\tau } Q_I (t)\rvert \gtrsim \frac {\tau (Q_I)} { \lvert  I\rvert + \textup{dist}(t,I) }  , \qquad t\not\in I.  
\end{equation*}
Squaring, and integrating against the measure $\sigma $, the assumed norm inequality gives us 
\begin{equation*}
 \frac {\tau (Q_I) ^2 } {\lvert  I\rvert } \int _{\mathbb R \setminus I}  \frac {\lvert  I\rvert } { (\lvert  I\rvert + \textup{dist}(t,I) ) ^2 } \; \sigma (dt) 
\lesssim \mathscr N ^2 \tau (Q_I).  
\end{equation*}
Dividing out by $ \tau (Q_I)$ will complete the first half of the $ A_2$ bound.  

For the second half, note that 
\begin{equation*}
\lvert  \mathsf R _{\sigma  } I (x)\rvert \gtrsim \frac {\sigma (I)} { \lvert  I\rvert + \textup{dist}(x, Q_I) }  , \qquad x\not\in Q_I.  
\end{equation*}
And then the argument is completed as before. 

\end{proof}

\subsection{Monotonicity, I}

We begin our discussion of the critical off-diagonal considerations, where we will use the condition 
\begin{gather}\label{e:nabla}
 \left\lvert \nabla ^{j}  \frac {x-t} { \lvert  x-t\rvert ^2  }\right\rvert
\lesssim   \frac 1 { \lvert  x-t\rvert ^{1+j} }, \qquad j=1,2. 
\end{gather}
\emph{Monotonicity} refers to the domination of off-diagonal inner products by positive operators. 
In this direction, we can establish the monotonicity principle, using the  operators and sets given by 
\begin{gather}
\label{e:Ttdef}
\mathsf T _{\tau } g (x)  \equiv \int _{\mathbb R ^2 _+}  \frac { g (y) \cdot x_2   } { y^2_2 + (y_1-x_1)^2 + x_2^2  } \; \tau (dy)   , 
\\
\label{e:Vdef}  
 V _I \equiv  \bigcup _{t\in I}\{ x = (x_1, x_2) \::\:   \lvert  x_1-t\rvert <  2x_2   \} .  
\end{gather}

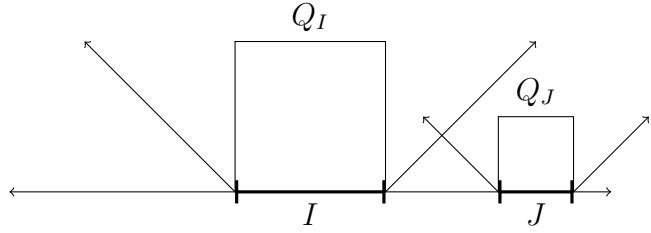
\begin{figure}
\begin{tikzpicture}
\draw[<->] (-3,0) -- (5,0);  
\draw (0,0) rectangle (2,2);  
\draw  (1, 2) node[above] { $ Q_I$}; 
\draw[|-|,very thick] (0,0) -- (2,0) node[midway,below] {$ I$}; 
\draw[->] (0,0) -- (-2,2);   \draw[->] (2,0) -- (4,2) ; 

\begin{scope}[xshift=3.5cm] 
\draw (0,0) rectangle (1,1);
\draw (.5,1) node[above] {$ Q_J$}; 
\draw[|-|, very thick] (0,0) -- (1,0) node[midway,below] {$ J$}; 
\draw[->] (0,0) -- (-1,1);   \draw[->] (1,0) -- (2,1); 
\end{scope}
\end{tikzpicture}

\caption{Two examples of sets $ V_I$, which are the regions between the two oblique lines at either endpoint of the intervals.  }
\label{f:vee}
\end{figure}

\begin{lemma}{\textnormal{[Monotonicity, I]}}
\label{l:monoI} 
If $ \varphi \in L^2(\mathbb{R}^2_+;\tau)$ is non-negative and 
compactly supported on the complement of the set $ V_I$,  and $ 10\cdot J \subset I$, then 
\begin{align}  \label{e:monoI2}
\left\vert   \left\langle  {\mathsf  R } ^{\ast } _{\tau } \varphi ,  h ^{\sigma } _{J'} \right\rangle _{\sigma }\right \vert 
& \simeq   \mathsf T_{\tau }   \varphi  ( x _{Q_J} )  \cdot 
\left\langle\frac t {\lvert  J\rvert }, h ^{\sigma } _{J'} \right\rangle _{\sigma } , \qquad J'\subset J.
\end{align}
Assume   that $ \varphi \in L^2(\mathbb{R}^2_+;\tau)$ is supported on the complement of the set $ Q_I$,   $ 10 \cdot J\subset I$  and $ f\in L^2(\mathbb{R};\sigma) $, supported on $ J$, 
has $ \sigma $-integral zero. 
Then these two estimates hold: 
\begin{align}
\label{e:MONO}
\left\vert   \left \langle{\mathsf  R } ^{\ast } _{\tau   }      \varphi    ,  h ^{\sigma } _{J'} \right\rangle _{\sigma } \right\vert 
& \lesssim 
 \mathsf T_{\tau } \lvert  \varphi \rvert (x_{Q_J} )  \cdot 
\left\langle\frac t {\lvert  J\rvert }, h ^{\sigma } _{J'} \right\rangle _{\sigma } , \qquad J'\subset J, 
\\
\label{e:monoI1}
\left\vert  \left\langle  {\mathsf  R } ^{\ast } _{ \tau }  \varphi  ,  f  \right\rangle _{\sigma } \right\vert 
&\lesssim \mathsf T_{\tau } \lvert  \varphi \rvert (x_{Q_J} ) \cdot  \int _{J} \lvert  f(t) \rvert \; \sigma(dt) . 
\end{align}

\end{lemma}

\begin{proof}
We are in a situation where \eqref{e:canonical*} holds.   The basic property is that 
\begin{equation}  \label{e:ht}
\langle  t, h ^{\sigma } _{J} \rangle _{\sigma } = \int _{J} (t-t_J) h ^{\sigma } _{J}(t) \; \sigma(dt) 
\end{equation}
and that the integrand is non-negative. From this, and a standard application of the kernel estimates in \eqref{e:nabla}, the 
estimates \eqref{e:MONO} and \eqref{e:monoI1} follow.  

\smallskip

We turn to the critical  equivalence \eqref{e:monoI2}, for which we only need to prove the lower bound. 
From the integral zero property of the Haar function, we have the equality below: 
\begin{align}
\left\langle {\mathsf  R } ^{\ast}_{\tau  }  \varphi , h ^{\sigma } _{J}  \right\rangle _{\sigma } 
&= 
\int\!\!\!\int _{\mathbb R ^2_+ \setminus V_I} \int _{J}  \varphi (x) h ^{\sigma } _{J} (t)  \frac { x-t} { \lvert  x-t\rvert ^2  } \; \sigma (dt)\, \tau (dx) 
\\   \label{e:RI}
&= 
\int\!\!\!\int _{\mathbb R ^2_+ \setminus V_I} \int _{J}  \varphi (x) h ^{\sigma } _{J} (t) 
\left\{
\frac { x-t} { \lvert  x-t\rvert ^2  } - 
\frac { x-t_J} { \lvert  x-t_J\rvert ^2  }
\right\} \; \sigma (dt)\, \tau (dx) . 
\end{align}

Now, the first coordinate of the term in braces above satisfies for all $ t\in J$, 
and $ x\in \mathbb R ^2 _+ \setminus V_I$, 
\begin{align}  \label{e:Bigger}
-\textup{sgn} (h _{J} ^{\sigma } (t)) \left\{
\frac { x_1-t} { \lvert  x-t\rvert ^2  } - 
\frac { x_1-t_J} { \lvert  x-t_J\rvert ^2  }  \right\} 
&\ge c 
\frac { \lvert  t-t_J\rvert }  {  \lvert  J\rvert ^2 + \textup{dist} (x_1, J) ^2  } .
\end{align}
Since $ \varphi \geq 0$, this clearly completes the proof of \eqref{e:monoI2}. 

Take $ 0 < \delta = t-t_J < \lvert  J\rvert/2 $.  
The difference on the left is the increment of the  function 
$
\frac { u } {u ^2 + x_2 ^2 } 
$, 
over the interval  $  u=x_1 - t_J$  to  $ u+ \delta = x_1 -t$.  
By the Fundamental Theorem of Calculus, the increment is 
\begin{align*}
\int _{u} ^{u+ \delta }  \frac { v ^2 - x_2 ^2 } { ( v ^2 + x_2 ^2 ) ^2 } \; d v 
& \geq \frac 34 
\int _{u} ^{u+ \delta }  \frac { v ^2} { ( v ^2 + x_2 ^2 ) ^2 } \; d v 
\\
& \geq \frac 34  \delta \frac { (u+ \delta ) ^2  } { ( (u+ \delta ) ^2 + x_2 ^2 ) ^2 } 
\\
&\geq    c  \frac {t - t_J} {  \lvert  x-t_J\rvert ^2  }. 
\end{align*}
In the numerator in the integral, we have $ v ^2 \geq 4 x ^2 _2 $, by the definition of $ V_J$, which 
fact yields the $ \tfrac 34$.  The  next  line is a consequence of $ \frac { v ^2  } { ( v ^2 + x_2 ^2 ) ^2 }$ being decreasing in $ v$.  
The last line follows since $ u=x_1 - t_J > 5 \lvert  J\rvert > 10 \delta   $.  
\end{proof}

\subsection{Energy Inequality,  I}
Essential for the control of the off-diagonal terms is the energy inequality.  We need two definitions.  
For a interval $ I $, set the \emph{energy of $ I$} to be  
\begin{equation} \label{e:eng-sigma}
E (\sigma , I) ^2  \equiv   \sigma (I) ^{-1} \sum_{\substack{J \::\: J\subset I\\ \textup{$ J$ is good} }} \left\langle\frac t {\lvert  I\rvert }, h ^{\sigma } _{J} \right\rangle_{\sigma } ^2 .
\end{equation}
Since the Haar functions have mean value zero with respect to $\sigma$, we are free to replace $ t$ above by $ t - t_I$, so that $ E (\sigma ,I) \lesssim 1$.  

Define  $ \mathcal W\!I$ to be the partition of the interval $ I$ consisting of those maximal intervals 
$ K\Subset_r I$ with $ \textup{dist} (K, \partial I) \ge \lvert  K\rvert ^{\epsilon } \lvert  I\rvert ^{1- \epsilon }  $.  So these are the maximal intervals that are `good with respect to $ I$.'    We refer to the collection $\mathcal W\!I$ as the `Whitney' collection of $I$. These intervals need not be  good, but they do satisfy the following.

\begin{proposition}\label{p:overlap} For any interval $ I$, and any good $J\Subset_ r I $, there is a $ K\in \mathcal WI$ which contains it. Moreover, 
\begin{equation}   \label{e:whitney}
\biggl\Vert \sum_{K\in \mathcal W\!I} 2^{r}K (x)\biggr\Vert_{\infty } \lesssim 1 . 
\end{equation} 
\end{proposition}

The reason for this definition is that the maximal \emph{good} intervals contained inside of $ I$ need not be Whitney, they can for instance 
have accumulation points strictly contained in $ I$.

\begin{proof}
Any good interval $ J\Subset_ r I$ must satisfy a stronger set of conditions than those that define $ \mathcal W\!I$, so 
there must be an interval $ K\in \mathcal W\!I$ that contains it.  Concerning the second claim, 
suppose $ K_1, K_2 \in \mathcal W\!I$ with $  2 ^{\frac 2 \epsilon } \lvert  K_1\rvert <\lvert  K_2\rvert < 2 ^{-\frac {r+2} {1- \epsilon }} \lvert  I\rvert$, and $ 2 ^{r} K_1 \cap  2 ^{r}K_2\neq \emptyset $.   
Then 
\begin{align*}
\textup{dist} (\pi K_1, \partial I) &\ge  \textup{dist} (K_1, \partial I)  - \lvert  K_1\rvert  
\\&\ge \textup{dist} (K_2, \partial I) - 2 ^{r} \lvert  K_2\rvert - (2 ^{r}+1)\lvert  K_1\rvert 
\\
& \ge \lvert  K_2\rvert ^{\epsilon } \lvert  I\rvert ^{1- \epsilon  }  - 2 ^{r+1} \lvert  K_2\rvert
\\
& \ge \tfrac 12  \lvert  K_2\rvert ^{\epsilon } \lvert  I\rvert ^{1- \epsilon  } \ge    \lvert  \pi K_1\rvert ^{\epsilon } \lvert  I\rvert ^{1- \epsilon  }. 
\end{align*}
That is, $ \pi K_1$ meets the criteria for membership in $ \mathcal W\!I$, which is a contradiction to maximality.  
\end{proof}

We now state and prove the first \emph{energy inequality}. 

\begin{lemma}{\textnormal{[Energy Inequality,  I]}}
\label{l:energyI}
For all intervals $ I_0$ and partitions $ \mathcal I$ of $ I _0$ into 
 (not necessarily good) dyadic intervals,  
\begin{equation}\label{e:energyI}
\sum_{I\in \mathcal I}  \sum_{K\in\mathcal W\!I} \mathsf T_{\tau } (Q _{I_0} \setminus Q_K) (x_{Q_K} )^2 E (\sigma, K) ^2 \sigma(K) 
\lesssim \mathscr R ^2 \tau (Q_ {I_0}) . 
\end{equation}
\end{lemma}

The argument is broken into  two Lemmas, of which the main one is a version of the inequality above, but with holes in the argument of $ \mathsf T_{\tau }$.

\begin{lemma}\label{l:holesI}  For all intervals $ I_0$ and partitions $ \mathcal I$ of $ I _0$ into 
dyadic intervals, we have 
\begin{equation}\label{e:holesI}
\sum_{I\in \mathcal I}  \sum_{K\in\mathcal W\!I} \mathsf T_{\tau } (Q _{I_0} \setminus V_{2^sK}) (x_{Q_K} )^2 E (\sigma, K) ^2 \sigma(K) 
\lesssim \mathscr R ^2 \tau (Q_ {I_0}) . 
\end{equation}
Here $ V_I $ is defined in \eqref{e:Vdef}, and  $ s = r/2$.  
\end{lemma}

\begin{proof}
The sum over the left is over positive quantities, and so it suffices to consider some finite sub-sum, and establish the bound for it. 
But, with a finite sum, we can then invoke \eqref{e:monoI2},  so that 
for each term in the finite sum, 
\begin{equation*}
\mathsf T_{\tau } (Q _{I_0} \setminus V_{2^sK}) (x_{Q_K} ) ^2 E (\sigma, K) ^2 \sigma(K)  
\lesssim
\sum_{J \::\: J\subset K}  \left\langle  {\mathsf R}^{\ast}_\tau (Q _{I_0} \setminus V_{2^sK}) ,  h _{J} ^{\sigma }\right\rangle _{\sigma } ^2  . 
\end{equation*}
Below, we use the notation $ V ^{0} _{K'} := Q _{I_0} \cap V _{K'}$ for intervals $ K'\subset I_0$, in order to ease notation.   
It is sufficient to show that 
\begin{equation}\label{e:Er1}
\sum_{I\in \mathcal I} \sum_{K\in \mathcal WI} \sum_{J \::\: J\subset K}  \left\langle  {\mathsf R}^{\ast}_\tau (Q _{I_0} \setminus V_{2^sK}^0) ,  h _{J} ^{\sigma }\right\rangle _{\sigma } ^2  \lesssim \mathscr R ^2 \tau (Q _{I_0}) . 
\end{equation}
Now, turn to the assumption of testing $ {\mathsf R}^{\ast}_\tau$ on Carleson cubes.   Obviously, 
\begin{align*}
\sum_{I\in \mathcal I} 
\sum_{K\in \mathcal WI}  \bigl\Vert  {\mathsf R}^{\ast}_\tau Q _{2^sK} \bigr\Vert_{ L ^2 (K;  \sigma )} ^2   &  \lesssim 
\mathscr R ^2 \sum_{I\in \mathcal I}  \tau (Q_I) \lesssim   \mathscr R ^2 \tau (Q _{I_0}) , 
\\
\sum_{I\in \mathcal I}  \left\Vert  {\mathsf R}^{\ast}_\tau Q _{I_0} \right\Vert_{ L ^2 (I; \sigma )} ^2   & \lesssim  \mathscr R ^2 \tau (Q _{I_0}) . 
\end{align*}
The first inequality also depends upon \eqref{e:whitney}, and the fact that $ 2 ^{s}K\subset I$.  
Taking the difference, we see that 
\begin{equation}  \label{e:zri}
\sum_{I\in \mathcal I} 
\sum_{K\in \mathcal WI} 
\left\Vert  {\mathsf R}^{\ast}_\tau (Q _{I_0}- Q _ { 2^sK})\right\Vert_{ L ^2 (I, \sigma )} ^2   \lesssim  \mathscr R ^2 \tau (Q _{I_0}) . 
\end{equation}
And so, \eqref{e:Er1} follows from the estimate below,
\begin{equation}\label{e:Er2}
\sum_{I\in \mathcal I}
\sum_{K\in \mathcal WI} \sum_{J \::\: J\subset K}  
\left\langle  {\mathsf R}^{\ast}_\tau (V_{2^sK}^0-Q_ { 2^sK}) ,  h _{J} ^{\sigma }\right\rangle _{\sigma } ^2  \lesssim \mathscr R ^2 \tau (Q _{I_0}) . 
\end{equation}

For the proof of this last estimate, it is convenient to use duality.  For $ K\in\mathcal W\!I$, and functions $ \varphi _K \in L ^2(K, \sigma )$, of integral zero, it follows that $ \varphi _K$ is in the linear span of 
the Haar functions $ \{h ^{\sigma } _{J} \::\: J\subset K,\ \textup{$ J$ is good}\}$.  We should show that 
\begin{equation} \label{e:eerr}
\sum_{I\in \mathcal I}  \sum_{K\in\mathcal W\!I} 
\left\langle  {\mathsf  R } ^{\ast } _{\tau  } (V_{2^sK} ^{0}-Q_ { 2^sK}),  \varphi _K\right\rangle _{\sigma } 
\lesssim \mathscr R \tau (Q _{I_0}) ^{1/2} \left\Vert \sum_{I\in \mathcal I}  \sum_{K\in\mathcal W\!I}  \varphi _K \right\Vert_{\sigma} . 
\end{equation}

First, apply monotonicity from  \eqref{e:monoI1}. Thus, the inner product above is at most   
\begin{align}  \label{e:eM}
\mathsf T_{\tau } (V_{2^sK}^0 - Q_ { 2^sK}) (x_{Q_K} )  \int _{K} \lvert  \varphi _{K}(t)\rvert \; \sigma (dt) .
 \end{align}
Write   $ V_{2^sK} ^0-Q_ { 2^sK} $  as the disjoint union of $ V_{K} ^{\textup{top}} \cup V_{K} ^ {\textup{bottom}}$, where $ V_{K} ^{\textup{top}} \equiv \{ (x,y) \in V_{2^sK} ^0- Q_ { 2^sK} \::\: 8y \ge \lvert  { 2^sK}\rvert \}$, that is, $ V_{K} ^{\textup{top}}$ is `away' from the $ x$-axis, see Figure~\ref{f:top}.  
For the top, we have, restricting the integration to $ V ^{\textup{top}}_K$
\begin{align*}
\mathsf T_{\tau } (V ^{\textup{top}}_K) (x_{Q_K} )  
\lesssim \inf _{x \in K} \mathsf P ^{\ast} _{\tau } Q _{I_0} (x).
\end{align*}
Therefore, it follows that 
\begin{align*}
\textup{LHS} ^{\textup{top}}\eqref{e:eerr} 
& \lesssim 
\sum_{I\in \mathcal I}
\sum_{K\in \mathcal WI}  \inf _{x \in K} \mathsf P ^{\ast} _{\tau } Q _{I_0} (x) 
\int _{K} \lvert  \varphi _{K}(t)\rvert \; \sigma (dt) 
\\& 
\lesssim \sum_{I\in \mathcal I}
\sum_{K\in \mathcal WI}  \inf _{x \in K} \mathsf P ^{\ast} _{\tau } Q _{I_0} (x)   \sigma (K)  ^{1/2} 
\lVert \varphi _K\rVert_{\sigma} .
\end{align*}
An application of Cauchy--Schwarz, and using the testing conditions for the  Poisson operator will complete the proof.  

For the set $ V ^{\textup{bottom}}_K$,  we just use the $ A_2$ condition in this form: 
\begin{equation*}
\mathsf T_{\tau } V_{K} ^{\textup{bottom}} (x_{Q_K} ) \frac {\sigma (K)} {\lvert  K\rvert } \lesssim \mathscr A_2.  
\end{equation*}
Thus, 
\begin{align*}
\textup{LHS} ^{\textup{bottom}}\eqref{e:eerr} 
& \lesssim 
\sum_{I\in \mathcal I}
\sum_{K\in \mathcal WI}   
\mathsf T_{\tau } V_{K} ^{\textup{bottom}} (x_{Q_K} ) \int _{K} \lvert  \varphi _{K}(t)\rvert \; \sigma (dt)  
\\& \lesssim  
\sum_{I\in \mathcal I}
\sum_{K\in \mathcal WI}   
\mathsf T_{\tau } V_{K} ^{\textup{bottom}} (x_{Q_K} ) \sqrt {\sigma (K)}  \lVert \varphi _{K}\rVert_{\sigma} 
\\
& \lesssim
\mathscr R 
\sum_{I\in \mathcal I}
\sum_{K\in \mathcal WI}    \sqrt { \lvert  K\rvert  \cdot \mathsf T_{\tau } V_{K} ^{\textup{bottom}} (x_{Q_K} ) } \lVert \varphi _{K}\rVert_{\sigma}  
\\& \lesssim \mathscr R \tau  (Q _{I_0}) ^{1/2} 
 \left\Vert \sum_{I\in\mathcal{I}}\sum_{K\in\mathcal W\!I} \varphi _K \right\Vert_{\sigma } . 
 \end{align*}
 Use the bounded overlap property of the $ V_{2^sK} ^{\textup{bottom}}$  proved in Proposition~\ref{p:V2} to get the last estimate. 
Our proof is complete. 
 
\end{proof}

\begin{figure}
\begin{tikzpicture}
\draw[<->] (-3,0) -- (5.5,0);
\begin{scope}[xshift=-1cm] 
\draw (0,0) rectangle (2,2);  \draw[|-|,very thick] (0,0) -- (2,0); 
\draw[->] (0,0) -- (-2,2);   \draw[->] (2,0) -- (4,2); 
\fill[fill=black, opacity=.4] (0,0) -- (-.75 ,.75) -- (0,.75) -- (0,0);
\fill[fill=black, opacity=.4]  (2.75,.75) -- (2,0) -- (2,.75) --  (2.75,.75);
\end{scope}

\begin{scope}[xshift=3.5cm] 
\draw (0,0) rectangle (1,1);  \draw[|-|, very thick] (0,0) -- (1,0); 
\draw[->] (0,0) -- (-1,1);   \draw[->] (1,0) -- (2,1); 
\fill[fill=black, opacity=.4] (0,0) -- (-.5 ,.5) -- (0,.5) -- (0,0);
\fill[fill=black, opacity=.4]   (1.5,.5) -- (1,.5) -- (1,0) --  (1.5,.5) ;
\end{scope}
\end{tikzpicture}

\caption{Two sets $ V_I ^{\textup{bottom}}$ are indicated in light gray.}
\label{f:top} 
\end{figure}
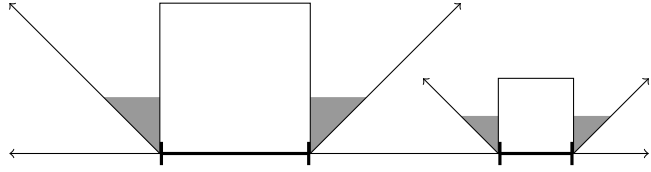

\begin{proposition}\label{p:V2} Let $ \mathcal I$ be a partition of interval $ I_0$ into dyadic intervals.  We have 
\begin{equation*}
\sum_{I\in \mathcal I} V ^{\textup{bottom}}_I (x,y) \le 2 . 
\end{equation*}
\end{proposition}

\begin{proof}
See Figure~\ref{f:top} for an illustration.  
The collection $ \mathcal I$ is a partition of dyadic intervals.  Assume that it is a sub-partition, with no two 
dyadic intervals sharing a common endpoint.  Then, the sets $ V_I ^{\textup{bottom}}$ are pairwise disjoint, as we argue. 
Suppose that $ \lvert  I_1\rvert \ge \lvert  I_2\rvert  $, and $ I_2$ lies to the right of $ I_1$. The dyadic 
property implies that $ \textup{dist} (I_1, I_2) \ge \lvert  I_2\rvert $. By translation  and dilation invariance, we can assume that the 
right hand endpoint of $ I_1$ is the origin, and that the left hand endpoint of $ I_2$ is equal to one, which is the length of $ I_2$. 
Then, the right boundary of $ V_{I_1} ^{\textup{bottom}}$ and the left boundary of $ V_{I_2} ^{\textup{bottom}}$ are the two lines 
\begin{equation*}
 \left\{ (t,t/2) \::\: 0\le t \le \tfrac {\lvert  I_1\rvert } {16}\right\}, \quad 
 \textup{and} \quad 
 \left\{ (1-s,s/2) \::\: 0\le s \le \tfrac 1 {16}\right\}. 
\end{equation*}
These two line segments do not  intersect, so the proposition is proved. 
\end{proof}

\begin{lemma}\label{l:hI}  For all  intervals $ I_0$, and partition $ \mathcal I$ of $ I _0$ into dyadic intervals, 
\begin{equation}\label{e:holes-V}
\sum_{I\in \mathcal I}  \sum_{K\in\mathcal W\!I} \mathsf T_{\tau } (V_{2^sK} \setminus Q_K) (x_{Q_K} ) ^2 E (\sigma, K) ^2 \sigma(K) 
\lesssim \mathscr R ^2 \tau (Q_ {I_0}) . 
\end{equation}

\end{lemma}

\begin{proof}
Observe that  the estimate below 
\begin{equation*}
\sum_{I\in \mathcal I}  \sum_{K\in\mathcal W\!I}\mathsf T_{\tau } (V_{2^sK} - Q _{2^sK}) (x_{Q_K} ) ^2 E (\sigma, K) ^2 \sigma(K) 
\lesssim \mathscr R ^2 \tau (Q_ {I_0})
\end{equation*}
follows from the argument above, beginning at \eqref{e:eerr}, since our first step was to apply monotonicity in the form of \eqref{e:eM}. Now, using the $ \mathscr A_2$ condition, and the bounded overlap property of Proposition~\ref{p:overlap},  one easily sees that 
\begin{equation*}
\sum_{I\in \mathcal I}  \sum_{K\in\mathcal W\!I}\mathsf T_{\tau } ( Q _{2^sK} \setminus Q_K)  (x_{Q_K} ) ^2  \sigma(K) 
\lesssim \mathscr R ^2 \tau (Q_ {I_0}). 
\end{equation*}
That completes the proof. 
\end{proof}

\begin{proof}[Proof of \eqref{e:energyI}]  
This is an immediate combination of Lemma~\ref{l:holesI} and \ref{l:hI}. 
\end{proof}

\subsection{Monotonicity, II}

Below, we will phrase the monotonicity estimate in terms of the $ L ^2 _0 (Q _{J}; \tau )$ norm, which is the norm 
for the subspace of $ L ^2 (Q_J; \tau )$ which is orthogonal to constants.   One should note that we could have used 
this type of definition for $ E (\sigma , I)$ in \eqref{e:eng-sigma}.  But, also note that the $  L ^2 _0 (Q _{J}; \tau )$ norm 
equals 
\begin{equation} \label{e:==}
2 \lVert g \rVert_{ L ^2 _0 (Q _{J}; \tau )} ^2 = 
\mathbb E ^{\tau } _{Q_J} \int _{Q} \lvert   g (x) - g (x')\rvert ^2 \; \tau(dx) .  
\end{equation}

\begin{lemma}{\textnormal{[Monotonicity Property, II]}}
\label{l:monoIIx} 
For an absolute constant $A$, this holds.  
Let $ I$ be an interval, and suppose that $ f \in L ^2 (\mathbb{R};\sigma )$ is not supported on $  I$.  Then, for intervals $   A\cdot J\subset I$, 
\begin{equation}  \label{e:monoII<}
 \lVert \mathsf R _{ \sigma  } f\rVert_{ L ^2 _0 (Q_J, \tau )} 
\lesssim \mathsf P _{\sigma } (\lvert  f\rvert, J) 
\left\lVert \frac {x} {\lvert  J\rvert } \right\rVert_{ L ^2 _0 (Q_J; \tau )}.
\end{equation}
Moreover, if $ f \ge 0$, 
\begin{equation}\label{e:monoII>}
 \mathsf P _{\sigma } (  \lvert  f\rvert ,   J) 
\left\lVert \frac {x} {\lvert  J\rvert } \right\rVert_{ L ^2 _0 (Q_J;\tau )}
\lesssim 
\left\lVert \mathsf R _ \sigma f\right\rVert_{ L ^2  (Q_J ;\tau )} .
\end{equation}

\end{lemma}

\begin{proof}
The canonical value of the inner products \eqref{e:canonical*} is used.  
The first inequality \eqref{e:monoII<} is simple. For coordinates $ j=1,2$, the function $ \mathsf R _ \sigma  ^{j} f $ is $ C ^2 $ 
and real-valued on $ Q_J$.  
It follows from \eqref{e:nabla}, and the mean value theorem that for any $ x, y = (y_1, y_2)\in Q_J$, there is an $ z = (z_1, z_2)$, on the line between $ x$ and $  y$ 
so that 
\begin{equation} \label{e:RR}
\mathsf R _{\sigma }  ^{j} f ( x)
-\mathsf R _ \sigma   ^{j} f (  y ) 
= (x- y ) \cdot \nabla \mathsf R _{\sigma } ^{j}  (z).   
\end{equation}
By inspection, 
\begin{equation*}
\left\lvert  \nabla \mathsf R _ \sigma   ^{j} f (z)\right\rvert \lesssim 
\frac {1} {\lvert  J\rvert } \cdot P _{\sigma } (\lvert  f\rvert, J ).  
\end{equation*}
So \eqref{e:monoII<} follows from \eqref{e:==}.  

\smallskip 
For the reverse inequality \eqref{e:monoII>}, we treat two cases separately.  
Assume first that 
\begin{equation}\label{e:qhalf}
\lVert x_1\rVert _{L ^2 _0 (Q_J; \tau )} \geq \tfrac 12 \lVert x \rVert _{L ^2 _0 (Q_J; \tau )}.  
\end{equation}
Then, we will show that for $ x= (x_1, x_2), x' = (x_1 ', x_2') \in Q _{J}$
\begin{align}
P _{\sigma } (f, J)  \Bigl\lVert  \frac {x_1} {\lvert  J\rvert } \Bigr\rVert _{L ^2 _0 (Q_J ; \tau )} 
& = \sqrt 2  
P _{\sigma } (f, J)  \Bigl\lVert  \frac {x_1 '-x_1} {\lvert  J\rvert } \Bigr\rVert _{L ^2 _0 (Q_J \times Q_J ; \tau \times \tau )} 
\\  \label{e:qgrad}
& \lesssim 
\lVert  R _{\sigma } ^1 f (x') - R _{\sigma } ^1 f (x)\rVert _{L ^2 _0 (Q_J \times Q_J ; \tau \times \tau )} 
 \leq 2 
 \lVert  R _{\sigma } ^1 f \rVert _{L ^2 _0 (Q_J   ;  \tau )} .  
\end{align}
This completes the proof of \eqref{e:monoII>} subject to \eqref{e:qhalf} being true.

It remains to prove \eqref{e:qgrad}. By the mean value theorem, we have 
for some  $ z = z(x, x')$  between $ x$ and $ x'$
\begin{align}\label{e:qqgrad}
 R _{\sigma } ^1 f (x_1', x_2') - R _{\sigma } ^1 f (x_1, x_2) 
& =  
 \begin{bmatrix}
x_1' - x_1 \\  x_2 ' - x_2 
\end{bmatrix} \cdot \nabla R _{\sigma } ^2 f (z), 
\\
& =
\begin{bmatrix}
x_1' - x_1 \\  x_2 ' - x_2 
\end{bmatrix} \cdot
\int _{\mathbb R \setminus I} 
\begin{bmatrix}
   \frac {- (z_1-t) ^2 + z_2 ^2 } { [(z_1-t) ^2 + z_2 ^2 ] ^2 } \\
  -  2\frac { z_2  (z_1-t)} { [(z_1-t) ^2 + z_2 ^2 ] ^2 } 
\end{bmatrix} 
f (t) \; \sigma (dt)
\\&= 
(x_1' - x_1)  E_1  f (z) 
+
(x_2' - x_2)  E_2 f(z) . 
\end{align}
For the second term, estimate 
\begin{align*}
\lVert (x_2' - x_2)  E_2 f(z) \rVert _{L ^2  (Q_J \times Q_J ; \tau \times \tau )} 
& \leq \lVert E_2 f(z) \rVert_{L ^\infty   (Q_J \times Q_J )} \lVert x_2' - x_2 \rVert _{L ^2  (Q_J \times Q_J ; \tau \times \tau )} 
\\
& \leq \Bigl\lVert \frac {x_2} {\lvert J\rvert} \Bigr\rVert _{L ^2_0  (Q_J;\tau)} 
{\lvert J\rvert} \cdot   \lVert E_2 f(z) \rVert_ {L ^\infty   (Q_J \times Q_J )}
\\&\lesssim A^{-1} 
\Bigl\lVert \frac {x_2} {\lvert J\rvert} \Bigr\rVert _{L_0 ^2  (Q_J;\tau)}  \mathsf{P}_{\sigma} (f,J). 
\end{align*}
The last inequality follows by inspection.  
The leading term of $A^{-1}$ follows from the assumption that $A\cdot J\subset I$, 
which implies that $ z_2 ^2 \leq \lvert J\rvert ^2 
\leq A_2^{-2} (z_1-t)^2$ in the integral defining $E_2f$.  This is  a small estimate.  

For the term $E_1f $, let $y$ be the center of $Q_J$, and write 
\begin{align*}
E_1f = \mathsf{P}_{\sigma} f (y) + E_3 f (z) 
\end{align*}
where, again by inspection, we will have 
\begin{equation}
\lVert (x_1' - x_1)  E_3 f(z) \rVert _{L ^2  (Q_J \times Q_J ; \tau \times \tau )} 
\lesssim A^{-1} 
\Bigl\lVert \frac {x_2} {\lvert J\rvert} \Bigr\rVert _{L_0 ^2  (Q_J;\tau)}  \mathsf{P}_{\sigma} (f,J). 
\end{equation}
On the other hand, the term $ (x_1' - x_1) \mathsf{P}_\sigma f (y)$ is the main term that we want. 
That is, we have 
\begin{equation}
\Bigl\lVert 
(x_1' - x_1) \mathsf{P}_\sigma f (y) 
\Bigr\rVert  _{L ^2  (Q_J \times Q_J ; \tau \times \tau )} 
\leq 
\lVert  R _{\sigma } ^2 f \rVert _{L ^2 (Q_J   ;  \tau )} -C A ^{-1} 
\Bigl\lVert \frac {x_1} {\lvert J\rvert} \Bigr\rVert _{L_0 ^2  (Q_J;\tau)}  \mathsf{P}_{\sigma} (f,J). 
\end{equation}
For a large enough constant $A$, we have completed the proof of \eqref{e:qgrad}. 
\smallskip 

If \eqref{e:qhalf} does not hold, we necessarily have 
\begin{equation*}
\lVert x_2\rVert _{L ^2 _0 (Q_J; \tau )} \geq \tfrac 12 \lVert x \rVert _{L ^2 _0 (Q_J; \tau )}.  
\end{equation*}
But, then there is no cancellation needed, as we can compare directly to the second coordinate of the Riesz transform, which is the Poisson integral. We have 
\begin{align}
 \mathsf P _{\sigma } (  \lvert  f\rvert ,   J) 
\left\lVert \frac {x_2} {\lvert  J\rvert } \right\rVert_{ L ^2 _0 (Q_J;\tau )}
&\leq 
 \mathsf P _{\sigma } (  \lvert  f\rvert ,   J) 
\left\lVert \frac {x} {\lvert  J\rvert } \right\rVert_{ L ^2 (Q_J;\tau )}
\\
&\lesssim 
\lVert R^2_ \sigma f \rVert _{ L ^2 (Q_J;\tau )} . 
\end{align}

\end{proof}

\subsection{Energy Inequality, II}

We focus on the energy inequality in the dual setting. 
For an interval $ I$, we define the energy in a different, but equivalent, way than before, 
\begin{equation}  \label{e:Etau} 
E (\tau ,I) ^2 \equiv \tau (Q_I) ^{-1}  \left\lVert   \frac { x} {\lvert  I\rvert } \right\rVert_{ L ^2 _{0} (Q_I, \tau )} ^2 . 
\end{equation}
Keep in mind that $ x\in \mathbb R ^2 _+$.  Here $ L ^2  _{0} (Q_I, \tau )$ denotes the norm of the function, less its mean.  

\begin{lemma}{\textnormal{[Energy Inequality, II]}}
\label{l:energyII}
For any interval $ I_0$ and partition $ \mathcal P$ of $ I_0$ into dyadic intervals, 
\begin{equation}\label{e:energyII} 
\sum_{I\in \mathcal P}
\sum_{K\in \mathcal W\!I}
\mathsf P _{\sigma }(I_0  \setminus K, K) ^2 E (\tau ,K) ^2 \tau (Q_K) \lesssim \mathscr R ^2 \sigma  (Q_ {I_0}) . 
\end{equation}
\end{lemma}

\begin{proof}
We can assume that $2^r >A$, where $A$ is the constant of Lemma \ref{l:monoIIx}.  
Using the $ A_2$ inequality, we can enlarge the holes, namely, 
\begin{align*}
\sum_{I\in \mathcal P} 
\sum_{K\in \mathcal W\!I} 
\mathsf P (\sigma  \cdot 2^r K \setminus K  , K) ^2  \tau (Q_K) &\lesssim  \mathscr A _2  
\sum_{I\in \mathcal P} \sum_{K\in \mathcal W\!I}  \lvert  K\rvert \cdot  
\mathsf P (\sigma  \cdot 2^r K ,K) 
\\& \lesssim \mathscr A_2 \sum_{I\in \mathcal P} \sum_{K\in \mathcal W\!I}  \sigma ({2 ^{r}K}) \lesssim \mathscr A _2  \sigma  ({I_0}) . 
\end{align*}
Note that this depends critically on the bounded overlap property \eqref{e:whitney}.  

It remains to consider the sum with the Poisson term being $ \mathsf P (\sigma \cdot (I_0 -2 ^r K), K)$.  
It suffices to prove the estimate with $ \mathcal P$ a finite sub-partition of $I_0 $, and the assumption that 
each $ \mathcal WI$ is also finite. The constant will be independent of this assumption. 
The  monotonicity property \eqref{e:monoII>} applies, so that it suffices to estimate 
\begin{align*}
\sum_{I\in \mathcal P} \sum_{K\in \mathcal W\!I} &  
 \lVert \mathsf R _ \sigma  (I_0- 2 ^{r} K))\rVert_{ L ^2 (Q_K; \tau )}^2 
 \\
 & \lesssim 
 \sum_{I\in \mathcal P} \sum_{K\in \mathcal W\!I} 
 \lVert \mathsf R _ \sigma  I_0 \rVert_{ L ^2 (Q_K; \tau )}^2 
 +  \lVert \mathsf R _ \sigma  (2 ^{r} K)\rVert_{ L ^2 (Q_K; \tau )}^2 
 \\
& \lesssim 
 \lVert \mathsf R _\sigma  I_0\rVert_{ L ^2 (Q_ {I_0}; \tau )}^2  
 +\sum_{I\in \mathcal P} \sum_{K\in \mathcal W\!I} 
 \lVert \mathsf R _  \sigma (2 ^{r} K)\rVert_{ L ^2 (Q_K; \tau )}^2 . 
 \end{align*}
And these two terms are controlled by the testing inequalities and \eqref{e:whitney}.

\end{proof}

\section{Global to Local Reduction} \label{s:global}

\subsection{Initial Reductions}

We can assume that $ f$ is supported on a (large) interval $ I^0$, and $ g$ is supported on $ Q _{I^0}$.  
By trivial application of the testing inequalities, we can further assume that $ f$ has $ \sigma $-integral zero, 
and $ g$ has $ \tau $-integral zero.  Thus, $ f, g$ are in the span of good adapted Haar functions.  
And we can assume that $ \lvert  I^0\rvert\ge 2 ^{r} \lvert  J\rvert  $ for all $ J$ in the Haar support of $ f$, 
and similarly $ \lvert  I ^{0}\rvert\ge 2 ^{r} \lvert  Q\rvert ^{1/2}   $ for all cubes $ Q$ in the Haar support of $ g$. 
	
Further restrictions on the Haar supports of $ f $ and $ g$ are made, these restrictions are phrased in terms of $ r$ and $ \epsilon $. 
The values needed for $ r$ and $ \epsilon $ are derived from the elementary estimates. 
For an integer $ 0\le s_f < r $ (which plays no further role in the argument),  
assume that 
\begin{equation} \label{e:HS} 
f = \sum_{ \substack{ I\in \mathcal D \::\: I\subset I^0 \\  \log_2 \lvert  I\rvert \in  r  \mathbb Z +s_f +1  }}
\Delta ^{\sigma } _{I} f \,,  
\end{equation}
and let $ \mathcal D _{f}^r \equiv \{ I\in \mathcal D \::\: I\subset I^0\,\  \log_2 \lvert  I\rvert \in r \mathbb Z +s_f \}$.  
Thus, $ \mathcal D _{f}^r$ is a grid containing all the children of intervals  in the Haar support of $ f$.  
Likewise, 
for an integer $ 0\le s_g < r $,   assume that 
\begin{equation*}
g= \sum_{ \substack{ Q\in \mathcal D^2 \::\: Q\subset Q_{I^0} \\ \log_2 \lvert  Q\rvert \in r  \mathbb Z +s_g +1  }}
\Delta ^{\tau  } _{Q} g \,,  
\end{equation*}
and let $ \mathcal D _{g}^r \equiv \{ I\in \mathcal D \::\: I\subset I^0\,\  \log_2 \lvert  I\rvert \in r \mathbb Z +s_g\}$.  Note that this is the projection of of the associated squares in the upper half plane to the real line.    
Our specificity about the martingale difference support for $ f$ and $ g$ has the purpose of easing the technical 
burdens at different points in the proof below.  
Consequently, we will reference the grids $ \mathcal D _{f}^r$ and $ \mathcal D _{g}^r$ when appropriate.

Let $ P _{\textup{Car}} ^{\tau } g \equiv \sum_{I \in \mathcal D} \Delta ^{\tau } _{Q_I} g $, where  the sum is only over Carleson cubes. 
It suffices to consider (good) functions in the range of this projection.  This proposition is proved in \S\ref{s:elementary}.

\begin{proposition}\label{p:carlesonCubes}  The following estimate holds: 
\begin{equation}\label{e:carlesonCubes}
\left\Vert {\mathsf R}^{\ast}_\tau (g - P ^{\tau }  _{\textup{Car}} g)\right\Vert_{\sigma } \lesssim \mathscr T \lVert g\rVert_{\tau } . 
\end{equation}
\end{proposition}

Thus, we assume throughout that $ g$ is a good function, with $ P _{\textup{Car}} ^{\tau } g =g$, as well as satisfying the further restriction on the 
Haar supports described above. 
We define the two triangular forms 
\begin{align}\label{e:above}
B ^{\textup{above}} (f,g) &\equiv \sum_{ I, J \::\: J \Subset_ {4r} I} 
\mathbb E ^{\tau } _{Q _{IJ}} \Delta ^{\tau } _{Q_I} g \cdot \left\langle {\mathsf R}^{\ast}_\tau Q _{IJ}  , \Delta ^{\sigma } _{J} f \right\rangle _{\sigma } ,  
\\ \label{e:below}
B ^{\textup{below}} (f,g) &\equiv \sum_{ I, J \::\: I\Subset_ {4r} J} 
\mathbb E ^{\sigma } _{J_{I}} \Delta ^{\sigma } _{J} f \cdot \left\langle \Delta ^{\tau} _{Q_I} g ,  {\mathsf R}_ \sigma   J_I  \right\rangle _{\tau  } , 
\end{align}
where $ J_I$ is the child of $ J$ that contains $ I$, and $ Q _{IJ}$ is the child of $ Q_I$ that contains $ Q_J$.  See Figure \ref{fig:QIJ}.
These two forms are dual to one another, but their analysis is different, due to the assumptions on the supports of $ \sigma $ and $ \tau $.

\begin{figure}
\begin{tikzpicture}
\draw[<->] (0,0) -- (6,0);
\draw[|-|,thick]  (4,0) -- (1,0) node[below,midway] {$ I$}; 
\draw  (1,0) rectangle (4,3) node[right] {$ Q_I$}; 
\draw  (3,0) rectangle (3.5,.5) node[above] {$ Q_J$}; 
\draw (2.5,0) rectangle (4,1.5) node[right] {$Q _{IJ} $}; 
\end{tikzpicture}
\caption{The sets $Q_I$, $Q_J$ and $Q_{IJ}$.}
\label{fig:QIJ}
\end{figure}
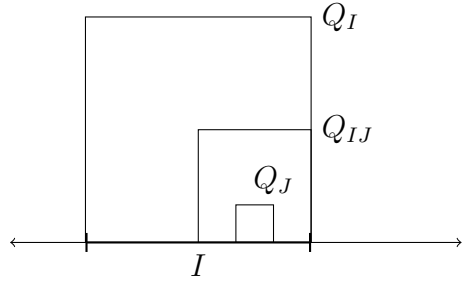

\begin{proposition}\label{p:triangles} The following estimate is true: 
\begin{equation}\label{e:triangles}
\left\vert 
\left\langle {\mathsf R}_ \tau^*  g, f\right\rangle _{\sigma} - B ^{\textup{above}} (f,g)  - B ^{\textup{below}} (f,g) 
\right\vert \lesssim \mathscr R  \lVert f\rVert_{\tau }\lVert g\rVert_{\sigma }.
\end{equation}
\end{proposition}
The proof of Proposition \ref{p:triangles} appears in \S\ref{s:elementary}. 
We concentrate on the  `above' form in the remainder of this section. 

\subsection{The Stopping Data}\label{s:stopData}

The function $ g$ is in the linear span of the martingale differences associated with good Carleson cubes and is supported on the cube $Q _{I^0} $. 
Construct stopping intervals for $ g$, which is  a collection of dyadic intervals $ \mathcal F$. 
Initialize $ \mathcal F $ to be the maximal elements of $ \mathcal D _g ^{r}$ contained in $ I^0$. 
In the inductive stage, if $ F\in \mathcal F$ is minimal, add to $ \mathcal F$ the maximal standard dyadic children $ I \in \mathcal D _{g} ^{r}$ of $ F$ that meet either of these conditions: 
\begin{enumerate}
\item (A large average)  $ \mathbb E _{Q _{I}} ^{\tau } \lvert  g\rvert \ge 10 \mathbb E _{Q _{F}} ^{\tau } \lvert  g\rvert $; 
\item (Energy Stopping) 
$ \sum_{K\in \mathcal W\!I}\mathsf T_{\tau } (Q _{F} \setminus Q_K) (x_{Q_K}) ^2 E (\sigma ,K) ^2  \sigma (K) \ge C_0 \mathscr R ^2 \tau  (Q_I)$.  
\end{enumerate}
 
The second condition arises from the Energy Inequality Lemma~\ref{l:energyI}.  
It  implies  that $ \mathcal F$ satisfies a $ \tau $-Carleson condition for a sufficiently large constant $ C_0$ for energy stopping.  Namely,
\begin{equation} \label{e:tauCarleson}
\sum_{F' \in \mathcal F \::\: F'\subsetneq F} \tau (Q_{F'}) \leq \tfrac 12  \tau (Q_F), \qquad F\in \mathcal F.  
\end{equation} 
It is also immediate from construction that the stopping intervals 
control the averages of $ g$ in the following sense:  For all intervals $ I\in \mathcal D _{g} ^{r}$, $ I\subset I_0$, we have 
\begin{equation} \label{e:avg<}  
\lvert \mathbb E ^{\tau } _{Q_I}  g\rvert 
= \left\lvert 
\sum_{K \::\: K\supsetneq I}  \mathbb E ^{\tau } _{Q_I}  \Delta ^{\tau } _{Q_K}g
\right\rvert
\lesssim
 \mathbb E _{Q_{F} }^{\tau }  \lvert  g\rvert , \qquad \pi _{\mathcal F} I=F. 
\end{equation}
(The notation $ \pi _{\mathcal F}I$, and several more, are  defined at the beginning of \S\ref{s:necessary}.)

We make this brief remark about the collections $ \mathcal F $ and $ \{ \mathcal W\!F \::\: F\in \mathcal F\}$. 
For each $ F\in \mathcal F$, and good $ J\Subset _{r} F$, there is a $ K\in \mathcal W\!F$ with $J\subset K$.   
For intervals $ F\in \mathcal F$, define Haar projections by 
\begin{align*}
H _F^{\tau } g &\equiv   \sum _{I \::\: \pi _{\mathcal F} I=F} \Delta ^{\tau } _{Q_I} g  , 
\\
\tilde H _F^{\sigma } f &\equiv \sum _{\substack{J \::\:  \tilde \pi _{\mathcal F} J=F}} \Delta ^{\sigma  } _{J} f.   
\end{align*}
In the second line, we take $ \tilde \pi _{\mathcal F} J$ to be the smallest member $F$  of $ \mathcal F$ so that $ J\Subset _{4r}  F$.  We call this inequality the \emph{quasi-orthogonality} bound; 
it  is basic to the proof: 
\begin{equation}\label{e:quasi}
\sum_{F\in \mathcal F} 
\left\{ \mathbb E _{Q _{F}} ^{\tau } \lvert  g\rvert  \cdot \tau (Q_F) ^{1/2}  + \lVert H _F^{\tau } g\rVert_{\tau} \right\} \lVert \tilde H_{F}^{\sigma } f\rVert_{\sigma}  \lesssim \lVert f\rVert_{\tau} \lVert g\rVert_{\sigma} . 
\end{equation}
This follows from  the  $ \tau $-Carleson property of $ \mathcal F$  and the quasi-orthogonality of the Haar projections.  
It will appear below with different choices of these orthogonal projections.  

Observe that we have 
\begin{align} \label{e:AFF}
\begin{split}
B ^{\textup{above}} (f,g) &=    \sum_{F\in \mathcal F} \sum_{F' \::\: F'\supset F}
B ^{\textup{above}} (\tilde H ^{\sigma } _{F}f, H ^{\tau } _{F'}g) .
\end{split}
\end{align}
Indeed, the definition of $ B ^{\textup{above}} $ is over a sum of pairs of good intervals $ J, I$ with 
 $ J\Subset _{4r} I$ and say $ \tilde \pi _{\mathcal F} J=F$.  
We necessarily have $ J\Subset _{4r} \pi (\pi _{\mathcal F}I_J)$, hence $ F\subset \pi _{\mathcal F}I$.  
Then it is clear that this  pair of intervals $ (J,I)$ appear exactly once on the right  \eqref{e:AFF}.  

We can now turn to the global to local reduction for the form $ B ^{\textup{above}}$.  
In the sum below we are taking that part of  the right side of \eqref{e:AFF} which is `separated'  by $ \mathcal F$.  
Namely, in the lemma below, we form the sum only over  pairs of intervals $ I,J$ that  do not have the same $ \mathcal F$-parent. 

\begin{lemma}{\textnormal{[Global to Local Reduction, I]}}
\label{l:global2localAbove} The following estimate holds: 
\begin{align} \label{e:Cglobal}
\Biggl\vert \sum_{I} \sum_{\substack{J \::\:  \tilde \pi _{\mathcal F} J \subsetneq I \\ J\Subset _{4r} I}} \mathbb E ^{\tau } _{Q_{IJ}} \Delta ^{\tau } _{Q_I} g \cdot 
\left\langle  {\mathsf R}^{\ast}_\tau Q _{IJ},    \Delta ^{\sigma } _{J} f \right\rangle _{\sigma}   \Biggr\vert 
\lesssim \mathscr R  \lVert g\rVert_{\tau} \lVert f\rVert_{\sigma} . 
\end{align}
\end{lemma}

\begin{proof}
We invoke, for the first time, the \emph{exchange argument}, namely exchanging the inequality concerning a singular integral for one involving  a purely positive operator.  
This entails (a) controlling the sums  
of martingale differences by the stopping values; 
(b) replacing the argument of the singular integral by a stopping interval;  
(c) appealing to interval, or Carleson cube, testing and quasi-orthogonality to complete the bound in this case; 
(d) for the complementary argument in the singular integral, appeal to monotonicity, to get a positive operator; 
(e) appeal directly to a so-called \emph{parallel corona} argument to prove the required inequality.  

The details in this case are as follows. We are to bound the sum 
\begin{equation} \label{e:caseA} 
\sum_{F\in \mathcal F}  \; 
\sum_{\substack{I \::\: F\subsetneq I\\}}  \; 
\sum_{\substack{J \::\:  \tilde \pi _{\mathcal F} J=F}}  
\mathbb E ^{\tau } _{Q_{IJ}} \Delta ^{\tau } _{Q_I} g
\cdot 
\left\langle  {\mathsf R}^{\ast}_\tau Q _{IJ},    \Delta ^{\sigma } _{J} f \right\rangle _{\sigma} .
\end{equation}
Write the argument of the  Riesz transform as $ Q _{IJ} = Q _{IF} = Q_F + (Q _{IF} - Q_F)$.  
In the case that the argument is $ Q _{F}$, observe that by construction of the stopping data, that 
\begin{equation*}
\left\vert   \sum_{I \::\: F\subsetneq I} 
\mathbb E ^{\tau } _{Q_{IJ}} \Delta ^{\tau } _{Q_I} g \right\vert \lesssim \mathbb E^{\tau}_{Q_F}\lvert g\rvert.  
\end{equation*}
Therefore, we can estimate using the testing inequality for the  Riesz transform: 
\begin{align*}
\Biggl\lvert 
\sum_{I \::\: F\subsetneq I}
\sum_{\substack{J \::\:  \tilde \pi _{\mathcal F} J=F}}   
\mathbb E ^{\tau } _{Q_{IJ}} \Delta ^{\tau } _{Q_I} g
\cdot & 
\left\langle  {\mathsf R}^{\ast}_\tau Q _{F},    \Delta ^{\sigma } _{J} f \right\rangle _{\sigma} 
\Biggr\rvert
\\& \lesssim 
\mathbb E^{\tau}_{Q_F}\lvert g\rvert  \left\vert
\sum_{\substack{J \::\:  \tilde \pi _{\mathcal F} J=F}}  
\left\langle  {\mathsf R}^{\ast}_\tau Q _{F},    \Delta ^{\sigma } _{J} f  \right\rangle _{\sigma } \right\vert
\\
& \lesssim \mathscr R \mathbb E^{\tau}_{Q_F}\lvert g\rvert \cdot\tau (Q_F) ^{1/2} 
\lVert \tilde H ^{\sigma } _{F} f   \rVert_{\sigma } ^2 . 
\end{align*}
The sum over $ F\in \mathcal F$ of this last expression is controlled by quasi-orthogonality, \eqref{e:quasi}.

When the argument of the Riesz transform is $ Q _{IF} - Q_F$, the analysis proceeds in a different  and more involved manner and so it is proved in the next Lemma.  

\end{proof}

\begin{lemma}\label{l:xx}  There holds 
\begin{align}\label{e:xx}
\Biggl\lvert \sum_{F\in \mathcal F}
\sum_{I \::\: F\subsetneq I}
\sum_{\substack{J \::\:  \tilde \pi _{\mathcal F} J=F}}   
\mathbb E ^{\tau } _{Q_{IJ}} \Delta ^{\tau } _{Q_I} g
\cdot & 
\left\langle  {\mathsf R}^{\ast}_\tau (Q _{IF} \setminus Q_F) ,    \Delta ^{\sigma } _{J} f \right\rangle _{\sigma} 
\Biggr\rvert
 \lesssim \mathscr R \lVert g\rVert_ \tau \lVert f\rVert _{\sigma } . 
\end{align}

\end{lemma}

\begin{proof}
The first stage of the proof is to pass to a new two weight inequality 
from which the estimate above follows.  
The second stage is to prove a new two weight inequality, which itself requires a delicate analysis.

Note that, again by the construction of stopping data, that 
\begin{equation*}
\left\vert 
\sum_{I \::\: F\subsetneq I}  \mathbb E ^{\tau } _{Q_{IJ}} \Delta ^{\tau } _{Q_I} g \cdot (Q _{IF} - Q_F)
\right\vert \lesssim  \sum_{F' \in \mathcal F \;:\; F'\supsetneq F} \mathbb E ^{\tau } _{Q _{F'}} \lvert  g\rvert \cdot 
\widehat Q _{F',F},  
\end{equation*}
where by definition, $ \widehat Q _{F',F} = Q _{F'} \setminus Q _{F''}$, where 
$ F''$ is the $  \mathcal F$-child of $ F'$ that contains $ F$. 
From the monotonicity principle, \eqref{e:MONO}, we see that  for fixed $ F\in \mathcal F$, 
\begin{align} 
\Biggl\lvert 
\sum_{I \::\: F\subsetneq I} 
\sum_{\substack{J \::\:  \widetilde \pi _{\mathcal F} J=F}}  
\mathbb E ^{\tau } _{Q_{IJ}} & \Delta ^{\tau } _{Q_I}  g
\cdot   
\langle  {\mathsf R}^{\ast}_\tau (Q _{IF} - Q_F),    \Delta ^{\sigma } _{J} f \rangle _{\sigma} 
\Biggr\rvert 
\\& \lesssim   \label{e:3xx} 
\sum_{\substack{F'\in \mathcal F \;:\; F'\supsetneq F } } 
 \mathbb E ^{\tau } _{Q _{F'}} \lvert  g\rvert  \sum_{K\in \mathcal W\!F} 
\widehat{ \mathsf T}_{\tau } (\widehat Q _{F',F} )( x _{Q_K} ) 
\sum_{\substack{J \::\:  J\subset K \\\widetilde \pi _{\mathcal F} J = F }}   
\langle  t  ,  h ^{\sigma } _{J}\rangle _{\sigma } 
\lvert  \hat f_{\sigma} (J) \rvert  , 
 \\\label{e:Ttilde}
 \textup{where} &\quad 
 \widehat {\mathsf T }_{\tau } g (x_1, x_2)  \equiv \int _{\mathbb R ^2 _+}  \frac { g (y)  } { y^2_2 +{ x_2^2 } + \lvert y_1 - x_1\rvert ^2 } \; \tau (dy)   .  
\end{align}
The operator $ \widehat {\mathsf T }_{\tau } g$ is a Poisson average, but missing the appropriate 
scaling for an average, as compared to \eqref{e:Ttdef}.  

The fact to be proved is  
\begin{equation}   \label{e:extra}
\sum_{F\in \mathcal F} \eqref{e:3xx} \lesssim \mathscr R \lVert f\rVert _{\sigma } \lVert g\rVert_ \tau . 
\end{equation}
This will follow from a novel $ L ^2 $ estimate  below, in which we introduce a new 
measure $ \mu $ on $ \mathbb R ^2 _+$, derived from the stopping data.  
\begin{gather} \label{e:3aa}
\biggl\lVert 
\sum_{\substack{F, F'\in \mathcal F\\ F\subsetneqq F'} } 
 \mathbb E ^{\tau } _{Q _{F'}} \lvert  g\rvert  
 \sum_{K\in \mathcal W\!F}  \widehat   {\mathsf T }_{\tau } (\widehat Q _{F',F}) (x _{Q_K}) \cdot    W_K 
\biggr\rVert _{L ^2 (\mathbb R ^2 _+, \mu )} \lesssim   \mathscr R   \lVert g\rVert _{\tau }, 
\\
\label{e:mu}
\mu \equiv \sum_{F\in \mathcal F} \sum_{K\in \mathcal W\!F} 
\delta _{x _{Q_K}}  
{ \sum_{\substack{J \::\: J\subset K \\ \tilde \pi _{\mathcal F}   J = F}}
\langle  t, h ^{\sigma } _{J}\rangle _{\sigma } ^2 }  , 
 \end{gather} 
 and $ W _{K}  = K \times [ \lvert  K\rvert/2 , \lvert  K\rvert)$ is the top half of a Carleson box over $ K$. 
 The reduction to \eqref{e:3aa} is elementary, and presented here.  
 
 There is however one point about the definition of $\mu$ that  requires clarification.   
 If $K\Subset_{4r} F$, and $F' \in \mathcal{F}$ 
 strictly contains $F$, then, there is no interval $J\subset K$ with $\tilde{\pi}_ {\mathcal{F}} J=F'$.   
 For the purposes of the proof of \eqref{e:extra}, we can further restrict the collection $\mathcal{W}\!F$ to those $K$ for which there is 
 some $J\subset K$ with $\tilde{\pi}_ {\mathcal{F}}J=F$. 
 (If $K \in \mathcal{W}\!F$ does not meet this condition, it makes no contribution to $\mu$ in \eqref{e:mu}.) 
 We do so without changing the notation, and note that with this change, 
 for each $K$, there are at most $O(1)$ choices of $F$ so that $K\in \mathcal{W}\!F$. 
 We recall this point below.

\begin{proof}[Proof of \eqref{e:3aa} implies \eqref{e:extra}]   
This  definition is  associated with the inner most sum in \eqref{e:3xx}. 
\begin{align}
\phi ^2 _{K}  = 
\sum_{\substack{J \::\:  J\subset K \\\widetilde \pi _{\mathcal F} J = F }}   
\lvert  \hat f_{\sigma} (J) \rvert  ^2 ,  \qquad K\in \mathcal W\!F . 
\end{align}
By orthogonality of the adapted Haar basis,  we have 
$ \sum_{F\in \mathcal F}\sum_{K \in \mathcal W\!F} \phi ^2 _K  \leq \lVert f\rVert _{\sigma } ^2 $. 
Using the definition \eqref{e:mu}, we have 
\begin{equation*}
\mu (W_K) = 
\sum_{\substack{J \::\:  J\subset K \\\widetilde \pi _{\mathcal F} J = F }}   
\bigl\langle  t  ,  h ^{\sigma } _{J}\bigr\rangle _{\sigma } ^2  ,  \qquad K\in \mathcal W\!F .  
\end{equation*}
Thus, we have by Cauchy-Schwarz in the index $ K$, and $ F$, 
\begin{align*}
\eqref{e:extra} & \lesssim  
\sum_{F\in\mathcal F} 
\sum_{\substack{F'\in \mathcal F \;:\; F\subsetneqq F' } } 
 \mathbb E ^{\tau } _{Q _{F'}} \lvert  g\rvert  \sum_{K\in \mathcal W\!F} 
\widehat{ \mathsf T}_{\tau } (\widehat Q _{F',F} )( x _{Q_K} ) 
\mu (W_K) ^{1/2} \phi _{K} 
\\
&\leq 
\Gamma  \Biggl[ \sum_{F\in \mathcal F}  \sum_{K\in \mathcal W\!F} 
\phi _{K} ^2 \Biggr] ^{1/2} 
 \leq \Gamma   \lVert f\rVert _{\sigma } , 
\\ 
\noalign{\noindent where the term $ \Gamma $ on the right is } 
\Gamma  & = \Biggl[ 
\sum_{K }  
\Biggl[  \sum_{F\in \mathcal F \,:\, K\in \mathcal W\! F}
\sum_{\substack{F' \in \mathcal F \;:\; F\subsetneqq F' } }  
\mathbb E ^{\tau } _{Q _{F'}} \lvert  g\rvert \cdot  \widehat{ \mathsf T}_{\tau } (\widehat Q _{F',F} )( x _{Q_K} ) \Biggr] ^2  \mu (W_K)\Biggr] ^{1/2}  . 
\end{align*}
This is another way to write the  left half of \eqref{e:3aa}. 
Therefore, the inequality \eqref{e:3aa}  implies \eqref{e:extra}. 

\end{proof}

Now,  the inequality \eqref{e:3aa} is a two weight inequality for a Poisson-like operator, with a `hole in the argument.' 
On the one hand, there is no general theorem one can appeal to for such an inequality, and on the other, 
 a direct argument is not too hard, since stopping data for $ g$ has been used to construct the new measure $ \mu $.

 Experience dictates that \eqref{e:3aa} is best proved by duality. Thus, for
non-negative function $\gamma \in L^{2}(\mathbb{R}_{+}^{2},\tau )$, we
should show that 
\begin{align}
\sum_{F^{\prime }\in \mathcal{F}}& \sum_{\substack{ F^{\prime \prime }\in 
\mathcal{F}\;:\;F^{\prime \prime }\subsetneqq F^{\prime }}}\mathbb{E}%
_{Q_{F^{\prime }}}^{\tau }\lvert g\rvert \sum_{K^{\prime \prime }\in 
\mathcal{W}\!F^{\prime \prime }}\widehat{\mathsf{T}}_{\tau }(\widehat{Q}%
_{F^{\prime },F^{\prime \prime }})(x_{Q_{K^{\prime \prime
}}})\int_{W_{K^{\prime \prime }}}\gamma \;d\mu   \label{e:3bb} \\
& \approx \sum_{F^{\prime }\in \mathcal{F}}\sum_{\substack{ F\in \textup{Ch}%
_{\mathcal{F}}(F^{\prime })}}\mathbb{E}_{Q_{F^{\prime }}}^{\tau }\lvert
g\rvert \sum_{K\in \mathcal{W}\!F}\widehat{\mathsf{T}}_{\tau }(\widehat{Q}%
_{F^{\prime },F})(x_{Q_{K}})\int_{Q_{K}}\gamma \;d\mu \lesssim \mathcal{R}%
\lVert g\rVert _{\tau }\lVert \gamma \rVert _{\mu }.
\end{align}%
To see that the approximate equality holds above, it suffices to show that
for each fixed $F^{\prime }\in \mathcal{F}$,%
\begin{align}  
\sum_{\substack{ F^{\prime \prime }\in \mathcal{F}\;:\;F^{\prime \prime
}\subsetneqq F^{\prime }}}
&
\sum_{K^{\prime \prime }\in \mathcal{W}\!F^{\prime
\prime }}\widehat{\mathsf{T}}_{\tau }(\widehat{Q}_{F^{\prime },F^{\prime
\prime }})(x_{Q_{K^{\prime \prime }}})\int_{W_{K^{\prime \prime }}}\gamma
\;d\mu   \label{suff} \\
&\approx \sum_{\substack{ F\in \textup{Ch}_{\mathcal{F}}(F^{\prime })}}%
\sum_{K\in \mathcal{W}\!F}\widehat{\mathsf{T}}_{\tau }(\widehat{Q}%
_{F^{\prime },F})(x_{Q_{K}})\int_{Q_{K}}\gamma \;d\mu \ .  \notag
\end{align}%
To prove \eqref{suff},  start on the right hand side  with a fixed $\mathcal{F}$-child $F\in 
\textup{Ch}_{\mathcal{F}}(F^{\prime })$ and a fixed $K\in \mathcal{W}\!F$. 
Note the decomposition 
\begin{equation*}
\int_{Q_{K}}\gamma \;d\mu =\sum_{L\in \mathcal{D}:\ L\subset
K}\int_{W_{L}}\gamma \;d\mu =\sum_{F^{\prime \prime }\in \mathcal{F}:\
F^{\prime \prime }\subset F}\sum_{K^{\prime \prime }\in \mathcal{W}%
\!F^{\prime \prime }:K^{\prime \prime }\subset K\ }\int_{W_{K^{\prime \prime
}}}\gamma \;d\mu ,
\end{equation*}%
which holds since if $L\in \mathcal{D}\ $satisfies $L\subset K$, but $L$ is 
\textbf{not} equal to any $K^{\prime \prime }\in \mathcal{W}\!F^{\prime
\prime }\mathcal{\ }$for some $F^{\prime \prime }\in \mathcal{F}$ with $%
F^{\prime \prime }\subset F$, then $\int_{W_{L}}\gamma \;d\mu =0$ by the
definition of $\mu $ above, and so does not contribute to the sum. Next we
note that if $F^{\prime \prime }\subset F$ and $K^{\prime \prime }\in 
\mathcal{W}\!F^{\prime \prime }$, then%
\begin{equation*}
\widehat{\mathsf{T}}_{\tau }(\widehat{Q}_{F^{\prime },F^{\prime \prime
}})(x_{Q_{K^{\prime \prime }}})\approx \widehat{\mathsf{T}}_{\tau }(\widehat{%
Q}_{F^{\prime },F})(x_{Q_{K}}),
\end{equation*}%
because (1) $\widehat{Q}_{F^{\prime },F^{\prime \prime }}=\widehat{Q}%
_{F^{\prime },F}$ by definition and the assumption that $F\in \textup{Ch}_{%
\mathcal{F}}(F^{\prime })$, and (2) $\widehat{\mathsf{T}}_{\tau }(\widehat{Q}%
_{F^{\prime },F})(x_{Q_{K^{\prime \prime }}})\approx \widehat{\mathsf{T}}%
_{\tau }(\widehat{Q}_{F^{\prime },F})(x_{Q_{K}})$ because both $K^{\prime
\prime }$ and $K$ have their triples contained in $F$. Altogether then, we
have for a fixed $\mathcal{F}$-child $F\in \textup{Ch}_{\mathcal{F}%
}(F^{\prime })$ and a fixed $K\in \mathcal{W}\!F$,%
\begin{eqnarray*}
\widehat{\mathsf{T}}_{\tau }(\widehat{Q}_{F^{\prime
},F})(x_{Q_{K}})\int_{Q_{K}}\gamma \;d\mu  &=&\sum_{F^{\prime \prime }\in 
\mathcal{F}:\ F^{\prime \prime }\subset F}\sum_{K^{\prime \prime }\in 
\mathcal{W}\!F^{\prime \prime }:K^{\prime \prime }\subset K\ }\widehat{%
\mathsf{T}}_{\tau }(\widehat{Q}_{F^{\prime
},F})(x_{Q_{K}})\int_{W_{K^{\prime \prime }}}\gamma \;d\mu  \\
&\approx &\sum_{F^{\prime \prime }\in \mathcal{F}:\ F^{\prime \prime
}\subset F}\sum_{K^{\prime \prime }\in \mathcal{W}\!F^{\prime \prime
}:K^{\prime \prime }\subset K\ }\widehat{\mathsf{T}}_{\tau }(\widehat{Q}%
_{F^{\prime },F^{\prime \prime }})(x_{Q_{K}})\int_{W_{K^{\prime \prime
}}}\gamma \;d\mu .
\end{eqnarray*}%
Now sum over all $\mathcal{F}$-children $F\in \textup{Ch}_{\mathcal{F}%
}(F^{\prime })$ and $K\in \mathcal{W}\!F$ and use%
\begin{equation*}
\sum_{\substack{ F\in \textup{Ch}_{\mathcal{F}}(F^{\prime })}}\sum_{K\in 
\mathcal{W}\!F}\sum_{F^{\prime \prime }\in \mathcal{F}:\ F^{\prime \prime
}\subset F}\sum_{K^{\prime \prime }\in \mathcal{W}\!F^{\prime \prime
}:K^{\prime \prime }\subset K\ }=\sum_{\substack{ F^{\prime \prime }\in 
\mathcal{F}\;:\;F^{\prime \prime }\subsetneqq F^{\prime }}}\sum_{K^{\prime
\prime }\in \mathcal{W}\!F^{\prime \prime }}
\end{equation*}%
to obtain (\ref{suff}).
 
\bigskip
The inequality \eqref{e:3bb} can be reduced to two testing inequalities, following the method known as the \emph{parallel corona.} 
This method uses stopping data of the relevant functions to split the sum. 
The measure $ \mu $ is built from stopping data for $ g$, already.  
We need stopping data for the dual function $ \gamma $.  
Let $ \mathcal K = \bigcup_{F\in \mathcal F} \mathcal W\!F$.  
Let $ \mathcal G\subset \mathcal K$ be stopping data for $ \gamma $.  Namely, without loss of generality, we can assume that 
$ \gamma $ is supported on the union of disjoint cubes  $ Q _{G_j}$, for $ j\geq 1$, and $ G_j\in \mathcal K$. 
Add these intervals $ G_j$ to $ \mathcal G$, and in the inductive step, if $ G\in \mathcal G$ is minimal, add to $ \mathcal G$ the maximal dyadic children $ G' \in \mathcal K$ of $ G$, 
if they exist, such that $ \mathbb E ^{\mu } _{Q _{G'}} \gamma > 10  \mathbb E ^{\mu } _{Q _{G}} \gamma$.  

Using this notation, we write the sum in \eqref{e:3bb} as equal to the sum of these two terms. The first is 
a sum over $ F'\in \mathcal F$,  of the expression below. 
\begin{align}  \label{e:xx1st}
 \mathbb E ^{\tau } _{Q _{F'}} \lvert  g\rvert 
 \sum_{\substack{F\in \textup{Ch} _{\mathcal F} (F') } } 
\sum_{\substack{K\in \mathcal W\!F \\  \pi _{\mathcal G} K \subset F'}}  
\widehat {\mathsf T }_{\tau } (\widehat Q _{F',F}) (x _{Q_K})  \int _{Q_K}  \gamma  \; d \mu   . 
\end{align}
The key point is that $ \pi _{\mathcal G} K \subset F'$.  
And the second is the sum over $ G\in \mathcal G$ of the expression below.  
The key point is that $ F'\subsetneq G = \pi _{\mathcal G} K$.  
\begin{equation}\label{e:xx2nd}
 \sum_{\substack{F'\in \mathcal F\\   F' \subsetneq  G} } 
   \mathbb E ^{\tau } _{Q _{F'}} \lvert  g\rvert 
 \sum_{\substack{F\in \textup{Ch} _{\mathcal F} (F') } } 
 \sum_{\substack{K\in \mathcal W\!F \\   \pi _{\mathcal G} K=G}}  
\widehat {\mathsf T }_{\tau } ( \widehat Q _{F',F}) (x _{Q_K}) \int _{Q_K}  \gamma  \; d \mu    . 
\end{equation}

The first term \eqref{e:xx1st} obeys this inequality of testing type.  Uniformly in $ F'\in \mathcal F$, 
\begin{equation}\label{e:x1st}
\eqref{e:xx1st} \lesssim 
 \mathscr R   \mathbb E ^{\tau } _{Q _{F'}} \lvert  g\rvert   \tau (Q_{F'}) ^{1/2} 
\Biggl[
\sideset{}{^{F'}}\sum _{G\in \mathcal G} 
 ( \mathbb E ^{\mu } _{Q_G} \gamma ) ^2 \mu (Q_G)
\Biggr] ^{1/2} . 
\end{equation}
Above, the notation for the sum means that the sum is restricted to those $G\in \mathcal{G}$ so that $ \pi_{\mathcal{G}}K=G$ 
for some $K\in \mathcal{W}\!F$ and $F\in \textup{Ch} _{\mathcal F} (F')$.  
 Recall that we impose the additional condition that each  $K$ is in at most a bounded number of collections $\mathcal{W}\!F$. 
Observe that  Cauchy-Schwarz and the  quasi-orthogonality principle applies to bound the sum over $ F\in \mathcal F$ on the right in \eqref{e:x1st}.  The bound is 
 $ \lesssim  \mathscr R \lVert g\rVert _{\tau } \lVert \gamma \rVert _{\mu }$, as required.   
 
 The second term  \eqref{e:xx2nd}  satisfies the testing inequality below, uniformly in $ G\in \mathcal G$. 
\begin{equation}\label{e:x2nd}
\eqref{e:xx2nd} \lesssim 
\mathscr R  \cdot  \mathbb E ^{\mu } _{Q_G} \gamma  \cdot \mu (Q_G) ^{1/2} 
\Biggl[ 
\sum_{\substack{F'\in \mathcal F \\  \pi _{\mathcal G}F'=G }}  (\mathbb E ^{\tau } _{Q _{F'}} g) ^2 \tau (Q_{F'}) 
\Biggr] ^{1/2} . 
\end{equation}
And, quasi-orthogonality completes the bound in this case as well, completing the proof of \eqref{e:3aa}, subject to the two testing inequalities.  
The  two testing  inequalities \eqref{e:x1st} and \eqref{e:x2nd} are proved below.      
\end{proof}

\begin{proof}[Proof of \eqref{e:x1st}]  
The term $ \mathbb E ^{\tau } _{Q _{F'}} \lvert  g\rvert$ on the right plays no role in the analysis.  
Apply Cauchy-Schwarz in $ L ^2 (\mu )$ in \eqref{e:xx1st}.  On the one hand, 
we have an instance of the energy inequality.  Namely, 
\begin{align}
\label{e:xx1a}
\int _{Q_{F'}}  
\biggl[   \sum_{\substack{F\in \textup{Ch} _{\mathcal F} (F') } } &
\sum_{K\in \mathcal W\!F} \widehat {\mathsf T }_{\tau } (\widehat  Q_{F',F}) (x _{Q_K})  Q_K    \biggr] ^2 \; d \mu 
\\
&= 
 \sum_{\substack{F\in \textup{Ch} _{\mathcal F} (F') } } \sum_{K\in \mathcal W\!F} \widehat {\mathsf T }_{\tau } ({\widehat  Q_{F',F}} ) (x _{Q_K})  ^2 \mu (Q_K) 
\lesssim \mathscr R ^2 \tau (Q_{F'}). 
\end{align}
Indeed, the sets $ Q_K$ are pairwise disjoint as  $F'$ is fixed, and recalling that 
\begin{gather*}
  \mu (Q_K) = \sum_{J \;:\;  J\subset K, \tilde \pi J\subset F} \langle t, h ^{\sigma } _{J} \rangle_ \sigma  ^2 , \quad  K\in\mathcal{W}F
\end{gather*}
the energy inequality, \eqref{e:energyI} implies the bound above.  
This is half of the expression on the right in \eqref{e:xx1st}.  

On the other hand, we turn our attention to the function $\gamma$, and the application of Cauchy-Schwarz 
gives us the term 
\begin{align*}
\sum_{\substack{F\in \textup{Ch} _{\mathcal F} (F') } } 
\sum_{\substack{K\in \mathcal W\!F \\  \pi _{\mathcal G} K \subset F'}} 
\bigl( \mathbb E ^{\mu } _{Q_K} \gamma \bigr) ^2 \mu (Q_K)
\lesssim \sideset{}{^{F'}}\sum _{G\in \mathcal G} 
 ( \mathbb E ^{\mu } _{Q_G} \gamma ) ^2 \mu (Q_G).
\end{align*}
We have appealed to disjointness of the $ Q_K$ again. And we  appeal to the construction of the stopping data and the notation of \eqref{e:x1st}. This completes the proof.   
\end{proof}

\begin{proof}[Proof of \eqref{e:x2nd}] 
In this case, we will only need the $ \mathscr A_2$ condition, and  the operator $ \widehat {\mathsf T } $ is dualized.   The expression to bound is 
\begin{align*}
\eqref{e:xx2nd}= \mathbb E ^{\mu } _{Q_G} \gamma 
 \int _{Q_G}  
  \sum_{\substack{F'\in \mathcal F\\   F' \subsetneq  G} } 
   \mathbb E ^{\tau } _{Q _{F'}} \lvert  g\rvert 
 \sum_{\substack{F\in \textup{Ch} _{\mathcal F} (F') } } 
 \sum_{\substack{K\in \mathcal W\!F \\   \pi _{\mathcal G} K=G}}  
     {\widehat  Q_{F',F}}  
\widehat {\mathsf T } ^{\ast} _{\mu } ( Q_K) \; d \tau .
 \end{align*}
 Applying Cauchy-Schwarz in the index $ F'$, we will get on the one hand the term below, 
 \begin{equation*}
  \sum_{\substack{F'\in \mathcal F\\   F' \subsetneq  G} } 
\bigl(   \mathbb E ^{\tau } _{Q _{F'}} \lvert  g\rvert  \bigr) ^2 \tau (Q _{F'}), 
\end{equation*}
which forms half of the right side of \eqref{e:x2nd}.  
Therefore, we should establish the $ L ^2 $-estimate below
\begin{align}  \label{e:XXxx}
\int _{Q_G} 
\sum_{\substack{F '\in \mathcal F \\ \pi _{\mathcal G} F '= G }} 
\Biggl[  
 \sum_{\substack{F\in \textup{Ch} _{\mathcal F} (F') } } 
{\widehat  Q_{F',F}}  
 \sum_{K\in \mathcal W\!F }\widehat {\mathsf T } ^{\ast} _{\mu } (Q_K) 
\Biggr] ^2  \; d \tau 
\lesssim \mathscr R ^2 \mu (Q_G).  
\end{align}

This is a consequence of the $ \mathscr A_2$ condition, after several reductions.  
The latter require some additional summing indices; we will gain geometric decay in all of them. 

\begin{itemize}
\item  Above, $F'$ varies, so that the  Carleson cubes   $ Q_K$ overlap. We make this definition.  
For integers $ k \geq 0$, write $ W ^{k}_K = K \times [2 ^{-k-1} \lvert  K\rvert, 2 ^{-k} \lvert  K\rvert)  $. 
These subsets of $ \mathbb R ^2 _+$ are disjoint in $ k\geq 0$, and $ K$, and $ \mu (W ^{k}_K) \leq 2 ^{-2k} 
\lvert  K\rvert ^2 \sigma (K) $, as follows from the definition of $ \mu $ in \eqref{e:mu}. 

\item  Let $ K _{F, \ell }$, for $ \ell >r$,  be a choice  $ K\in \mathcal W\!F$ with  $ 2 ^{\ell } \lvert  K\rvert = \lvert  F\rvert  $. 
By the Whitney property of $ \mathcal W\!F$, there are $ O (1)$ possible choices of such an interval.  

\item For integers $ m  \geq \ell  $, let $\widehat   Q _{F', F} ^{\ell ,m} = Q _{F'} \cap  (Q _{ 2 ^{m+1 - \ell }F  } \setminus Q _{2 ^{m - \ell }F })$. 

\item   We employ a standard `separation of scales' trick.  
Let $ \mathcal D _{ m } $ be a subset of our dyadic grid $ \mathcal D$ so that   for all  $ F_1 \neq F_2 \in \mathcal D _{m }$,  
if $ \lvert  F_1\rvert = \lvert  F_2\rvert  $, then $ \textup{dist} (F_1, F_2 ) \geq 2 ^{ m  } \lvert  F_1\rvert $, 
and if $ \lvert  F_1\rvert < \lvert  F_2\rvert  $, then $2 ^{2m } \lvert  F_1\rvert \leq \lvert  F_2\rvert$.   
Then, set $  \textup{Ch} _{\mathcal F, m } (F') = \mathcal D _{m } \cap  \textup{Ch} _{\mathcal F } (F')$.  
Note that a dyadic grid is the union of $ Cm 2 ^{m}$ such grids with `separation of scales.' 
\end{itemize}
To prove \eqref{e:XXxx}, 
it suffices to show the estimate below.  It  has geometric decay in 
$ k$ and  $ m$, and hence $  \ell $.   Uniformly in $ F'\in \mathcal F$, $ k  \geq 0$, and $  m\geq \ell >r$, 
\begin{equation}\label{e:x2a} 
 \int _{Q_{F'} }  
\Biggl[  \sum_{\substack{F\in \textup{Ch} _{\mathcal F, m } (F') } } 
{\widehat  Q _{F', F} ^{\ell ,m}} 
 \widehat {\mathsf T }^{\ast} _{\mu } (W^k_{K_{F,\ell}})  \Biggr] ^2 \; d \tau 
\lesssim \mathscr R ^2      2 ^{-2k-m (2 - 4\epsilon )   } 
  \sum_{\substack{F\in \textup{Ch} _{\mathcal F , m} (F') } } 
 \mu (W_{K_{F,\ell}} ) .
\end{equation}
We still must sum over the  $ O (m 2 ^{m})$ choices of grids $ \mathcal D _m$, which we can do because of the $ 2 ^{-2m (1- 2 \epsilon )}$ above. 
This estimate implies \eqref{e:XXxx}, with the full sum over $ F'\subset G$.

\smallskip 
The square in \eqref{e:x2a} is expanded.  One term concerns 
the sum  over $F\in \textup{Ch} _{\mathcal F, m } (F')  $ of 
\begin{align}  \label{e:we}
\int _{\widehat  Q_{F', F} ^{\ell ,m} }   \widehat {\mathsf T }^{\ast} _{\mu }
(W^k_{K_{F,\ell}} )   ^2 \; d \tau 
 & 
\leq   \sup _{x \in \widehat   Q _{F',F}  ^{\ell ,m}} \widehat {\mathsf T }^{\ast} _{\mu }
(W^k_{K_{F,\ell}} )(x) 
\int _{\widehat  Q_{F', F} ^{\ell ,m} }   \widehat {\mathsf T }^{\ast} _{\mu }
(W^k_{K_{F,\ell}} )  \; d \tau 
\\
 & 
=   \sup _{x \in \widehat   Q _{F',F}  ^{\ell ,m}} \widehat {\mathsf T }^{\ast} _{\mu }
(W^k_{K_{F,\ell}} )(x) 
\int _{W^k_{K_{F,\ell}} }  \widehat {\mathsf T } _{\tau } (\widehat  Q_{F', F} ^{\ell ,m}) 
\; d \mu 
\\   \label{e:we1}
&\leq 
  \sup _{x \in \widehat   Q _{F',F}  ^{\ell ,m}} \widehat {\mathsf T }^{\ast} _{\mu }
(W^k_{K_{F,\ell}} )(x) 
\sup _{x'\in W^k_{K_{F,\ell}} } 
 \widehat {\mathsf T } _{\tau } (\widehat  Q_{F', F} ^{\ell ,m}) (x') \cdot \mu (W^k_{K_{F,\ell}} )
\end{align}
Above, we replaced one $ \widehat {\mathsf T} ^{\ast}  $ with a supremum, and dualized the second $ \widehat {\mathsf T} ^{\ast} $, to make the appeal to the $ \mathscr A_2$ condition easier. (This same argument appears a second time.)

\begin{figure}
\begin{tikzpicture}
\draw (-6,0) -- (6,0); 
\draw (-6,0) rectangle (6,5) ; 
\draw (-4,0) rectangle (4,3) ; 
\draw (0,3.5) node[above] {$\widehat  Q_{F', F} ^{\ell ,m} = Q _{F'} \cap  (Q _{ 2 ^{m+1 - \ell }F  } \setminus Q _{2 ^{m - \ell }F })  $}; 
\draw (0, 0.5) rectangle (1,1) node [right] {$ W^k_{K_{F,\ell}}$}; 
\draw [<->] (-4,0.75) -- (0,0.75) node[midway,above] {$ \lesssim 2 ^{m (1- \epsilon)} \lvert  K _{F, \ell }\rvert $}; 
\draw[<->] (0.5,1) --(0.5,3) node[midway,right] {$ \approx 2^ {m} \lvert K_{F, \ell} \rvert $}; 
\end{tikzpicture}

\caption{The sets used in the proof of \eqref{e:x2nd}.}
\label{f:QW}
\end{figure}
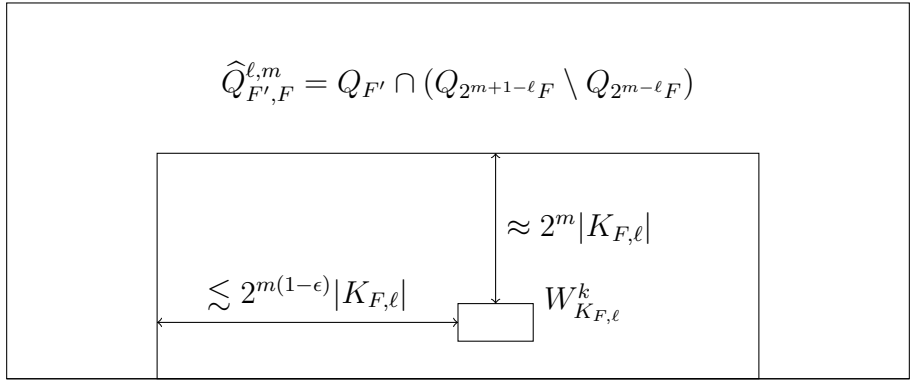

The two supremums in \eqref{e:we1} are estimated as follows.  In both, we use the definition of $ \widehat {\mathsf T }$ in \eqref{e:Ttilde}. 
Recall that $ W^k_{K_{F,\ell}} = K _{F, \ell } \times [2 ^{-k-1} \lvert  K_{F, \ell }\rvert, 2 ^{-k} \lvert  K_{F, \ell }\rvert)  $, 
where  $ K _{F , \ell }$ is a choice of $ K\in \mathcal W\!F$ with $ 2 ^{\ell } \lvert  K\rvert= \lvert  F\rvert  $, 
and $ \widehat   Q _{F', F} ^{\ell ,m} = Q _{F'} \cap  (Q _{ 2 ^{m+1 - \ell }F  } \setminus Q _{2 ^{m - \ell }F })$.  See Figure~\ref{f:QW}. 
Then, for the first supremum in \eqref{e:we1}
\begin{align}
  \sup _{x \in \widehat   Q _{F',F}  ^{\ell ,m}} \widehat {\mathsf T }^{\ast} _{\mu }
(W^k_{K_{F,\ell}} )(x )      
&=   \sup _{x \in \widehat   Q _{F',F}  ^{\ell ,m}} 
\int _{W^k_{K_{F,\ell}}}  \frac {  1 } {  y_2 ^2 + (y_1 -x _1) ^2 + x_2 ^2  } \; \mu  (dy)
\\
& \leq  \mu (W^k_{K_{F,\ell}})  
 \sup _{x \in \widehat   Q _{F',F}  ^{\ell ,m}}  \sup _{y \in W^k_{K_{F,\ell}}}  \frac {  1  } {  x_2^2 +  (y_1 -x _1) ^2    }
 \\  \label{e:we2} 
 & \lesssim  2 ^{ - 2 m (1- \epsilon ) } 
   \frac {\mu (W^k_{K_{F,\ell}} )} {  \lvert  {K _{F, \ell }}\rvert ^2    }  . 
\end{align}
Above,  the supremum  is small, at most $2^{-2m(1- \epsilon)} \lvert K_{F,\ell} \rvert ^{-2}$, as one sees from Figure \ref{f:QW}. 
For the second supremum in \eqref{e:we1}, compare to the Poisson average of $\tau$: 
\begin{align}
\sup _{x\in W^k_{K_{F,\ell}} } 
 \widehat {\mathsf T } _{\tau } (\widehat  Q_{F', F} ^{\ell ,m}) (x) 
 & = 
 \sup _{x\in W^k_{K_{F,\ell}} } \int _{\widehat  Q_{F', F} ^{\ell ,m} }   
  \frac { 1   } {  y_2 ^2 + (y_1 -x _1) ^2 + x_2 ^2  }  \; \tau  (dy) 
  \\ \label{e:we3} 
  & \lesssim  \lvert K_{F, \ell}\rvert ^{-1} \mathsf P _{\tau}( \widehat  Q_{F', F}, K_{F, \ell}). 
\end{align}
In using these two estimates, we use the inequality $ \mu (W  _{K} ^{k}) \lesssim  2 ^{-2k} \left\vert K\right\vert^2 \sigma (K)$,  
a direct consequence of the definition of $ \mu $ in \eqref{e:mu}.  
Combining \eqref{e:we2} and \eqref{e:we3}, we  have  
\begin{align*}
\eqref{e:we} & \lesssim  2 ^{ - 2 m (1- \epsilon ) } 
  \frac {\mu (W^k_{K_{F,\ell}})} { \lvert  {K _{F, \ell }}\rvert  ^3   } 
   \mathsf P _{\tau} (  \widehat  Q_{F', F}, K_{F, \ell})   \cdot \mu (W^k_{K_{F,\ell}} )  
   \\
   & \lesssim  2 ^{-2k- 2m (1- \epsilon) }  
    \frac { \sigma (K _{F, \ell } ) }  { \lvert  {K _{F, \ell }}\rvert    } 
 \mathsf P _{\tau} ( \widehat  Q_{F', F}, K_{F, \ell}) \cdot \mu (W^k_{K_{F,\ell}} )  
   \\
   & \lesssim   2 ^{-2k-2m (1- 2\epsilon )}   \mathscr A_2 \cdot \mu (W^k_{K_{F,\ell}} )  . 
\end{align*}
This completes the proof \eqref{e:x2a} in this case.  

\medskip 

The second term of the square concerns those $ F_1 \neq F_2 \in  \textup{Ch} _{\mathcal F, m } (F') $ 
such that $ \widehat  Q_{F', F_1} ^{\ell ,m} \cap \widehat  Q_{F', F_2} ^{\ell ,m}\neq \emptyset $.  
We can then assume that $ \lvert  F_1\rvert \geq 2 ^{2m} \lvert  F_2\rvert  $,  by our separation of scales argument. 
We will show that 
\begin{gather}\label{e:we5}
\int _{\widehat  Q_{F', F_1} ^{\ell ,m} }   \widehat {\mathsf T }^{\ast} _{\mu }
(W_ {K _{F_1, \ell }} ^k)  
 \sideset{}{^{\ell,m}}\sum _{F_2} 
{\widehat  Q_{F', F_2} ^{\ell ,m} }   \cdot  \widehat {\mathsf T }^{\ast} _{\mu }
(W_ {K _{F_2, \ell }} ^k )  \; d \tau 
\lesssim 
2 ^{-2k-  4m   } \mathscr A_2 \mu (W_ {K _{F_1, \ell }}), 
\\
\textup{where} \qquad 
 \sideset{}{^{\ell,m}}\sum _{F_2}  \cdots  := \sum_{ \substack{F_2 \in  \textup{Ch} _{\mathcal F, m } (F')  \\ \widehat  Q_{F', F_1} ^{\ell ,m} \cap \widehat  Q_{F', F_2} ^{\ell ,m} \neq \emptyset   }} \cdots . 
\end{gather}
This inequality completes the proof of  \eqref{e:x2a} upon summation of $ F_1$. 

    With $ F_1$ fixed, we can then estimate using the similar kind of reasoning as above. We in particular appeal to \eqref{e:we2} and \eqref{e:we3}.  We pull one power of $  \widehat {\mathsf T }_{\tau  } ^{\ast} $ outside the integral, and then dualize the one inside to obtain 
    
\begin{align*}
\textup{LHS of \eqref{e:we5}  } 
& \lesssim 
\sup _{x\in\widehat   Q _{F', F_1} ^{\ell ,m}}  
 \widehat {\mathsf T }^{\ast} _{\mu }
(W^k_ {K _{F_1, \ell }} )   
  \sideset{}{^{\ell,m}}\sum  _{F_2} 
\int _{ W^k_ {K _{F_2, \ell }}}    \widehat {\mathsf T }_{\tau  }  ( \widehat  Q_{F', F_2} ^{\ell ,m} )  \; d \mu  
\\ \noalign{\noindent followed by \eqref{e:we2}, and a trivial inequality, }
& \lesssim 2 ^{-2 m (1- \epsilon )}
 \frac {\mu (W^k _{K _{F_1, \ell }})} { \lvert  {K _{F_1, \ell }}\rvert ^2  }  
 \sideset{}{^{\ell,m}}\sum _{F_2}
\sup _{x\in W^k_ {K _{F_2, \ell }} } 
   \widehat {\mathsf T }_{\tau  }  ( \widehat  Q_{F', F_2} ^{\ell ,m} )( x  )  \cdot   \mu   (W^k_ {K _{F_2, \ell }}  ) 
\\
\noalign{\noindent  and inside the sum use  \eqref{e:we3}, and the bound $ \mu (W^k_K) \lesssim 2 ^{-2k}\lvert  K\rvert ^2 \sigma (K) $,} 
& \lesssim 
2 ^{-2k- 2 m (1- \epsilon )   }  
   \frac{\mu (W^k _{K _{F_1, \ell }})} { \lvert  {K _{F_1, \ell }}\rvert ^2  } 
 \sideset{}{^{\ell,m}}\sum _{F_2}
\mathsf P _{\tau} (  \widehat  Q_{F', F_2} ^{\ell ,m} , K _{F_2, \ell }) 
 \sigma   (W^k_ {K _{F_2, \ell }}  ) \lvert K _{F_2, \ell }\rvert 
\\ 
\noalign{\noindent  now appeal to the   $ \mathscr A_2$ condition with holes \eqref{e:A2}, } 
& \lesssim 
2 ^{-2k- 2 m (1- \epsilon )  } \mathscr A_2   
   \frac{\mu (W^k _{K _{F_1, \ell }})} { \lvert  {K _{F_1, \ell }}\rvert ^2  } 
 \sideset{}{^{\ell,m}}\sum _{F_2}
   \lvert  K _{F_2, \ell }\rvert ^2  
\\
\noalign{\noindent  but the last sum is easy to estimate by $  2^{-2m}\lvert  {K _{F_1, \ell }}\rvert ^2 $, so that  }
& \lesssim 
 2 ^{-2k-  4m(1-\epsilon) } \mathscr A_2 \cdot 
   {\mu (W^k _{K _{F_1, \ell }})} . 
\end{align*}
  This is a better estimate than in the first case, so  the proof of \eqref{e:x2a} is complete. 
\end{proof}
\section{Local Estimates} \label{s:local}

The bound for the form $ B ^{\textup{above}} (f,g)$, defined in \eqref{e:above}, 
is a sum over pairs of intervals $ (I,J)$ so that $ J\Subset _{4r} I$.  But, we have proved the  global to local estimate \eqref{e:Cglobal}. In it, we have restricted the sum to pairs of intervals $ (I,J)$ so that  
$ \tilde \pi _{\mathcal F} J \subsetneq I$, where 
$ \tilde \pi _{\mathcal F} J$ is the minimal stopping interval $ F \in \mathcal F$ so that $ J\Subset _{4r} F$.  
This condition is equivalent to $ \tilde \pi _{\mathcal F} J \subset I_J$.  
Therefore, it suffices to bound the complementary sum,  in which $ I_J \subsetneqq \tilde \pi _{\mathcal F} J$, 
which is equivalent to $ I\subset \tilde \pi _{\mathcal F} J$.   And, since $ J\Subset  _{4r}I$, we then see that $ \pi _{\mathcal F} I= \tilde \pi _{\mathcal F} J$.  
The sum is explicitly given by
\begin{align}
\sum_{I} \sum_{\substack{J \::\:   I \subset \tilde \pi _{\mathcal F} J  \\ J\Subset _{4r} I}}
&  \mathbb E ^{\tau } _{Q_{IJ}} \Delta ^{\tau } _{Q_I} g \cdot 
\left\langle  {\mathsf R}^{\ast}_\tau Q _{IJ},    \Delta ^{\sigma } _{J} f \right\rangle _{\sigma}  
\\ 
& = 
\sum _{F\in \mathcal F}   \; \sum_{I \;:\; \pi _{\mathcal F} I=F} 
\; 
\sum_{\substack{J \::\:  \tilde \pi _{\mathcal F} J  = F \\ J\Subset _{4r} I}} \mathbb E ^{\tau } _{Q_{IJ}} \Delta ^{\tau } _{Q_I} g \cdot 
\left\langle  {\mathsf R}^{\ast}_\tau Q _{IJ},    \Delta ^{\sigma } _{J} f \right\rangle _{\sigma}.  
\end{align}
 Note that with the stopping interval $ F$ fixed,  we  have $ F =  \tilde \pi _{\mathcal F} J = \pi _{\mathcal F} I$, 
since $ J\Subset _{4r}I$.   

We want to make the sum a bit more restrictive.  Specializing the sum to the case of $ I_J\in \mathcal F$ above, we have 

\begin{proposition}\label{p:rr} There holds 
\begin{equation}\label{e:44}
\Bigl\lvert 
\sum _{F\in \mathcal F}  \; \sum_{I \;:\; \pi _{\mathcal F} I=F} 
\;  
\sum_{\substack{J \::\:  \tilde \pi _{\mathcal F} J  = F \\ I_J \in \mathcal F  ,\ J\Subset _{4r} I}} \mathbb E ^{\tau } _{Q_{IJ}} \Delta ^{\tau } _{Q_I} g \cdot 
\left\langle  {\mathsf R}^{\ast}_\tau Q _{IJ},    \Delta ^{\sigma } _{J} f \right\rangle _{\sigma} 
\Bigr\rvert 
\lesssim \mathscr T \lVert g\rVert _{\tau } \lVert f\rVert _{\sigma } 
\end{equation}
\end{proposition}

\begin{proof}
Only using the testing condition,  and assuming that $ \pi ^{2} _{\mathcal F} F'=F$, 
\begin{equation*}
\Bigl\lvert 
\mathbb E ^{\tau } _{Q _{F'}} g 
\sum _{\substack{ J \;:\; \Subset _{4r-1} F' } }  
\left\langle  {\mathsf R}^{\ast}_\tau Q _{F'},    \Delta ^{\sigma } _{J} f \right\rangle _{\sigma}
\Bigr\rvert
\lesssim \mathscr T 
\mathbb E ^{\tau } _{Q _{F'}} \lvert  g\rvert \cdot  \tau (Q _{F'}) ^{1/2} 
\Bigl\lVert  \sum _{\substack{ J \;:\; J \Subset _{4r-1} F' } }  \Delta ^{\sigma } _{J} f \Bigr\rVert _{\sigma }.  
\end{equation*}
This can be summed over $ F\in \mathcal F$, and $ F' $ an $ \mathcal F$-child of $ F$, using quasi-orthogonality.  
\end{proof}

The proposition above leaves us with the task of proving a bound with the additional constraint that $ I_J\not\in \mathcal F$.

\begin{lemma}\label{l:local} For each $ F\in \mathcal F$, we have 
\begin{equation}\label{e:local}
\Bigl\lvert 
 \sum_{I \;:\; \pi _{\mathcal F} I=F} 
\; 
\sum_{\substack{J \::\:  \tilde \pi _{\mathcal F} J  = F \\ I_J \not\in \mathcal F,  J\Subset _{4r} I}} \mathbb E ^{\tau } _{Q_{IJ}} \Delta ^{\tau } _{Q_I} g \cdot 
\left\langle  {\mathsf R}^{\ast}_\tau Q _{IJ},    \Delta ^{\sigma } _{J} f \right\rangle _{\sigma}  
\Bigr\rvert
\lesssim \mathscr R  
\{\mathbb E^{\tau}_{Q_F}\lvert g\rvert \cdot\tau (Q_F) ^{1/2} + \lVert H^{\tau } _{F} g\rVert_{\tau } \}
\lVert \tilde H^{\sigma } _{F} f\rVert_{\sigma } . 
 \end{equation}
\end{lemma}

This case is much more complicated.  
The following step is a simple appeal to the Carleson cube testing hypothesis.  
For each $ J$ with $\tilde \pi _{\mathcal F} J=F$, define $ \varepsilon _J$  by the formula 
\begin{equation*}
\varepsilon _J \mathbb E^{\tau}_{Q_F}\lvert g\rvert \equiv \sum_{\substack{I \::\: \pi _{\mathcal F} I =F\\ I_J \not\in \mathcal F, \  J\Subset _{4r} I }} 
\mathbb E ^{\tau } _{Q _{IJ}} \Delta ^{\tau } _{Q_I} g . 
\end{equation*}
It follows from the construction of the stopping tree $ \mathcal F$, that $ \lvert  \varepsilon _J \rvert \lesssim 1 $. 
We perform part of the exchange argument. 
In the bilinear form, the argument of the Riesz transform is $ Q _{IJ}= Q_F - (Q _{F} - Q _{IJ})$. 
With just $ Q_F$, we have 
\begin{align*}
\sum_{I \::\: \pi _{\mathcal F}I=F} 
\sum_{\substack{J \::\: \tilde \pi _{\mathcal F} J =F\\ I_J \not\in \mathcal F, \  J\Subset_{4r} I }}  
\mathbb E ^{\tau } _{Q _{IJ}} \Delta ^{\tau } _{Q_I} g 
\Bigl\langle  {\mathsf R}_{\tau }^{\ast }  Q_F ,  \Delta ^{\sigma } _{J} f\Bigr\rangle _{\sigma } 
& =  \mathbb E^{\tau}_{Q_F}\lvert g\rvert\Bigl\langle  {\mathsf R}_{\tau }  ^{\ast}Q_F , \sum_{J \::\: \tilde \pi _{\mathcal F} J=F}  \varepsilon_J  \Delta^\sigma_J f\Bigr\rangle_{\sigma } 
\\
& \le \mathbb E^{\tau}_{Q_F}\lvert g\rvert 
\left\Vert  Q_F{\mathsf R}_{\tau } ^{\ast}  Q_F \right\Vert_{\sigma } 
\left\Vert  
\sum_{J \::\: \tilde \pi _{\mathcal F} J=F} \varepsilon_J  \Delta^\sigma_J f
\right\Vert_{\sigma } 
\\
& \le \mathscr T \mathbb E^{\tau}_{Q_F}\lvert g\rvert \cdot\tau (Q_F) ^{1/2} \Vert \tilde H^{\sigma } _{F} f\Vert_{\sigma} . 
\end{align*}
This uses the testing inequality and orthogonality of the martingale differences.  
Quasi-orthogonality is used to sum this last estimate over all $ F\in \mathcal F$.  

\smallskip
It remains to bound the term where the argument of the Riesz transform is $ (Q _{F} - Q _{IJ})$, that is the \emph{stopping form}
\begin{equation}\label{e:stopI}
B ^{\textup{stop}} _F (f,   g)
\equiv 
\sum_{I \::\: \pi _{\mathcal F}I=F} 
\mathbb E ^{\tau } _{Q _{IJ}} \Delta ^{\tau } _{Q_I} g \cdot 
\sum_{\substack{J \::\: \tilde \pi _{\mathcal F} J =F\\ I_J \not\in \mathcal F, \  J\Subset_{4r} I }}  
\left\langle {\mathsf R} ^{\ast}   _{\tau } (Q _{F} - Q _{IJ}),  \Delta^\sigma_J f\right\rangle _{\sigma } . 
\end{equation}
No naive application of the remaining part of the exchange argument will be successful.  
Instead, we will use a sophisticated recursion to prove the following lemma.

\begin{lemma}\label{l:stopI} The following estimate holds uniformly over $ F\in \mathcal F$: 
\begin{equation}\label{e:stopI<}
\left\vert  B ^{\textup{stop}} _F (f,   g)\right\vert 
\lesssim \mathscr R \lVert g\rVert_{\tau } \lVert f\rVert_{ \sigma } . 
\end{equation}
\end{lemma}
One then can invoke quasi-orthogonality to sum the estimate above since it is applied to the corresponding Haar projections of $f$ and $g$, namely, to $\tilde H^{\sigma } _{F} f$ and $H^{\tau} _{F} g$ respectively.

\subsection{The Initial Definitions and the Size Lemma}
We regard the interval $ F\in \mathcal F$ as fixed. 
We need an elaborate decomposition of the bilinear form $ B ^{\textup{stop}} _F (f, g)$,  by dividing up the  intervals $ I, J$ over which the sum is formed. 

Let $ \mathcal P$ be a finite  collection of pairs  $ (I,J) \in \mathcal D _g  \times \mathcal D_f $. 
Denote a generic element by $  (P_1, P_2)\in \mathcal P$.  Set $ \mathcal P _{j} \equiv \{ P_j \::\: (P_1,P_2) \in \mathcal P\}$ to be the projection onto the respective coordinate.  Also, let $ \tilde {\mathcal P}_1 \equiv \{ (P_1) _{P_2} \::\: (P_1,P_2) \in \mathcal P \}$.  

We say that $ \mathcal P$ is \emph{admissible} if 
\begin{enumerate}
\item For $ (P_1, P_2) \in \mathcal P$, we have $ P_2 \Subset _{4r} P_1 \subsetneq F$. 

\item For $ (P_1, P_2) \in \mathcal P$, there holds $ P_2 \Subset _{4r} P_1$,  
and  both  $ P_1$ and $ P_2$ are good. 

\item  Let $ F'$ be an energy stopping interval of $ F$,  as defined at the beginning of \S\ref{s:stopData}.  
No interval in $   \tilde {\mathcal P}_1 $ is    contained in  $ F'$. 

 \item  For each fixed $Q$, 
 the collection $ \{P_1\::\: (P_1, Q)\in \mathcal P\}$ is \emph{convex}.  Namely, if 
 $ P_1 \subset P_2 \subset P_3 $ and $P_1, P_3 \in \mathcal P_1$, and $ P_2\in \mathcal D _{g} $ is good, then $ P_2\in \mathcal P_1$.
 
\end{enumerate}

A remark that we will use repeatedly is that a convex subcollection  $ \mathcal P'$ of an admissible $ \mathcal P$ is again admissible. 
 Set 
\begin{equation*}
B _{\mathcal P} (f,g) \equiv \sum_{(P_1, P_2) \in \mathcal P}
 \mathbb E ^{\tau } _{\tilde P_1} \Delta ^{\tau } _{P_1} g \cdot 
\left\langle {\mathsf R}  _{\tau }^{\ast} (Q _{F} - Q _{ \tilde P_1}),  \Delta ^{\sigma } _{P_2} f\right\rangle _{\sigma } ,
\end{equation*}
It suffices to consider admissible $ \mathcal P$.

\begin{proposition}\label{p:stopAdmiss}  We have for an admissible collection $ \mathcal P$, 
\begin{equation*}
B ^{\textup{stop}} _F (f,   g) = B  _{\mathcal P} (f,   g)
\end{equation*}
\end{proposition}

\begin{proof}
From the definition of the stopping form in \eqref{e:XS}, we take 
\begin{equation*}
\mathcal P = \{ ( P_1, P_2) \;:\;  \textup{$ P_1, P_2$ good},\  \pi _{\mathcal F} P_2 = \pi _{\mathcal F} P_1 = F,\  P_2\Subset _{4r} P_1\}
\end{equation*}
This collection is admissible.   The first condition is immediate.  The second condition follows from $ \pi _{\mathcal F} P_2=F$ and $ P_2 \Subset _{4r} P_1$. And the third condition follows from the definition of $ \mathcal F$-parent. 
\end{proof}

Let $ \mathscr N _{\mathcal P}$ be the norm of the bilinear form $B _{\mathcal P}$, that  is, $ \mathscr N _{\mathcal P}$ is  the best constant in the inequality 
\begin{equation*}
\lvert  B _{\mathcal P} (f,g) \rvert \le \mathscr N _{\mathcal P} \lVert g\rVert_{\tau } \lVert f\rVert_{\sigma } .  
\end{equation*}
Thus, it suffices to show that $ \mathscr N _{\mathcal P} \lesssim \mathscr R$ for all admissible $ \mathcal P$. 

With an abuse of language, we say that admissible collections $ \mathcal P ^{t}$, $ t\ge 0$, are \emph{orthogonal} if  
for any $ s\neq t$, then $ \mathcal P _2^{s} \cap \mathcal P _2^{t} =\emptyset $ and $ \tilde {\mathcal P} ^{s}_1 \cap \tilde {\mathcal P} ^{t}_1 = \emptyset $.  
A simple lemma is then an estimate of the norm of a form which is the union of orthogonal collections.  

\begin{lemma}\label{l:Sortho} Given orthogonal admissible collections $ \mathcal P ^{t}$, $ t\ge 0$,
\begin{equation*}
\mathscr N _{\bigcup _{t} \mathcal P ^{t}}  \le \sqrt 2 \sup _{t} \mathscr N _{\mathcal P ^{t}} . 
\end{equation*}
\end{lemma}

\begin{proof}
Let $ \Pi _{t}  ^{\sigma }$ be the Haar projection in $ L ^2 (\mathbb R ; \sigma ) $ onto the span of the functions $\{ h ^{\sigma } _{J} \::\: J\in \mathcal P ^{t}_2\}$. 
And, let $ \Pi _t ^{\tau }$ 
be the Haar projection in $ L ^2 ( \mathbb R ^2 _+; \tau  ) $ onto the span of the functions $\{ h ^{\tau  } _{Q_I} \::\: I\in \mathcal P ^{t}_1\}$.  

The projections $ \Pi _{t}  ^{\sigma }$ are indeed orthogonal in $ L ^2 (\mathbb R ;\sigma )$. 
Concerning the projections $ \Pi _t ^{\tau }$, note that an interval $ I$ has two children. The children can be in a  collection    $ \tilde {\mathcal P} ^{t}_1$ for  two distinct choices of $ t$.
Therefore, we have 
\begin{equation*}
\sum_{t} \lVert \Pi ^{\tau  }_t g\rVert_{\tau  } ^2 \le 2 \lVert g\rVert_{\sigma } ^2 \,. 
\end{equation*}
The Lemma is completed by estimating 
\begin{align*}
\lvert B _{\bigcup _{t} \mathcal P ^{t}} (f,g)\rvert &=  
\Bigl\lvert  \sum_{t} B _{\mathcal P ^{t}} (f,g) \Bigr\rvert_{} 
\\
&\le 
  \sum_{t} \lvert B _{\mathcal P ^{t}} ( \Pi ^{\sigma  } _{t} f,  \Pi ^{\tau  }_t g) \rvert
  \\
 & \le 
   \sum_{t} \mathscr N _{\mathcal P ^{t}} \left\lVert  \Pi ^{\tau } _{t} g\right\rVert_{\tau }\left \lVert \Pi ^{\sigma }_t f\right\rVert_{ \sigma } 
   \\
   &\le \sup _{t} 
   \mathscr N _{\mathcal P ^{t}}  \times 
   \Bigl[ \sum_{t} \left\lVert  \Pi ^{\tau } _{t} g\right\rVert_{\tau } ^2 
   \times \sum_{t}\left\lVert \Pi ^{\sigma }_t f\right\rVert_{\sigma } ^2    \Bigr] ^{1/2} 
   \\
   &\le \sqrt 2 \sup _{t} 
   \mathscr N _{\mathcal P ^{t}}  \cdot   \lVert g\rVert_{\tau } \lVert f\rVert_{\sigma } . 
\end{align*}
\end{proof}

The critical notion of \emph{size} serves as a crude approximation to the
norm of the bilinear form $B_{\mathcal{P}}$, and it has to be defined with
some care. Set 
\begin{gather}
\lambda =\lambda _{\mathcal{P}}\equiv \sum_{\substack{ P_{2}\in \mathcal{P}%
_{2}}}\left\langle t,h_{P_{2}}^{\sigma }\right\rangle _{\sigma }^{2}\delta
_{x_{Q_{P_{2}}}},  \label{e:size} \\
\textup{size}(\mathcal{P})^{2}\equiv \sup_{I\in \tilde{\mathcal{P}}%
_{1}:\tau (Q_{I})>0}\tau (Q_{I})^{-1}\sum_{K\in \mathcal{W}\!{I}}\mathsf{T}%
_{\tau }(Q_{F}\setminus Q_{K})(x_{Q_{K}})^{2}\frac{\lambda (\textup{Saw}_{%
\mathcal{P}}K)}{\lvert K\rvert ^{2}}\,, \\
\textup{Saw}_{\mathcal{P}}I\equiv \bigcup \{x_{Q_{P_{2}}}:P_{2}\in 
\mathcal{P}_{2}:P_{2}\Subset _{r}I\}.
\end{gather}

The measure $\lambda $ is derived from the energy terms, but we do not have
scale modifications above. The latter reappear in the definition of size,
which also uses a (not quite standard) sawtooth-type definition associated
to Carleson measure estimates. Besides the sawtooth region being a discrete
set, we caution the reader that (a) we do not have Carleson measures in this
argument, and (b) the distinction between this definition and the standard
definition is important, as it stems from the formulation of the energy
stopping condition. The next proposition is the base step in our recursion. 
\begin{proposition}\label{p:size}  There holds for admissible $ \mathcal P$, 
$ \textup{size} (\mathcal P) \lesssim \mathscr R$.  
\end{proposition}


\begin{proof}
We argue by contradiction. In the definition of size, namely \eqref{e:size},
we have 
\begin{align*}
\mathsf{T}_{\tau } (Q_F \setminus Q_K)(x_{Q_K} ) ^2 \frac{ \lambda ( 
\textup{Saw}_ {\mathcal{P}}K)}{ \lvert K\rvert ^2 } \leq \mathsf{T}_{\tau }
(Q_F \setminus Q_K)(x_{Q_K} )^2E (\sigma , K) ^2 \sigma (K).
\end{align*}
Above, we are using the notion of energy defined in \eqref{e:eng-sigma}.
Thus, if the conclusion does not hold, we have for some $I\in \tilde {%
\mathcal{P}}_1$, we have 
\begin{equation*}
C \mathscr R ^2 \tau (Q _{I}) \leq \sum_{K\in \mathcal{W}\!{I}} \mathsf{T}%
_{\tau } (Q_F \setminus Q_K)(x_{Q_K} ) ^2 E (\sigma , K) ^2 \sigma (K).
\end{equation*}
Above, we can take the constant $C$ as large as we wish. In view of the
selection of stopping data in \S \ref{s:stopData}, we see that $I$ must be
in $\mathcal{F}$, which is a contradiction. 
\end{proof}


Our proof will show that the norm of the stopping form associated to $%
\mathcal{P}$ satisfies $\mathscr N _{\mathcal{P}} \lesssim \textup{size} (%
\mathcal{P})$. The main lemma in the proof of this fact is this.

\begin{lemma}[Size Lemma]
\label{l:size}
Any admissible collection $ \mathcal P$ admits a decomposition into collections $ \mathcal P ^{\textup{big}} \cup \mathcal P ^{\textup{small}}$, so that on the one hand, 
$
\mathscr N _{ \mathcal P ^{\textup{big}} } \lesssim \textup{size} (\mathcal P) 
$, 
and on the other, $ \mathcal P ^{\textup{small}} $  is the union of  admissible 
orthogonal collections $ \bigcup _{t\ge 0}\mathcal P _{ t} ^{\textup{small}}$,    with 
\begin{equation}\label{e:small}
\sup _{t\ge 0} \textup{size} (\mathcal P ^{\textup{small}} _{t}) \le \tfrac 14 \textup{size} (\mathcal P).  
\end{equation}
\end{lemma}


\begin{proof}[Proof of Lemma \protect\ref{l:stopI}]
The form $B ^{\textup{stop}} _F (f, g) = B _{\mathcal{P}_0} (f,g) $, for an
admissible collection $\mathcal{P}_0$ which then satisfies $\textup{size} (%
\mathcal{P}_0) \lesssim \mathscr R$. By recursive application of the Size
Lemma, we have that $\mathcal{P }_{0} $ is the union of collections $%
\mathcal{P }_{n, t}$, where (a) $n, t \ge 1$, (b) for fixed $n\ge 1 $ the
collections $\{\mathcal{P }_{n,t} \::\: t\ge 1\}$ are admissible and
orthogonal, and (c) we have the inequality 
\begin{equation*}
\mathscr N _{ \mathcal{P }_{n,t} } \lesssim 2 ^{-2n} \mathscr R , \qquad
n,t\ge 0.
\end{equation*}
In view of Lemma~\ref{l:Sortho}, this gives us a proof of $\mathscr N _{%
\mathcal{P}_0} \lesssim \mathscr R$, completing the proof.
\end{proof}



\subsection{Main Construction}

We only know how to directly estimate $\mathcal{N }_{\mathcal{P}}$ if there
is some `decoupling' between the Haar coefficient $\langle f, h ^{\sigma }
_{P_2} \rangle_{\sigma}$, and the argument of the Riesz transform $Q _{F}
\setminus Q _{\tilde P_1}$.

The tool to achieve the decoupling, and so the decomposition of the Size
Lemma, is the collection $\mathcal{L}=\bigcup_{t=0}^{\infty }\mathcal{L}_{t}$%
, the latter defined recursively. Set $\mathbf{S}\equiv \textup{size}%
(\mathcal{P})$, and take $\mathcal{L}_{0}$ to be the minimal intervals $L\in \tilde{%
\mathcal{P}}_{1}$ such that 
\begin{equation}
\sum_{J\in \mathcal{W}\!L}\mathsf{T}_{\tau }(Q_{F}\setminus
Q_{J})(x_{Q_{J}})^{2}\frac{\lambda (\textup{Saw}_{\mathcal{P}}J)}{\lvert
J\rvert ^{2}}\geq c\mathbf{S}^{2}\tau (Q_{L}).  \label{e:0}
\end{equation}%
The constant $0<c<1$ will be sufficiently small, but absolute. There must be
such intervals, by the definition of size in \eqref{e:size}, and the minimal
intervals exist since $\tilde{\mathcal{P}}_{1}$ is a finite collection.
Then, for $t>0$, inductively define $\mathcal{L}_{t}$ to be the minimal
intervals $L\in \tilde{\mathcal{P}}_{1}$ such that 
\begin{equation}
\lambda (\textup{Saw}_{\mathcal{P}}L)\geq (1+c)\sum_{\substack{ L^{\prime
}\in \mathcal{L}_{t-1} \\ L^{\prime }\subsetneq L}}\lambda (\textup{Saw}_{%
\mathcal{P}}{L^{\prime }}).  \label{e:tent}
\end{equation}%
See Figure~\ref{f:tents}.

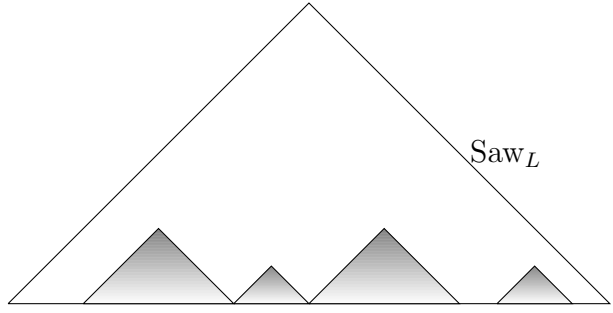
\begin{figure}[tbp]
\begin{tikzpicture}
 \draw (0,0) -- (4,4) -- (8,0)  node[right,midway] () {$ \textup{Saw}_L$} -- (0,0); 
 \shadedraw (1,0) -- (2,1) -- (3,0) -- (1,0); 
  \shadedraw (3,0) -- (3.5,.5) -- (4,0) -- (3,0); 
   \shadedraw (4,0) -- (5,1) -- (6,0) -- (4,0);  
   \shadedraw (6.5,0) -- (7,.5) -- (7.5,0) -- (6.5,0);  
 \end{tikzpicture}
\caption{The shaded smaller tents have been selected, and $\textup{Saw}_L$
is the minimal sawtooth region with $\protect\lambda (\textup{Saw}_L)$
larger than $1+ c $ times the $\protect\lambda $-measure of the shaded tents.
}
\label{f:tents}
\end{figure}

\smallskip

The collection $\mathcal{P }_{\textup{small}}$ is the union of the
following collections. First, set 
\begin{equation*}
\mathcal{P }_{\textup{small}} ^{0} \equiv \{ (P_1, P_2 ) \in \mathcal{P }%
\::\: \textup{$\tilde P_1$ does not have a parent in $\mathcal{L}$}\}.
\end{equation*}
And, for $L\in \mathcal{L}$, set 
\begin{gather}  \label{e:ddot}
\mathcal{P }_{\textup{small}, L} \equiv \{ (P_1, P_2) \::\: \ddot\pi _{%
\mathcal{L}}P_2 = \pi _{\mathcal{L}} \tilde P_1=L\,, \tilde P_1 \subsetneq
L\}, \\
\ddot\pi _{\mathcal{L}}P_2 \quad \textup{is the minimal element of $L\in 
\mathcal{L}$ with $P_2 \Subset_r L$.}
\end{gather}

Note that in the `small' collections, there is no `decoupling' between $P_2$
and $Q_F \setminus Q _{\tilde P_1}$. We check that the `small' collection
meets the required conditions.


\begin{proof}[Proof of \eqref{e:small}]
The collections $\mathcal{P }_{\textup{small}} ^{0} $ and $\{ \mathcal{P }_{%
\textup{small}, L} \,:\, L \in \mathcal{L}\}$ are convex subsets of an
admissible collection. Hence they are admissible. The collections are
orthogonal, since a given interval contained in some element of $\mathcal{L}$%
, is contained in a minimal element of $\mathcal{L}$.

We verify that the collections above have small size. For $\mathcal{P }_{%
\textup{small}} ^{0}$, since each interval $\tilde P_1$ must fail %
\eqref{e:0}, the size of this collection is smaller by the factor $c$. We
turn to $\mathcal{P }_{\textup{small}, L} $. In the case that $L\in 
\mathcal{L}_0$, this follows immediately from the definition of $\mathcal{L}%
_0$ in \eqref{e:0}. In the case that $L\in \mathcal{L}$ is not minimal, then 
$L$  has $\mathcal{L}$-children, which will appear below. We have for $I\in (%
\tilde {\mathcal{P}}_{\textup{small}, L})_1 $, and $\mathcal{L}_I \equiv \{
L^{\prime }\in \mathcal{L }\::\: \pi _{\mathcal{L}} ^{1}L^{\prime }=L,\
L^{\prime }\subset I\}$ that 
\begin{align*}
{\lambda ( \textup{Saw}_ {\mathcal{P }_{\textup{small}, L} }I)} &\le {%
\lambda \bigl( \textup{Saw}_ {\mathcal{P}} I \setminus \bigcup_{L^{\prime
}\in \mathcal{L}_I}\textup{Saw}_ {\mathcal{P}}{L^{\prime }}\bigr)} \\
& \le c {\ \sum_{L^{\prime }\in \mathcal{L}_I }\lambda ( \textup{Saw}_ {%
\mathcal{P}} {L^{\prime }})} \le c {\ \lambda ( \textup{Saw}_ {\mathcal{P}}
I)} .
\end{align*}
This uses \eqref{e:tent}, the selection rule for $\mathcal{L}$, and the fact
that in \eqref{e:ddot}, we exclude the case that $\tilde P_1 = L$. But
then, it follows directly from the definition of size in \eqref{e:size} that 
$\textup{size} (\mathcal{P }_{\textup{small}, L}) < \frac 14 \mathbf{S}$,
provided $0< c < 1/4$.
\end{proof}



\subsection{The Big Collections}

Note that if $(P_1, P_2) \in \mathcal{P}$ is such that $\tilde P_1$ has
parent $L\in \mathcal{L}$, then we must have $P_2 \Subset _{4r} \pi L$.
Thus, $\ddot \pi _{\mathcal{L}} P_2$ is well-defined, see \eqref{e:ddot}.
And for some integer $t\ge 1$, necessarily $\ddot\pi ^{t+1} _{\mathcal{L}}
P_2= \pi _{\mathcal{L}} \tilde P_1=L $. The last notation is defined here: $%
\ddot \pi _{\mathcal{L}} P_2$ is defined in \eqref{e:ddot}, $\ddot\pi ^{1} _{%
\mathcal{L}} P_2 \equiv \ddot\pi _{\mathcal{L}} P_2 $ and inductively define 
$\ddot\pi ^{t+1} _{\mathcal{L}} P_2$ to be the minimal member of $\mathcal{L}
$ that \emph{strictly} contains $\ddot\pi ^{t} _{\mathcal{L}} P_2$.

In order that $(P_1, P_2) \in \mathcal{P }^{\textup{big}} = \mathcal{P }%
\setminus \mathcal{P }^{\textup{small}}$ we must have that $\tilde P_1 $
has a parent in $\mathcal{L}$, and either $\tilde P_1 \in \mathcal{L}$ or $%
\ddot\pi _{\mathcal{L}} P_2 \subsetneq \tilde P_1$. Thus, $\mathcal{P }%
\setminus \mathcal{P }^{\textup{small}}$ is the union over $L\in \mathcal{L}
$, and $t\geq 1$ of the collections 
\begin{align}  \label{e:zptl}
\mathcal{P }^{1} _{L} &\equiv \{ (P_1, P_2) \in \mathcal{P }_{L} \::\:
\ddot\pi ^{1} _{\mathcal{L}} P_2=\tilde P_1=L\}, \\
\mathcal{P}^t_L & \equiv \{ (P_1, P_2) \in \mathcal{P }\::\: \ddot\pi _{%
\mathcal{L}} ^{t} P_2 = \pi _{\mathcal{L}}\tilde P_1 =L\}, \qquad t > 1.
\end{align}
For fixed $t$, these collections are mutually orthogonal. The norms of these
collections satisfy the estimate below.

\begin{lemma}\label{l:norm-t} For $ L\in \mathcal L$ and $ t \ge 1$, there holds 
$
\mathscr N_{\mathcal P^t_L}  \lesssim (1+c) ^{-t/2} \mathbf S 
$.
\end{lemma}

It follows immediately that 
\begin{equation*}
\sum_{t=1} ^{\infty } \mathscr N _{\bigcup _{L\in \mathcal{L}}\mathcal{P}%
^t_L } \lesssim \mathbf{S }\sum_{t=1} ^{\infty } (1+c ) ^{-t/2} \lesssim 
\mathbf{S }.
\end{equation*}
The proof of the `Size Lemma' \ref{l:size} will be complete after proving
Lemma~\ref{l:norm-t}.


\begin{proof}[Proof of Lemma~\protect\ref{l:norm-t}, $t=1$]
There is a single choice of $\tilde{P}_{1}$, namely $L$. For each $K\in 
\mathcal{W}\!L$, we can estimate by the definition of $\mathcal{W}\!L$, and
monotonicity \eqref{e:MONO}, that 
\begin{align*}
\Biggl\lvert\sum_{\substack{ (P_{1},P_{2})\in \mathcal{P}_{L}^{1} \\ %
P_{2}\subset K}}& \mathbb{E}_{Q_{L}}^{\tau }{\Delta _{P_{1}}^{\tau }g}\cdot
\left\langle \mathsf{R}_{\tau }^{\ast }(Q_{F}-Q_{L}),\Delta _{P_{2}}^{\sigma
}f\right\rangle _{\tau }\Biggr\rvert \\
& \leq \sum_{\substack{ P_{2}\in (\mathcal{P}_{L}^{1})_{2} \\ P_{2}\subset K
}}\Biggl\lvert\sum_{P_{1}\in (\mathcal{P}_{L}^{1})_{1}}\mathbb{E}%
_{Q_{L}}^{\tau }{\Delta _{P_{1}}^{\tau }}g\cdot \left\langle \mathsf{R}%
_{\tau }^{\ast }(Q_{F}-Q_{L}),\Delta _{P_{2}}^{\sigma }f\right\rangle _{\tau
}\Biggr\rvert \\
& \leq \bigl\lvert\mathbb{E}_{Q_{L}}^{\tau }{\Delta _{P_{1}}^{\tau }g}%
\bigr\rvert\mathsf{T}_{\tau }(Q_{F}\setminus
Q_{K})(x_{Q_{K}})\sum_{P_{2}:P_{2}\subset K}\left\langle \frac{t}{\lvert
K\rvert },h_{P_{2}}^{\sigma }\right\rangle _{\sigma }\lvert \hat{f}_{\sigma
}(P_{2})\rvert  \\
& \leq \mathbf{S}\bigl\lvert\mathbb{E}_{Q_{L}}^{\tau }{\Delta
_{P_{1}}^{\tau }g}\bigr\rvert\tau (Q_{K})^{1/2}\left[ \sum_{P_{2}:P_{2}%
\subset K}\hat{f}_{\sigma }(P_{2})^{2}\right] ^{1/2}.
\end{align*}%
Above, we use Cauchy--Schwarz in $P_{2}$, and importantly, the definition of
size to gain the factor $\mathbf{S}\tau (Q_{K})^{1/2}$ above. It is clear
that quasi-orthogonality permits the bound 
\begin{align*}
\mathbf{S}\bigl\lvert\mathbb{E}_{Q_{L}}^{\tau }{\Delta _{P_{1}}^{\tau }g}%
\bigr\rvert& \sum_{K\in \mathcal{W}\!L}\tau (Q_{K})^{1/2}\left[
\sum_{P_{2}:P_{2}\subset K}\hat{f}_{\sigma }(P_{2})^{2}\right] ^{1/2} \\
& \lesssim \mathbf{S}\bigl\lvert\mathbb{E}_{Q_{L}}^{\tau }{\Delta
_{P_{1}}^{\tau }g}\bigr\rvert\cdot \tau (Q_{L})^{1/2}\lVert f\rVert _{\sigma
}\lesssim \mathbf{S}\lVert g\rVert _{\tau }\lVert f\rVert _{\sigma }.
\end{align*}%
The proof is complete.
\end{proof}



\begin{proof}[Proof of Lemma~\protect\ref{l:norm-t}, $t> 1$]
There are two stages of the proof, first to show that the size of the
collections $\mathcal{P}^t_L $ decay exponentially in $t$, and second that
the size dominates the norm of the bilinear form.

Firstly, we verify that $\textup{size} (\mathcal{P}^t_L )$ has exponential
decay in $t$, with the key inequality being \eqref{e:tent}, in the
construction of $\mathcal{L}$. For $I \in (\tilde {\mathcal{P}^t_L})_1$, we
have by backwards induction, 
\begin{align*}
\lambda (\textup{Saw} _{{\mathcal{P}}^{t} _{L}} I) &\le (1+ c ) ^{-1} \sum
_{\substack{ L^{\prime }\in \mathcal{L }\::\: L^{\prime }\subset I  \\ \pi
^{t-1} _{\mathcal{L}}L^{\prime }=L}}\lambda (\textup{Saw}_{\mathcal{P}}{%
L^{\prime }}) \\
&\;\;\vdots \\
& \le (1+ c ) ^{-t} \sum_{\substack{ L^{\prime }\in \mathcal{L }\::\:
L^{\prime }\subset I  \\ \pi ^{1} _{\mathcal{L}}L^{\prime }=L}} \lambda (%
\textup{Saw}_{\mathcal{P}}{L^{\prime }}) \le (1+ c ) ^{-t} \lambda (%
\textup{Saw}_ {\mathcal{P}} I).
\end{align*}
But, if the sawtooth regions have small size, we then have immediately from
the definition of size in \eqref{e:size} that 
\begin{equation}  \label{e:st}
\textup{size} (\mathcal{P}^t_L ) \lesssim (1+ c ) ^{-t/2}\, \mathbf{S}.
\end{equation}
This is the required exponential decay. \bigskip

Secondly, turn to the estimation of the norm $\mathscr N_{\mathcal{Q}} $,
where $\mathcal{Q}= \mathcal{P}^t_L$. Namely that $\mathscr N_{\mathcal{Q}}
\lesssim (1+ c ) ^{-t/2} \mathbf{S}$. We are free to assume that the Haar
support of $g$ is contained in $\mathcal{Q}_1$. Construct stopping data $%
\mathcal{G}$ for $g$ in an ordinary way. Namely, take the maximal elements
of $\mathcal{G}$ to be the maximal elements in $\tilde {\mathcal{Q}}_1$. In
the recursive step, if $I\in \mathcal{G}$ is minimal, add to $\mathcal{G}$
the maximal children $I^{\prime }\subsetneq I$ such that $I^{\prime }\in 
\tilde {\mathcal{Q}}_1$ and $\mathbb{E }^{\tau } _{Q_{I^{\prime }}} \lvert
g\rvert \ge 10 \mathbb{E }^{\tau } _{Q_I} \lvert g\rvert $.

Let $\mathcal{S}$ be the $\mathcal{L}$-children of $L$. Define $\ddot{\pi}_{%
\mathcal{G}}P_{2}$ to be the minimal element $G\in \mathcal{G}$ with $%
P_{2}\Subset _{r}G$. Hold an interval $G\in \mathcal{G}$ fixed. Then, define 
\begin{align}
\Xi (G)& \equiv \Biggl\lvert\sum_{\substack{ (P_{1},P_{2})\in \mathcal{Q} \\ 
\ddot{\pi}_{\mathcal{G}}P_{2}=G}}\mathbb{E}_{Q_{\tilde{P}_{1}}}^{\sigma }{%
\Delta _{Q_{P_{1}}}^{\tau }}g\cdot \left\langle \mathsf{R}_{\tau }^{\ast
}(Q_{F}-Q_{\tilde{P}_{1}}),\Delta _{P_{2}}^{\sigma }f\right\rangle _{\tau }%
\Biggr\rvert \\
\noalign{\noindent but, as $ t >1$, each $ P_2$ satisfies $ P_2 \Subset_r S$
for some $ S\in \mathcal S$, so that we can sum over $S\in \mathcal S$, }& =%
\Biggl\lvert\sum_{S\in \mathcal{S}}\sum_{\substack{ (P_{1},P_{2})\in 
\mathcal{Q} \\ \ddot{\pi}_{\mathcal{G}}P_{2}=G,\ P_{2}\Subset _{r}S}}\mathbb{%
E}_{Q_{\tilde{P}_{1}}}^{\sigma }{\Delta _{Q_{P_{1}}}^{\tau }}g\cdot
\left\langle \mathsf{R}_{\tau }^{\ast }(Q_{F}-Q_{\tilde{P}_{1}}),\Delta
_{P_{2}}^{\sigma }f\right\rangle _{\tau }\Biggr\rvert.
\end{align}%
By convexity in $P_{1}$ and the definition of the stopping intervals $G$,
observe that for each $P_{2}$, the argument of the Riesz transform above is 
\begin{equation}
\Biggl\lvert\sum_{\substack{ P_{1}\,:\,(P_{1},P_{2})\in \mathcal{Q} \\ \ddot{%
\pi}_{\mathcal{G}}P_{2}=G,\ P_{2}\Subset _{r}S}}\mathbb{E}_{Q_{\tilde{P}%
_{1}}}^{\sigma }{\Delta _{Q_{P_{1}}}^{\tau }}g\cdot (Q_{F}-Q_{\tilde{P}_{1}})%
\Biggr\rvert\lesssim \mathbb{E}_{G}^{\tau }\lvert g\rvert \cdot
(Q_{F}\setminus Q_{S}).
\end{equation}%
The monotonicity principle \eqref{e:MONO} applies, yielding the inequality
below. This is the decoupling step. 

We continue with the estimate of $ \Xi (G)$ as follows. 
\begin{align}
\Xi (G)& \lesssim \mathbb{E}_{G}^{\tau }\lvert g\rvert \cdot \sum_{S\in 
\mathcal{S}}\mathsf{T}_{\tau }(Q_{F}\setminus Q_{S})(x_{Q_{S}})\sum
_{\substack{ P_{2}:\ddot{\pi}_{\mathcal{G}}P_{2}=G \\ P_{2}\Subset S}}%
\Bigl\langle\frac{t}{\lvert S\rvert },h_{P_{2}}^{\sigma }\Bigr\rangle%
_{\sigma }\lvert \hat{f}_{\sigma }(P_{2})\rvert , \\
& =\mathbb{E}_{G}^{\tau }\lvert g\rvert \cdot \sum_{S\in \mathcal{S}}\mathsf{%
T}_{\tau }(Q_{F}\setminus Q_{S})(x_{Q_{S}})\sum_{K\in \mathcal{W}\!S}\sum
_{\substack{ P_{2}:\ddot{\pi}_{\mathcal{G}}P_{2}=G \\ P_{2}\Subset K}}%
\Bigl\langle\frac{t}{\lvert S\rvert },h_{P_{2}}^{\sigma }\Bigr\rangle%
_{\sigma }\lvert \hat{f}_{\sigma }(P_{2})\rvert , \\
& \lesssim \mathbb{E}_{G}^{\tau }\lvert g\rvert \cdot \sum_{S\in \mathcal{S}}%
\mathsf{T}_{\tau }(Q_{F}\setminus Q_{S})(x_{Q_{S}})
\Biggr[ 
\sum_{K\in \mathcal{W}\!S}\sum_{\substack{ P_{2}:\ddot{\pi}_{\mathcal{G}}P_{2}=G \\ %
P_{2}\Subset K}}
\Bigl\langle\frac{t}{\lvert S\rvert },h_{P_{2}}^{\sigma }\Bigr\rangle ^{2} 
\times  
\sum_{K\in \mathcal{W}\!S}\sum_{\substack{ P_{2}:\ddot{\pi}_{%
\mathcal{G}}P_{2}=G \\ P_{2}\Subset K}}\lvert \hat{f}_{\sigma }(P_{2})\rvert
^{2}\Biggr] ^{\frac{1}{2}}
\end{align}%
and then using $\mathsf{T}_{\tau }(Q_{F}\setminus Q_{S})(x_{Q_{S}})\lesssim 
\mathsf{T}_{\tau }(Q_{F}\setminus Q_{S})(x_{Q_{K}})$,\ we continue with%
\begin{align}
& \lesssim \mathbb{E}_{G}^{\tau }\lvert g\rvert \cdot \sum_{S\in \mathcal{S}%
} \Biggl[ 
\sum_{K\in \mathcal{W}\!S}\mathsf{T}_{\tau }(Q_{F}\setminus
Q_{S})(x_{Q_{K}})^{2}\sum_{\substack{ P_{2}:\ddot{\pi}_{\mathcal{G}}P_{2}=G
\\ P_{2}\Subset K}}
\Bigl\langle\frac{t}{\lvert S\rvert },h_{P_{2}}^{\sigma }\Bigr\rangle^{2}
\times 
\sum_{K\in \mathcal{W}\!S}\sum_{\substack{ P_{2}:\ddot{\pi}_{%
\mathcal{G}}P_{2}=G \\ P_{2}\Subset K}}\lvert \hat{f}_{\sigma }(P_{2})\rvert
^{2}
\Biggr] ^{\frac{1}{2}} 
\\
& \lesssim \textup{size} (\mathcal{P}^t_L ) \cdot 
\mathbb{E}_{G}^{\tau }\lvert g\rvert \cdot \sum_{S\in \mathcal{S}%
} \tau \left( S\right)  ^{\frac{1}{2}}
\Biggl[
\sum_{K\in \mathcal{W}\!S}\sum_{\substack{ P_{2}:\ddot{\pi}_{\mathcal{G}%
}P_{2}=G \\ P_{2}\Subset K}}\lvert \hat{f}_{\sigma }(P_{2})\rvert
^{2}
\Biggr] ^{\frac{1}{2}} \\
& \leq   \textup{size} (\mathcal{P}^t_L )  \mathbb{E}_{G}^{\tau }\lvert g\rvert \cdot 
\Biggl[
\sum_{S\in \mathcal{S}}\tau \left( S\right) 
\times 
\sum_{S\in \mathcal{S}}\sum_{K\in \mathcal{W}\!S}\sum_{\substack{ P_{2}:%
\ddot{\pi}_{\mathcal{G}}P_{2}=G \\ P_{2}\Subset K}}\lvert \hat{f}_{\sigma
}(P_{2})\rvert ^{2}
\Biggr] ^{\frac{1}{2}}
\\
& \lesssim  \textup{size} (\mathcal{P}^t_L ) 
\cdot \mathbb{E}_{G}^{\tau }\lvert g\rvert \cdot \tau (Q_{G})^{1/2}\lVert
\Pi _{G}^{\sigma }f\rVert _{\sigma }.
\end{align}%
Above, we have appealed to the definition of size in \eqref{e:size}. 
Then, Cauchy-Schwartz in $ S\in \mathcal S$, which are pairwise disjoint intervals contained inside of $ G$.  The last line has the notation 
$\Pi _{G}^{\sigma }\equiv \sum_{P_{2}:\ddot{\pi}_{\mathcal{G}%
}P_{2}=G}\Delta _{P_{2}}^{\sigma }$. Finally, use the quasi-orthogonality
bound \eqref{e:quasi} and the geometric decay size bound \eqref{e:st} to see that 
\begin{equation}
\sum_{G\in \mathcal{G}}\Xi (G)\lesssim (1+c)^{-t/2}\mathbf{S}\,\lVert
g\rVert _{\tau }\lVert f\rVert _{\sigma }.  \label{e:Xi}
\end{equation}%
This completes the proof that $\mathscr N_{\mathcal{Q}}\lesssim
(1+c^{2})^{-t/2}\mathbf{S}$. The Lemma is complete.
\end{proof}

\section{The Form \texorpdfstring{$B ^{\textup{below}}$}{B below}} \label{s:Bbelow}

The analysis of the form $ B ^{\textup{below}} (f,g)$  defined in \eqref{e:below} has an analysis that is similar to that of the 
form $ B ^{\textup{above}}$, due to the close similarity in the two energy inequalities \eqref{e:energyI} and \eqref{e:energyII}. 
We will state  the highlights of the analysis.   

The first step is the construction of stopping intervals $ \mathcal F$.
Recall the formula \eqref{e:HS}, concerning the Haar support of $ f$, and the definition of $ \mathcal D_f^r$. 
The Haar support of $ f$ are intervals $ \pi I$, $ I\in \mathcal D_f^r$ such that $ \pi I$ is good.  

In the initial stage, we take the maximal elements of $ \mathcal F$ to be the children of the maximal intervals in the Haar support of $ f$. 
In the inductive stage, if $ F\in \mathcal F$ is minimal, we add to $ \mathcal F$ the maximal sub-children $ F'\subsetneq F$,  
with $ F'\in \mathcal D_f^r$,  such that either 
\begin{enumerate}
\item (A large average) $ \mathbb E ^{\sigma } _{F'} \lvert  f\rvert \ge 10 \mathbb E ^{\sigma } _{F} \lvert  f\rvert  $,  
\item (Energy Stopping)
$\sum_{K\in \mathcal W\!F'}\mathsf P (\sigma \cdot (F \setminus K), K) ^2 E (\tau ,Q_{K}) ^2 \tau (Q_K) \ge C_0 \mathscr R ^2 \sigma (F')$.    
\\ Recall \eqref{e:Etau}, and Lemma~\ref{l:energyII} and that $ C_0$ is sufficiently large constant.
\end{enumerate}

Define  Haar projections by 
\begin{align} \label{e:HH}
\begin{split}
H _{F} ^{\sigma } f &\equiv \Delta ^{\sigma } _{\pi F} f + \sum_{J \::\: \pi _{\mathcal F}J=F} \Delta ^{\sigma } _{J} f ,
\\
\tilde H _{F} ^{\tau} g &\equiv   \sum_{\substack{I \::\: \pi _{\mathcal F}I=F\\ I \subset F }}  \Delta  ^{\tau} _{Q_I}  g.
\end{split}
\end{align}
We have the quasi-orthogonality inequality, compare to \eqref{e:quasi},   
\begin{equation} \label{e:QQ}
\sum_{F\in \mathcal F} 
\left\{ \mathbb E ^{\sigma } _{F} \lvert  f\rvert \cdot \sigma (F) ^{1/2} + \lVert H ^{\sigma } _{F} f\rVert_{\sigma }\right\} \lVert \tilde H ^{\tau } _{F} g\rVert_{\tau } 
\lesssim \lVert f\rVert_{\sigma } \lVert g\rVert_{ \tau } \,. 
\end{equation}
As before, we will use this inequality as written, and with different choices of the orthogonal projections $ \tilde H ^{\tau } _{F}$.  

This is the global-to-local reduction for $ B ^{\textup{below}} (f,g)$. The sum is over intervals that are `separated' by $ \mathcal F$.  

\begin{lemma}{\textnormal{[Global to Local Reduction, II]}}
\label{l:global2localBelow} The following estimate holds: 
\begin{equation}\label{e:g2lB}
\left\vert 
\sum_{ J} \sum_{\substack{I \::\: I\Subset _{4r} J \\  I\subset \pi _{\mathcal F} I\subsetneq J}} 
\mathbb E ^{\sigma } _{J_I} \Delta ^{\sigma } _{J} f \cdot \left\langle  {\mathsf  R} _{\sigma }  J_I,  \Delta ^{\tau } _{Q_I} g\right\rangle _{\tau }
\right\vert
\lesssim \mathscr R \lVert f\rVert_{\sigma } \lVert g\rVert_{ \tau } . 
\end{equation}
\end{lemma}

\subsection{The Local Estimate}

The last step is to control the local form, namely, to prove the following.

\begin{lemma}\label{l:local2} For each $ F\in \mathcal F$, we have 
\begin{equation}\label{e:local2}
\left\vert  B ^{\textup{below}} ( H^{\sigma } _{F} f,  \tilde H^{ \tau} _{F} g)\right\vert 
\lesssim \mathscr R 
\left\{  \mathbb E ^{\sigma } _{F} \lvert  f\rvert \cdot  \sigma (F)^{1/2} + \lVert H^{\sigma  } _{F} f\rVert_{\sigma  }\right\} \lVert \tilde H^{\tau  } _{F} g\rVert_{\tau  } . 
 \end{equation}
\end{lemma}

After a second application of the exchange argument, it remains to consider the \emph{stopping form}: 
\begin{equation} \label{e:XS}
B ^{\textup{stop}} _F (f,   g)
\equiv \sum_{I \::\: \pi _{\mathcal F} I= F}  \mathbb E ^{\sigma } _{I_J} \Delta ^{\sigma } _{I} f \cdot 
 \sum_{\substack{J \::\: \pi _{\mathcal F}J=F\\ J\Subset _{4r} I }}  \left\langle {\mathsf R}_{\sigma } (F-I_J),  \Delta^\tau _J g \right\rangle _{\sigma } . 
\end{equation}

\begin{lemma}\label{l:stopII} The following estimate is true: 
\begin{equation}\label{e:stopII<}
\left\vert  B ^{\textup{stop}} _F (f,   g)\right\vert 
\lesssim \mathscr R \lVert f\rVert_{ \sigma }  \lVert g\rVert_{\tau } . 
\end{equation}
\end{lemma}
This argument is a variant of the proof of Lemma~\ref{l:stopI}.

\section{Elementary Estimates} \label{s:elementary}

This section collects some considerations which, while not completely elementary, rely solely upon the 
$ A_2$ condition.  Then, reductions are proved, namely 
the Carleson cube projection in Proposition~\ref{p:carlesonCubes}, and the reduction to the `above' and `below' 
projections in Proposition~\ref{p:triangles}.  

\subsection{Weak-Boundedness}

By the weak-boundedness condition, we mean this estimate.  

\begin{proposition}\label{p:weak} 
Let $ \sigma $ and $ \tau $ satisfy the $  A_2$ condition \eqref{e:A2}.  Then, for any two 
intervals $ I, J$, intersecting only at their boundaries, we have 
\begin{equation}\label{e:weak}
\left\vert  \left\langle {\mathsf  R}  _{\sigma } f \cdot I,   g \cdot Q_J \right\rangle _{\tau }\right\vert  \lesssim \mathscr A_2 ^{1/2} 
\lVert f\rVert_{\sigma } \lVert g\rVert_{\tau}. 
\end{equation}
\end{proposition}

For the proof, we will need the classical result of Muckenhoupt \cite{MR0311856}, characterizing the two weight Hardy inequality on the line.

\begin{priorResults}\label{p:hardy} 
 For weights $ \hat{w }$ and $\hat\sigma $  supported on $ \mathbb R _+$,  
\begin{gather}
\left\Vert  \int_{(0,x)}f \;\hat\sigma(dy)  \right\Vert_{\hat  w}
\leq \mathscr B \lVert f \rVert_{\hat\sigma } \label{e:Muck}\,, 
\\ 
\textup{where} \quad 
\mathscr B^{2}\simeq \sup_{0<r }  \int_{(r, \infty )}\hat{w } (dx) \times 
\int_{(0,r)}\hat\sigma(dy)   \,. \label{e:Bis}
\end{gather}
\end{priorResults}

\begin{proof}[Proof of Proposition \ref{p:weak}]
We abandon the possibility of cancellation, taking the absolute value of the kernel of the Riesz transform. 
It is clear that we can assume that $ I $ and $ J$ share a common endpoint, 
which after a translation, can be taken to be $ 0$.  
Assume that $ I$ lies to the left of $ J$.  

The point of the next steps is to pass to dual formulations of a Hardy inequality with respect to changed measures.  
Write $ \tilde \sigma (dx) \equiv \sigma (-dx) \cdot \mathbf 1_{[0, \infty )}$, and also derive from $ \tau $ a one dimensional weight 
by setting 
\begin{equation*}
\tilde \tau (s,t) \equiv \int _{ s< \lvert  x \rvert  < t } \; d \tau \,, \qquad 0< s < t  .
\end{equation*}
Likewise for $ g \in L ^2 (Q_J;\tau)$, set $ \tilde g$ to be the non-negative function with 
\begin{equation*}
\int _{s} ^{t} \tilde g (u) \, d \tilde \tau  = \int 
 _{ s< \lvert  x \rvert  < t }  \lvert  g (x)\rvert \; d \tau \,, \qquad 0< s < t  . 
\end{equation*}
Then, by a conditional expectation argument, $ \lVert \tilde g\rVert_{\tilde \tau } \le \lVert g\rVert_{\tau }$.  
Finally, turning to the Riesz transform, we have 
\begin{align*}
\left\vert   \left\langle  {\mathsf  R}  _{\sigma }  f, g \right\rangle _{\tau }\right\vert & = \left\vert 
\int _{I} \int _{Q_J}  \frac {x-t} {\lvert  x-t\rvert ^2  }  f (t) g (x) \, \tau (dx)\, \sigma (dt) 
\right\vert 
\\
& =  \left\vert 
\int _{0} ^{\infty } \int _{Q_J}  \frac {x+t} {\lvert  x+t\rvert ^2  }  f (-t) g (x) \, \tau (dx)\, \tilde \sigma (dt) 
\right\vert 
\\
& \lesssim 
\int _{0} ^{\infty } \int _{0} ^{\infty } \frac {1 } {s+t}  \tilde f (t) g (s) \, \tilde \tau (ds) \, \tilde \sigma (dt).  
\end{align*}
Here, $ \tilde f (t) = f (-t)$.   This last form is divided into dual forms.  

The first is integration over the region $ \left\{ (s,t) \::\:  0< t < s\right\}$, on which $\frac {1 } {s+t} \le \frac 1 t$, and so it suffices to bound the form 
\begin{align*}
\int_{(0,\infty )}  \frac {\tilde f (t)}t \int_{(0,t )}   g (s) \, \tilde \tau (ds) \, \tilde \sigma (dt)
&\le 
\lVert \tilde f\rVert_{\tilde \sigma } \left[
\int_{(0,\infty )}  \left[ \int_{(0,t )}   \tilde g (s) \, \tilde \tau (ds)  \right] ^2 \frac {\tilde \sigma (dt)} { t ^2 }
\right] ^{1/2} .
\end{align*}
The last expression is the Hardy inequality with the weight $ \frac {\sigma (dt)} { t ^2 }$.  
Whence it suffices to bound the  constant $ \mathscr B ^2 $, given the expression in \eqref{e:Bis}.  It is a supremum over $ r >0$ of 
\begin{align*}
 \int_{(r, \infty )}  \frac {\tilde \sigma (dt)} { t ^2 } \times  \int_{(0,r)}  \tilde \tau (ds)
 & \le \int_{(-r,\infty )}  \frac {r} {t ^2 } \, \sigma (dt)  \times \frac 1r \int _{ Q _{[0,r]}}  \, \tau (dx)   , 
\end{align*}
as follows by inspection.  
But, the last expression is clearly dominated by our $ A_2$ assumption.  

The second expression is 
 integration over the region $ \{ (s,t) \::\:  0< s \le t\}$, on which $\frac {1 } {s+t} \le \frac 1 s$, and so it suffices to bound the form 
 \begin{align*}
\int_{(0,\infty )}  \frac {\tilde g (s)} s \int_{(0,s)}   \tilde f (t) \,  \tilde \sigma (dt)\, \tilde \tau (ds) \,
&\le 
\lVert \tilde g\rVert_{\tilde \tau  } \left[
\int_{(0,\infty )}  \left[ \int_{(0,s)}   \tilde f (t) \,  \tilde \sigma (dt)  \right] ^2 \frac {\tilde \tau  (ds)} { s ^2 }
\right] ^{1/2}. 
\end{align*}
In this case, the expression in \eqref{e:Bis} is the supremum over $ r>0$ of 
\begin{equation*}
 \int_{(r, \infty )}  \frac {\tilde \tau  (ds)} { s ^2 } \times  \int_{(0,r)}  \tilde \sigma  (dt) 
\leq 
\int _{ \lvert  x\rvert \ge r } \frac r { s ^2 } \tilde \tau  (ds) 
\times \frac 1r\int_{(-r,0)}  \sigma (dt) \lesssim \mathscr A_2. 
\end{equation*}
 The proof is complete. 
\end{proof}

\subsection{Proof of Proposition~\ref{p:carlesonCubes}}
We show that the we only need consider the martingale projection of $g\in L^2(\tau)$ formed over Carleson cubes.  
The basic facts are that  (a) the relevant martingale differences in $ L ^2 (\mathbb{R}^2_+;\tau )$ 
are associated to cubes $ Q$ whose distance to the boundary is at least the side length of $Q$, and (b)   
the $ L ^{2}$-boundedness of the Poisson operator  is a consequence of the assumptions.  

Indeed, we can assume that $ g\in L ^2 (\mathbb{R}^2_+;\tau )$ is such that $ P ^{\tau } _{\textup{Car}}g=0$. 
The upper half-plane is partitioned by cubes  of the form  $ \tilde Q_I = I \times [ \lvert  I\rvert, 2 \lvert  I\rvert)  $, 
for dyadic $ I$, and $ g = \sum_{I\in \mathcal D} \Pi _{I} ^{\tau }g$, where  
\begin{equation*}
 \Pi _{I} ^{\tau }g := \sum_{P \;:\; P\subset \tilde Q_I} \Delta ^{\tau }_P g. 
\end{equation*}
These projections are disjointly supported in $ I$.  
 
Let $ \chi _I$ be a smooth function on $ I$ with $ \mathbf 1_{I} \le \chi _I  \leq \mathbf 1_{2I}$, 
and for the derivative, $ \lvert  \chi '_I \rvert \leq 2\left\vert I\right\vert^{-1} \mathbf 1_{2 I}$.  Then,  using the separation between 
$ f$ and $ g$, it is easy to see that 
\begin{equation*}
\lvert  \langle    \mathsf R ^{\ast} _{\tau } \Pi _I ^{\tau } g , \chi _{I} f \rangle _{\sigma } \rvert 
\lesssim  \langle \lvert   \Pi _{I} ^{\tau }g\rvert,  \mathsf P _{\sigma } \lvert  f\rvert   \rangle _{\tau }. 
\end{equation*}
Summing over $ I$, we have 
\begin{equation*}
\langle \lvert  g\rvert ,   \mathsf P _{\sigma } \lvert  f\rvert   \rangle _{\tau } 
\lesssim \mathscr T  \lVert g\rVert _{L ^2  (\mathbb{R}^2_+;\tau )} \lVert f\rVert _{L ^2 (\mathbb R , \sigma )} . 
\end{equation*}
Above, we are using a two weight inequality for the Poisson operator. The latter, by Sawyer's Theorem \cite{MR930072}, is equivalent to the two-weight testing inequalities. 
But, the Poisson operator is one coordinate of the Cauchy transform, so that the testing inequalities for the Cauchy transform imply the two weight inequality for the Poisson operator.

Using the mean zero property of $ \Pi _{I} ^{\tau }g$, we can also verify that 
\begin{equation*}
\lvert  \langle    \mathsf R ^{\ast} _{\tau } \Pi _I ^{\tau } g , (1- \chi _{I}) { f} \rangle _{\sigma } \rvert 
\lesssim  \langle \lvert   \Pi ^{\tau }_{I} g\rvert,  \mathsf P ^{\sigma } \lvert  f\rvert   \rangle _{\tau }. 
\end{equation*}
And, then the sum over $ I$ is controlled just as before.

\subsection{Proof of Proposition~\ref{p:triangles}}

In the proof of Proposition~\ref{p:triangles}, it suffices to consider $ g$ in the span of (good) martingale 
differences $ \Delta ^{\tau } _{Q_J}$, namely associated with Carleson cubes.   There are quite a few 
subcases of the proof, all controlled by goodness, and the $ A_2$ condition.

Define a sub-bilinear form, and several collections of pairs of intervals as follows: 
\begin{align*}
B ^{\mathcal P} (f,g) & \equiv \sum_{ (I, J) \in \mathcal P} 
\left\vert \left\langle  {\mathsf  R} ^{\ast } _{\tau } \Delta _{Q_J} ^{\tau } g , \Delta ^{\sigma } _{I} f \right\rangle _{\sigma } \right\vert, 
\\
\mathcal P _{\textup{diagonal}} & \equiv \left\{ (I,J) \::\: 3I \cap 3J \neq \emptyset,   2 ^{-4r} \lvert  I\rvert < \lvert  J\rvert \le   \lvert  I\rvert    \right\}, 
\\
\mathcal P _{\textup{far}} & \equiv \left\{ (I,J) \::\: 3I \cap 3J= \emptyset  \right\}, 
\\
\mathcal P _{\textup{near}} & \equiv \left\{ (I,J) \::\:  J\subset 3I \setminus I\, \textup{ or } I\subset 3J \setminus J \right \}.  
\end{align*}

\begin{lemma}\label{l:star} For $ \star \in \{\textup{diagonal}, \textup{far}, \textup{near}\}$, the following holds 
\begin{equation*}
B ^{\mathcal P _{\star}} (f,g) \lesssim \mathscr R \lVert f\rVert_{\sigma } \lVert g\rVert_{\tau }.  
\end{equation*}
\end{lemma}
The details of these cases are straightforward, and have appeared in on these cases have appeared in \cites{10031596,10014043,primer}.
Briefly, the `diagonal' case depends upon the  weak boundedness principle, Proposition \ref{p:weak}.  
The `far' and `near' cases 
appeal to cancellation from the `small' martingale difference, 
and the $ A_2$ condition. 
An important element of these arguments is that goodness of the small interval $ I$ implies that 
it is necessarily relatively far from the  boundary of interval $ J$.  In this case, note that $ I$ will also be relatively far   
from   boundary of    $ Q_J$, and vice versa.  
These properties follow from 
 the randomization of the dyadic grids \emph{in only the  horizontal direction}; note that the randomization in the vertical direction is not needed because of the positivity we have in the second Riesz transform.

The next two cases concern the reduction of the `large' martingale difference to the child that contains the `small' martingale difference. 
There are two different estimates which are as follows.  

\begin{lemma}\label{l:restrict}  Both of these bilinear forms are bounded by $ \mathscr R \lVert f\rVert_{\sigma } \lVert g\rVert_{\tau }$, 
\begin{gather*}
 \sum_{ (I,J) \::\: I\Subset J} 
\left\vert \left\langle  {\mathsf  R} ^{\ast } _{\tau } ( \Delta _{Q_{J} }^{\tau } g \cdot  \mathbf 1_{Q_J \setminus Q _{JI}}) , \Delta ^{\sigma } _{I} f \right\rangle _{\sigma } \right\vert, 
\\
 \sum_{ (I,J) \::\: J\Subset I} 
\left\vert \left\langle  {\mathsf  R}  _{\sigma} (\Delta ^{\sigma } _{I} f  \cdot \mathbf 1_{I \setminus I_J}) , \Delta _{Q_J} ^{\tau } g  \right\rangle _{\tau  } \right\vert.  
\end{gather*}  
Recall that $ I_J$ is the child of $ I$ that contains $ J$, and $ Q _{JI} $ is the dyadic cube, child of $ Q_J$ that contains $ Q_I$.  
{In both estimates, we have the full martingale difference on the small interval.} 
\end{lemma}

The proof of this estimate is very similar to that of $ \mathcal P _{\textup{near}}$, and we again omit the details.



\end{document}